\documentclass[smallcondensed]{svjour3}

\usepackage[centertags]{amsmath}
\usepackage{amssymb}
\usepackage{mathptmx}
\usepackage[colorlinks=true,linkcolor=blue,citecolor=blue]{hyperref}
\usepackage{url}
\usepackage[T1]{fontenc}

\usepackage{tikz}
\usetikzlibrary{calc}

\smartqed

\title{Dynamics of the Douglas-Rachford Method for Ellipses and $p$-Spheres\thanks{This is a pre-print of an article published in Set-Valued and Variational Analysis. The final authenticated version is available online at: \url{https://doi.org/10.1007/s11228-017-0457-0}}}
\dedication{This work is dedicated to the memory of Jonathan M. Borwein our greatly missed friend, mentor, and colleague. His early contributions to this paper, like his contribution to so much else, was both vital and inspirational.}

\author{Jonathan M. Borwein \and Scott B. Lindstrom \and Brailey Sims \and Anna Schneider \and Matthew P. Skerritt}

\institute{
	Jonathan M. Borwein \and Scott B. Lindstrom \and Brailey Sims \and Matthew P. Skerritt \at \textsc{carma}, University of Newcastle, Australia \\
	\email{brailey.sims@newcastle.edu.au} \\
	\textsc{orcid}: \href{https://orcid.org/0000-0003-2162-9097}{0000-0003-2162-9097} \\
	\email{scott.lindstrom@uon.edu.au} \\
	\textsc{orcid}: \href{https://orcid.org/0000-0003-4287-4788}{0000-0003-4287-4788} \\
	\email{matthew.skerritt@uon.edu.au} \\
	\textsc{orcid}: \href{https://orcid.org/0000-0003-2211-7616}{0000-0003-2211-7616}
	\and 
	Anna Schneider \at Universit\"{a}t der Bundeswehr, M\"{u}nchen
}

\date{\today}

\def\R{\hbox{$\mathbb R$}}

\begin{document}
	
	\maketitle
	
	\begin{abstract}
		We expand upon previous work that examined behavior of the iterated Douglas-Rachford method for a line and a circle by considering two generalizations: that of a line and an ellipse and that of a line together with a $p$-sphere. With computer assistance, we discover a beautiful geometry that illustrates phenomena which may affect the behavior of the iterates by slowing or inhibiting convergence for feasible cases. We prove local convergence near feasible points, and---seeking a better understanding of the behavior---we employ parallelization in order to study behavior graphically. Motivated by the computer-assisted discoveries, we prove a result about behavior of the method in infeasible cases.
	\end{abstract}
	
	\section{Introduction and Preliminaries}\label{sec:intro}
	
	The Douglas-Rachford algorithm \cite{DR} was introduced over half a century ago in connection with nonlinear heat flow problems to find a feasible point (point in the intersection) of two closed constraint sets $A$ and $B$ in a Hilbert space $H$. 
	
	We will denote the induced norm by $\|\cdot\|$. The projection onto a proximal subset $C$ of $H$ is defined for all $x \in H$ by
	\[P_C(x) := \left \{ z \in C : \|x - z\| = \inf_{z' \in C}\|x - z'\|\right \}\]
	When $C$ is closed and convex the projection operator $P_C$ is single valued and firmly nonexpansive. When $C$ is a closed subspace it is also linear and self-adjoint. For additional information, see, for example, \cite[Definition 3.7]{BC}. The reflection mapping through the set $C$ is then defined by
	\[R_C := 2P_C - I,\]
	where $I$ is the identity map on $H$.
	
	\begin{definition}[Douglas-Rachford Method]\label{def:DR}
		For two closed sets $A$ and $B$, and an initial point $x_0 \in H$, the Douglas-Rachford method generates a sequence $(x_n)_{n=1}^\infty$ as follows:
		\begin{equation}
		x_{n+1} \in T_{A,B}(x_n) \quad \text{where} \quad T_{A,B} := \frac{1}{2}\left( I + R_{B}R_{A}\right).
		\end{equation}	
	\end{definition}
	Figure~\ref{fig:oneiterate} illustrates the construction of one iteration of the Douglas-Rachford method.
	
	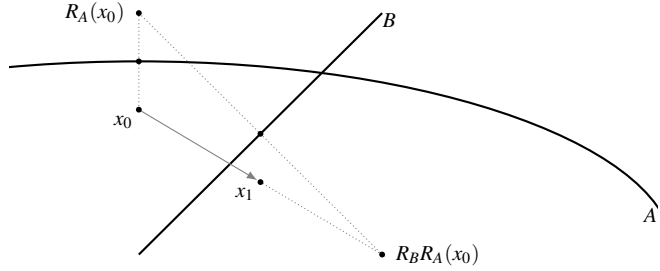
\begin{figure}
		\begin{center}
			\begin{tikzpicture}[x=0.7\textwidth,y=0.7\textwidth] 
			\clip (-0.2,-0.025) rectangle (0.8,0.4);
			\coordinate (x0) at (0, 0.225);
			\coordinate (PAx0) at (0, 0.3); 
			\coordinate (RAx0) at (0, 0.375);
			\coordinate (RBRAx0) at (0.375,0);
			\coordinate (PBRAx0) at (0.1875,0.1875); 
			\coordinate (x1) at (0.1875, 0.1125); 
			
			\draw[thick] (0,0) ellipse (0.825 and 0.3);
			\draw[thick] (0,0) -- (0.375,0.375);
			\draw[gray,-latex] (x0) -- ($(x0)!0.975!(x1)$);
			\draw[gray,densely dotted] (x0) -- (RAx0) -- (RBRAx0) -- (x1);
			\foreach \pt in {(x0), (PAx0), (RAx0), (PBRAx0), (RBRAx0), (x1) } {
				\fill \pt circle (0.0045);
			}
			\node[inner sep=2pt, below left] at (x0) {{\(x_0\)}};
			\node[inner sep=5pt, left] at (RAx0) {\(R_A\!\left(x_0\right)\)};
			\node[inner sep=5pt, right] at (RBRAx0) {\(R_BR_A\!\left(x_0\right)\)};
			\node[inner sep=2pt, below left] at (x1) {\(x_1\)};
            
			\node[inner sep=0pt, below right] at (0.375,0.375) {\(B\)};
            \node[inner sep=0pt, below left] at (0.8,0.073) {\(A\)};
			
			\end{tikzpicture}
			\caption{One iteration of the Douglas-Rachford method}
			\label{fig:oneiterate}
		\end{center}
	\end{figure}
	
	\begin{remark}[Notation]\label{remark:standards}
		Throughout, $x_n,x_0$ are as in Definition~\ref{def:DR}, $A,B$ are closed. When the two sets $A$ and $B$ are clear from the context we will simply write $T$ in place of $T_{A,B}$.
	\end{remark}
	
	\begin{theorem} [Bauschke, Combettes, and Luke \cite{BCL}]
		Suppose $A,B \subseteq H$ are closed and convex with non-empty intersection. Given $x_0 \in H$ the sequence of iterates $T_{A,B}$ converges weakly to an $x \in \operatorname{Fix} T_{A,B}$ with $P_{A}(x) \in A \cap B$.
	\end{theorem}
	In finite dimensions convergence in norm for convex sets is therefore assured. Notwithstanding the absence of a satisfactory theoretical justification, the Douglas-Rachford iteration scheme has been used to successfully solve a wide variety of practical problems in which one or both of the constraints are non-convex. Phase retrieval problems are one important instance, and the case of a line $L$ and circle $C$ in 2-dimensional Euclidean space---prototypical of such problems---was investigated by Borwein and Sims \cite{BS} as a specific case of the higher dimensional problem of a line and a sphere in Hilbert space.
	
	Despite the seeming simplicity of the situation, the Douglas-Rachford method applied to $L$ and $C$ proved surprisingly difficult to analyze. Among the partial results obtained in the feasible case was local convergence to each of the two feasible points. Based on this and extensive computer experimentation, Borwein and Sims were led to ask whether this could be extended to convergence to one or other of the two intersection points for all starting points except those lying on a \lq\lq singular set" $S_{0}$; the line of symmetry perpendicular to $L$ and passing through the centre of $C$. Borwein and Arag\'{o}n Artacho \cite{AB} established sizable domains of attraction for each of the feasible points, and the global question was answered in the affirmative by Benoist \cite{Benoist} who obtained the result by constructing a suitable Lyapunov function, see figure  \ref{fig:2sphere-lyapunov}.
	
	The singular set $S_{0}$ is invariant under the Douglas-Rachford operator $T_{C,L}$ and contains period 2 points if and only if $L$ passes through the centre of $C$ in which case all the points of $L$ inside $C$ are period 2 points. When $L$ is tangential to $C$ all points on $S_{0}$ are fixed by $T_{C,L}$, for other positions of $L$ the iterates exhibits periodic behaviors when rational commensurability is present, while in the absence of such commensurability the behaviors may be quite chaotic. See \cite{BS} for more details.
	
	\begin{figure}
		\begin{center}
			\includegraphics[width=0.7\textwidth]{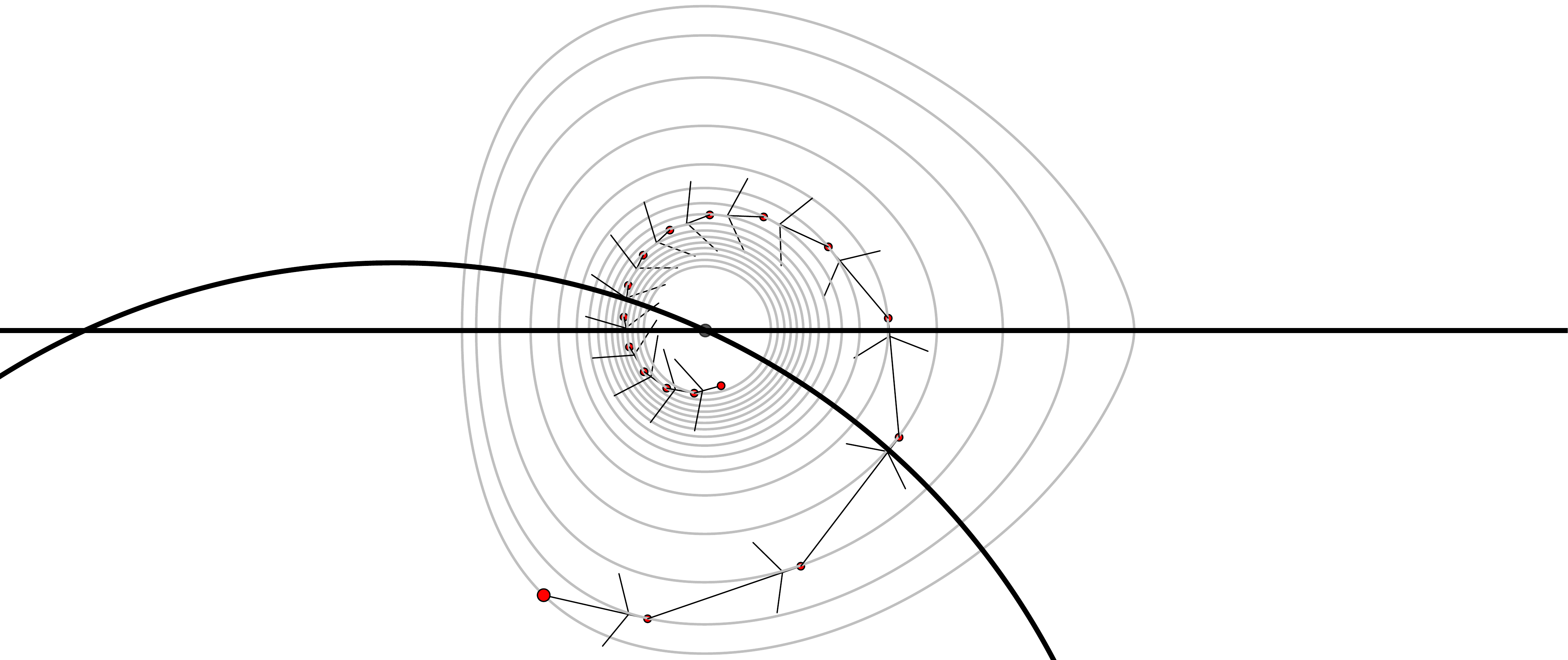}
			\caption{Douglas-Rachford on the 2-sphere and line showing the level sets of the Lyapunov function \cite{Benoist}}
			\label{fig:2sphere-lyapunov}
		\end{center}
	\end{figure}
	
	In order to gain further insights into the behavior of the Douglas-Rachford algorithm in the case of nonconvex constraint sets we consider two generalizations of a line  and  sphere (circle) in 2 dimensional Euclidean space, namely:  that of a line together with an ellipse and that of a line together with a $p$-sphere.
	
	These seemingly innocuous generalizations, while open to exploration and local analysis about the feasible points, may be impossible to analyze in full. The singular set is no longer a simple curve but rather exhibits a complex (and fascinating) geometry involving a rich array of periodic points and associated domains of attraction. In the case of an ellipse and line we observe the appearance of higher order periodic points as the ellipticity is increased.
	
	\begin{definition}[Periodic points and domains]
		The following terms, already used informally, help inform our discussion.
		\begin{enumerate}
			\item A point $x$ is a \emph{periodic point of period $m$} (or a \emph{period $m$ point}) if $T_{A,B}^{m}x = x$ (A period $1$ point is simply a fixed point of $T_{A,B}$).
			\item The \emph{domain of attraction} (or \emph{attractive domain})for a period $m$ point $x$ is the set of all $x_0$ satisfying 
			\begin{equation}\label{attractive}
				\underset{k\rightarrow \infty}{\lim}T^{km}_{A,B}(x_0) = x.
			\end{equation}
			\item A point $x$ is \emph{attractive} if its domain of attraction contains a neighborhood of $x$.
			\item The \emph{singular set} consists of all points not belonging to a domain of attraction for any feasible point.
			\item A period $m$ point is said to be \emph{repelling} if there exists a neighborhoods $N_{x}$ of $x$ such that for every $x_{0} \in N_{x}\backslash \{x\}$ the sequence $(T^{km}_{A,B}(x_{0}))_{k=1}^{\infty}$ eventually lies outside of $N_{x}$.
		\end{enumerate}
	\end{definition}
	Of course if $S$ is a domain of attraction for a period $m$ point $x$ then for $k = 1,2,\cdots m-1$ it follows that $T^k(S)$ is a domain of attraction for $T^k(x)$. This is a notable feature in many of our graphics, see for instance Figure~\ref{fig:basinsdefinition}.
	
	\begin{figure}
		\begin{center}
			\includegraphics[width=0.7\textwidth]{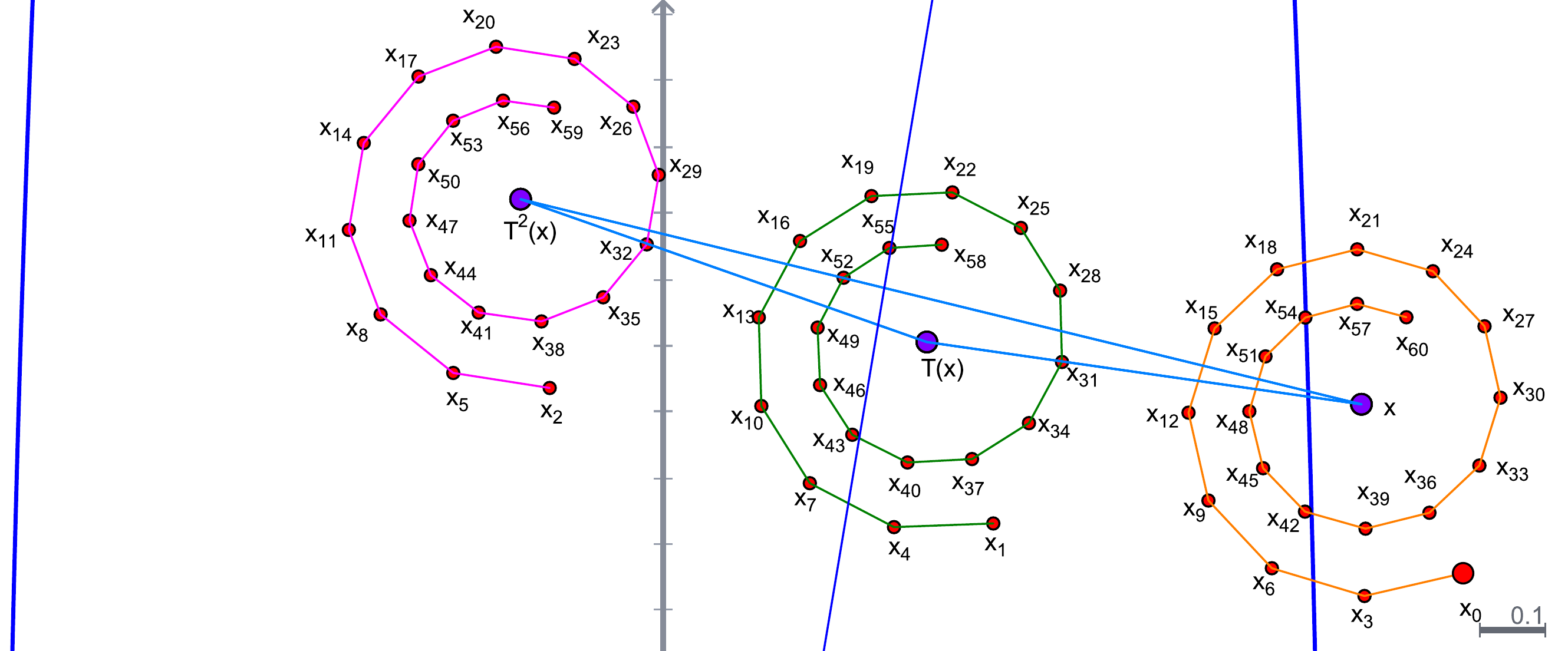}
			\caption{Domain of attraction for period 3 points in the case of $E_{6}$, $L_{8}$}
			\label{fig:basinsdefinition}
		\end{center}
	\end{figure}
	
	\subsection{Notation}
	By a suitable rotation and scaling of axes we may without loss of generality take our ellipse and  $p$-sphere respectively to be
	\begin{align}\label{notations}
	E_{b}\ &:=\ \Big\{ (x,y)\in \mathbb{R}^2 | \; \varphi_b(x,y):=x^2+\left(\frac{y}{b}\right)^2 =1\Big\} \quad 	\textrm{and}\\
	S_{p}\ &:= \Big\{ (x,y)\in \mathbb{R}^2 | \; \theta_p(x,y):=\left(x\right)^p+\left(y\right)^p = 1 \Big\},\nonumber
	\end{align}
	and will write:
	\begin{align*}
	L_{m,\beta}\ :=\ \{(x,y)\in \R^{2} | y=mx+\beta\}\quad \textrm{and} \quad L_m := L_{m,0}.
	\end{align*}
	When it is clear from the context what the parameters are we will simply write $E$, $S$ or $L$ respectively. Similarly, when the context makes it clear we will write $T$ in place of   $T_{E,L}$ or $T_{S,L}$.
	
	\subsection{Computation of projections}
	
	For the case of the $2$-sphere, the closest point projection has a simple closed form. For $x \ne 0$, $P_S(x)=x/\|x\|$. Such a simple closed form is immediately lost for any ellipse with $b\ne 1$ or any $p$-Sphere with $p\notin\{1,2\}$ because---where $\varphi_b,\theta_p$ are as in \eqref{notations}---the induced Lagrangian problems
	\begin{align*}
	P_{E_b}(x)&= \Big\{x'\;\; |\;\; \lambda \nabla d(x,\cdot)^2 (x') =\nabla \varphi_b(x'),\quad \varphi_b(x')=1\Big\} \\
	P_{S_p}(x)&= \Big\{x' \;\;|\;\; \lambda \nabla d(x,\cdot)^2 (x') =\nabla \theta_p(x'),\quad \theta_p(x')=1\Big\}
	\end{align*}
	yield implicit relations that no longer admit explicit solutions. To compute the required projections necessitates the use of numerical methods.
	
	A description of the optimized function solvers used is available in the appendix \cite{APPENDIX}. Many of our implementations of these function solvers---for example, that used to generate Figure~\ref{fig:basinsdefinition}---employ the interactive geometry software \emph{Cinderella}, available at \url{https://cinderella.de/}.
	
	\section{The case of an ellipse and a line}
	
	In the case of an ellipse and a line, the singular set---in contrast to the case of a circle and a line---is no longer a simple curve,  and appears to contain periodic points in many cases. For example, Figure~\ref{fig:BasinsAll} shows periodic points for $T_{E_8,L_6}$ with attendant subsets of their attractive domains. The singular set is larger than suggested here (see Figure~\ref{fig:FullEllipseBasinsFullPage} for a more complete depiction).
	
	\begin{figure}
		\begin{center}
			\includegraphics[angle=90,width=0.7\textwidth]{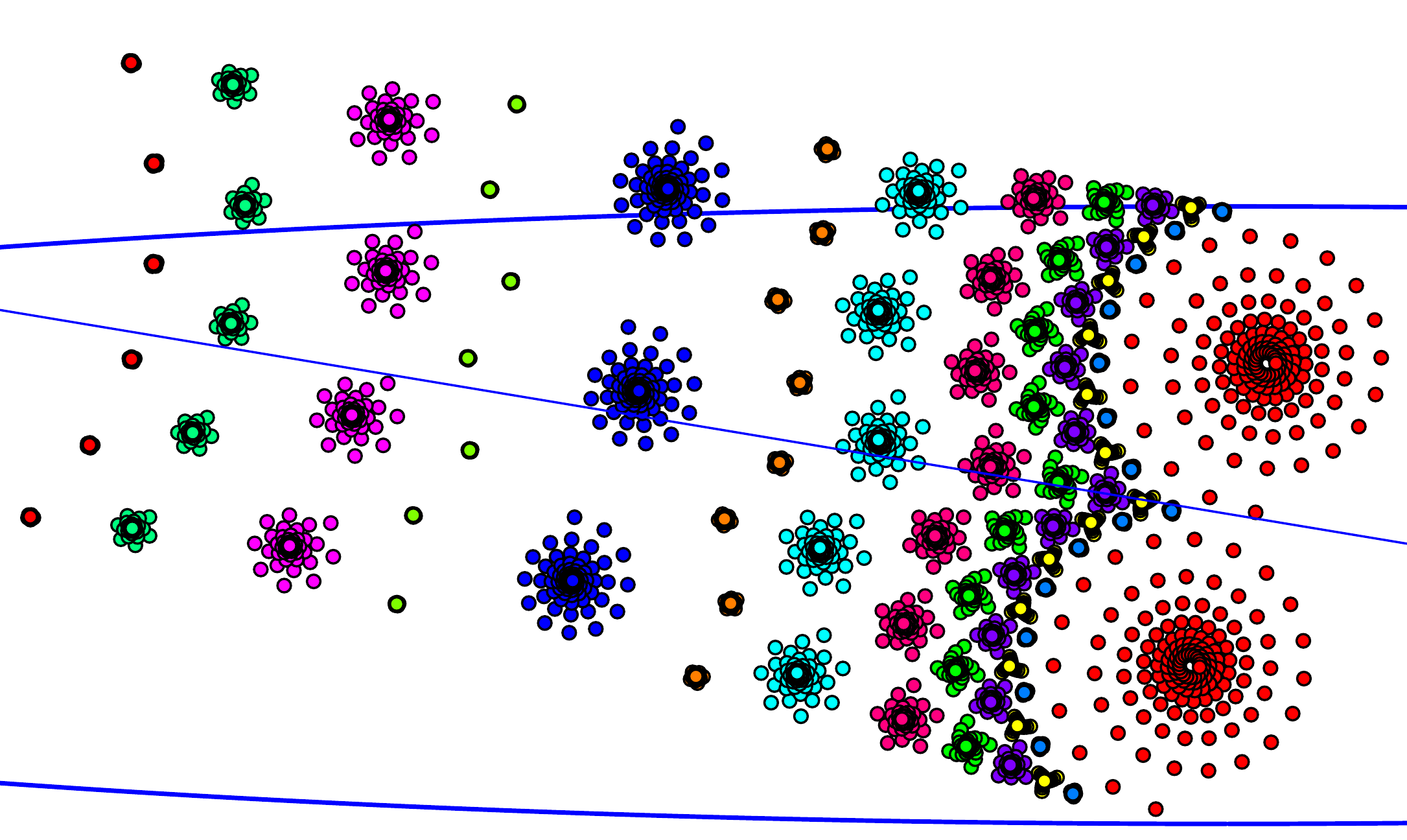}
			\caption{Partial domains of attraction of $T_{E_8,L_6}$ for points of different periodicities}
			\label{fig:BasinsAll}
		\end{center}
	\end{figure}
	
	For simplicity, we set the line intercept at $0$ for our pictures and tables. It should be noted that, in contra distinction to the case of a circle and line, similar behavior can be observed with nonzero intercepts (although symmetries are lost).
	
	\begin{table}
		\begin{center}
			\caption{Periods observed for attractive domains for various ellipse and line configurations}
			\label{tbl:periodicitytable1}
			\begin{tabular}{l | l l l l l}
				& $E_2$& $E_3$& $E_4$& $E_5$ & $E_6$  \\ \hline
				$L_1$& 2 & 2    & 2      & 2& 2\\
				$L_2$&  & 2,3 & 2,3    & 2,3& 2,3\\
				$L_3$&  & 2,3 & 2,3,5  & 2,3,4,5,7& 2,3,4,5,7\\
				$L_4$&  &   & 2,3,5,7& 2,3,4,5,7,9& 2,3,4,5{\scriptsize($\times2$)},7,9\\
				$L_5$&  &   & 2,3    & 2,3,4,5,7,9,11,13& 2,3,4,5,7,9,11,13\\
				$L_6$&  &   & & 2,3,5,7& 2,4,5{\scriptsize($\times 2$)},7,9,11,13,15\\
				$L_7$& & & & 3 &2,3,4,5,7,9,11,13\\
				$L_8$& & & & &2,3,5 \\ \hline
				\multicolumn{6}{l}{Note: some periodicity's were observed in more than one domain.}
			\end{tabular}
		\end{center}
	\end{table}

	The number and periodicity of the points appears to be related to both the eccentricity of the ellipse, and the angle of the line. As the eccentricity is increased, we observe growth in both the number of periodic points and the maximum periodicity. Table~\ref{tbl:periodicitytable1}, obtained experimentally using  \emph{Cinderella}, summarizes our findings. Note that the method used required interactively moving a point in the geometry package, and visually observing the attractive domains. As such it is regrettably possible that some periodic points were missed, either because their domains of attraction were too small or they were not attractive points.
	
	We can describe the period 2 points of $T_{E_b,L_m}$ with a closed form that, while complicated to state, is quick to evaluate \cite{APPENDIX}. Determining period 2 points algebraically is useful for corroborating some of the behaviors we observe in \emph{Cinderella}. However, the degree of complication associated with the analysis of even this simplest case of a non-fixed periodic point suggests that fully describing all behavior globally with explicit forms would be an impractical undertaking. This, in part, led us to pursue the computer assisted evidence-gathering approach we describe in subsections~\ref{subsec:Numerical}.
	
	The nature of the periodic points is also sensitive to small perturbations of the line. We can see above that for lines of small slope there are only a few attractive periodic points, and as the slope increases additional points with higher periodicity emerge. As the slope becomes large, some of the attractive domains appear to shrink in size until eventually the associated periodic point ceases to be attractive. This appears to be the eventual fate of all periodic points.
	
	\begin{figure}
		\begin{center}
			\includegraphics[width=0.16\textwidth]{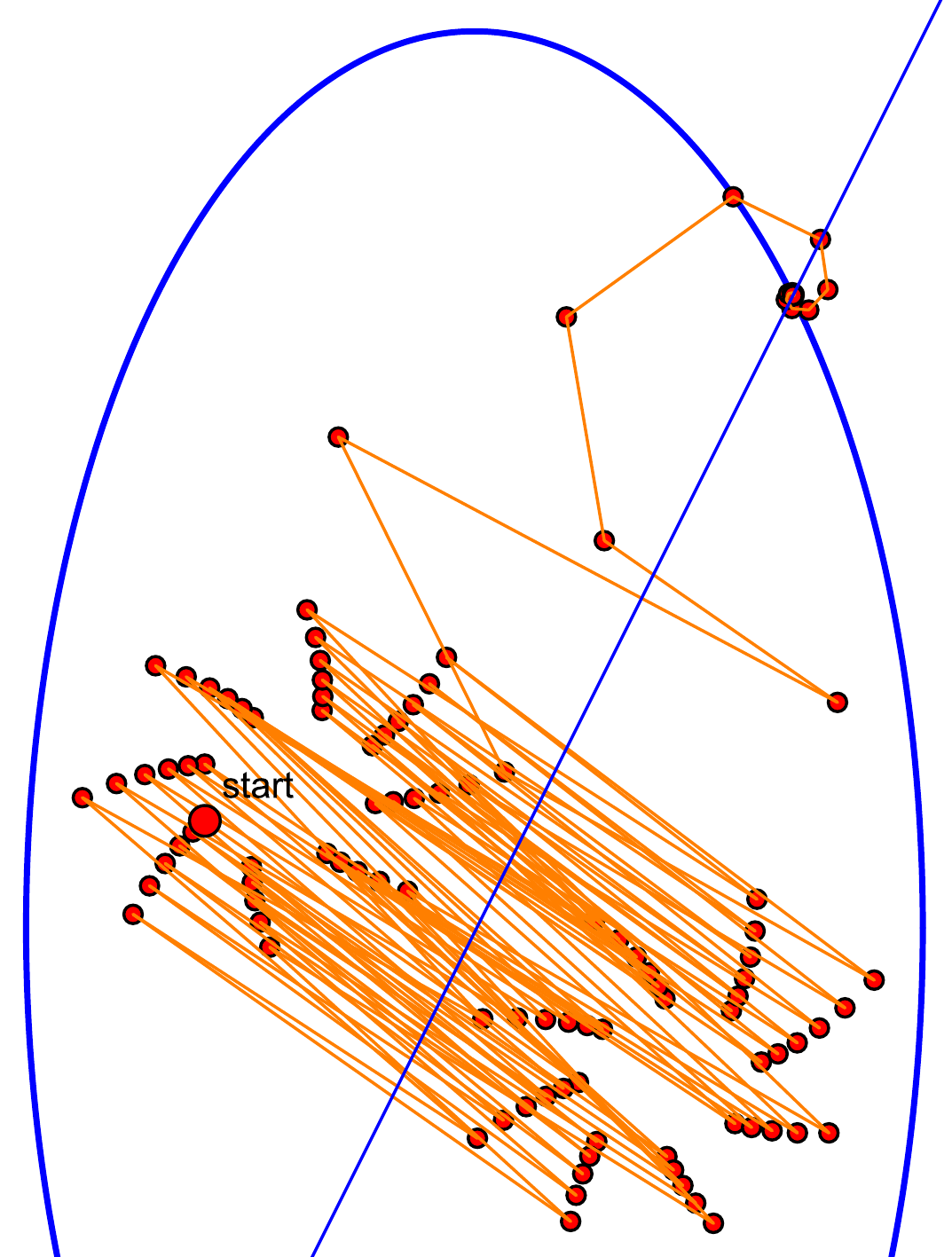} \hspace{0.02\textwidth}
			\includegraphics[width=0.16\textwidth]{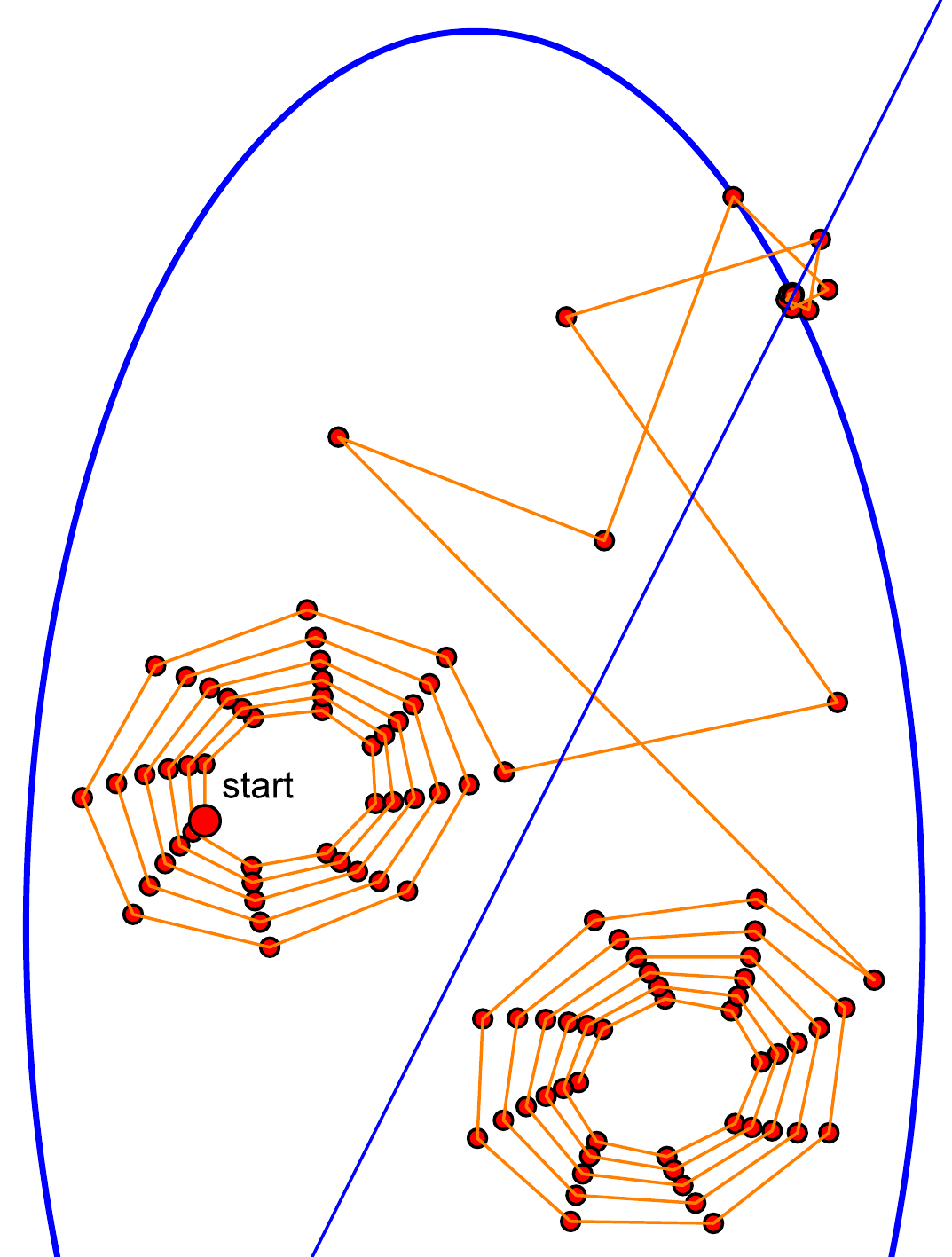} \hspace{0.02\textwidth}
			\includegraphics[width=0.16\textwidth]{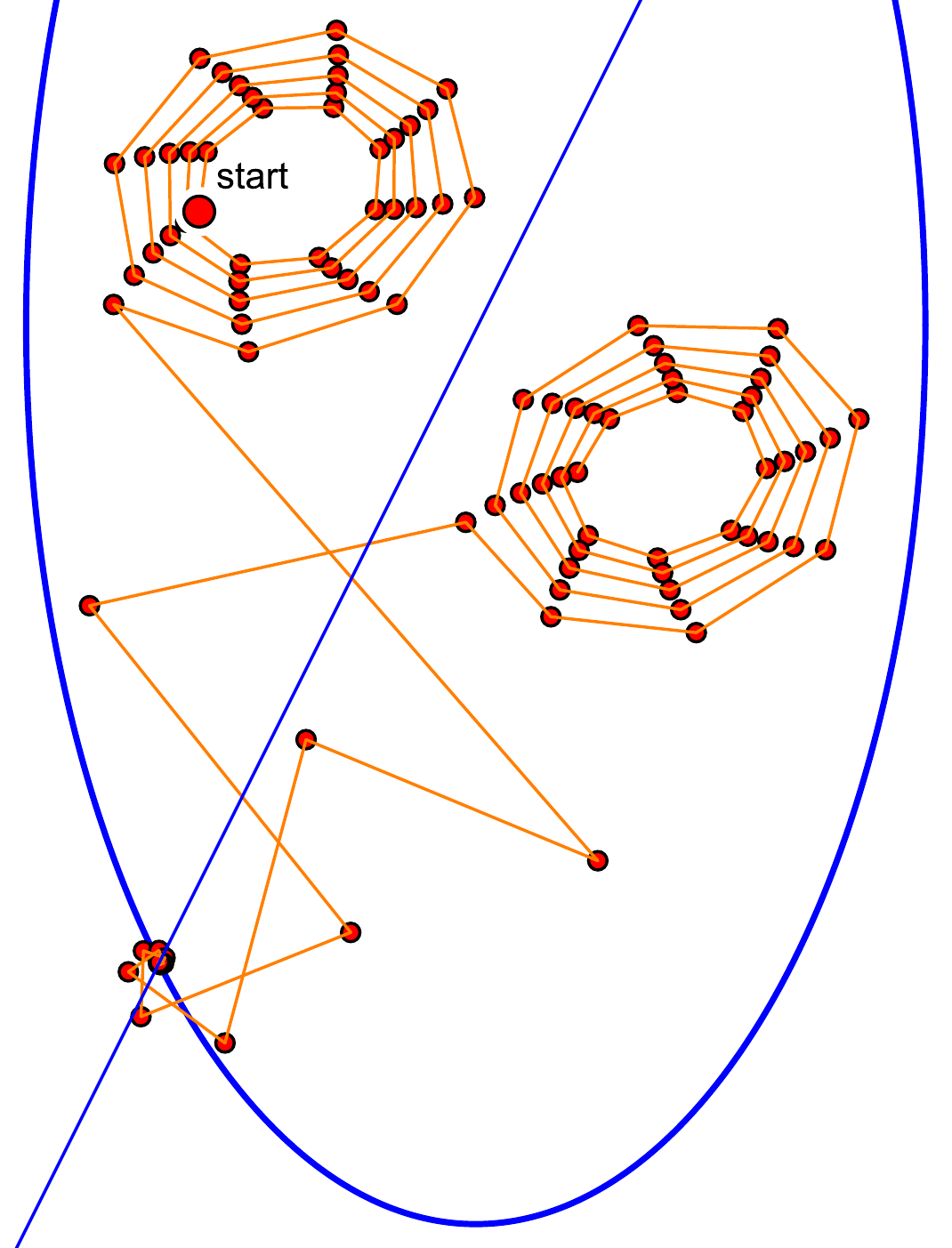} \hspace{0.02\textwidth}
			\includegraphics[width=0.16\textwidth]{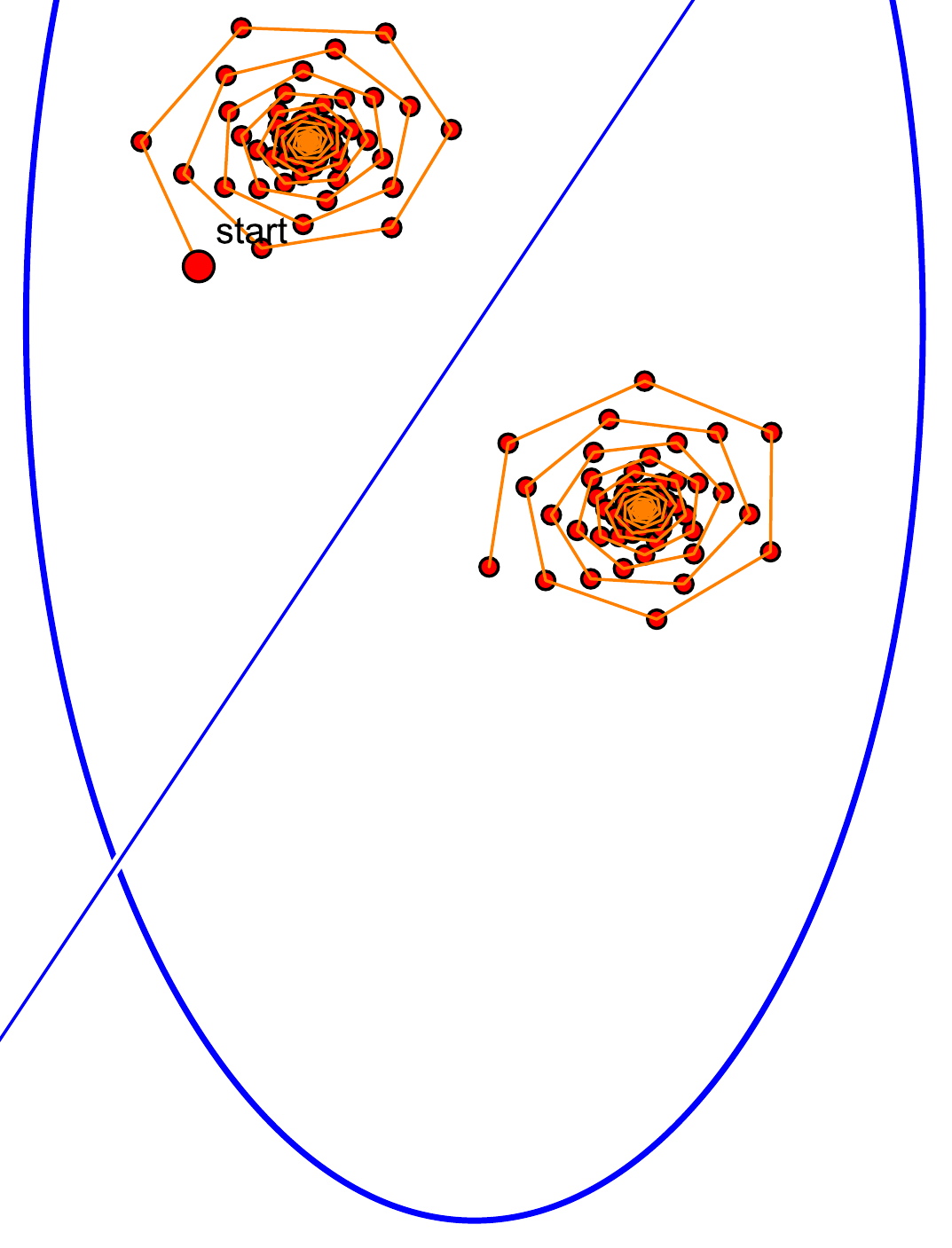}
			\caption{Sensitivity of behaviors to small changes in line slope}
			\label{fig:SourcesandSinks}
		\end{center}
	\end{figure}
	
	\begin{figure}
		\begin{center}
			\setlength\tabcolsep{0.01\textwidth} 
			\begin{tabular}{cccc}
				\includegraphics[width=0.16\textwidth]{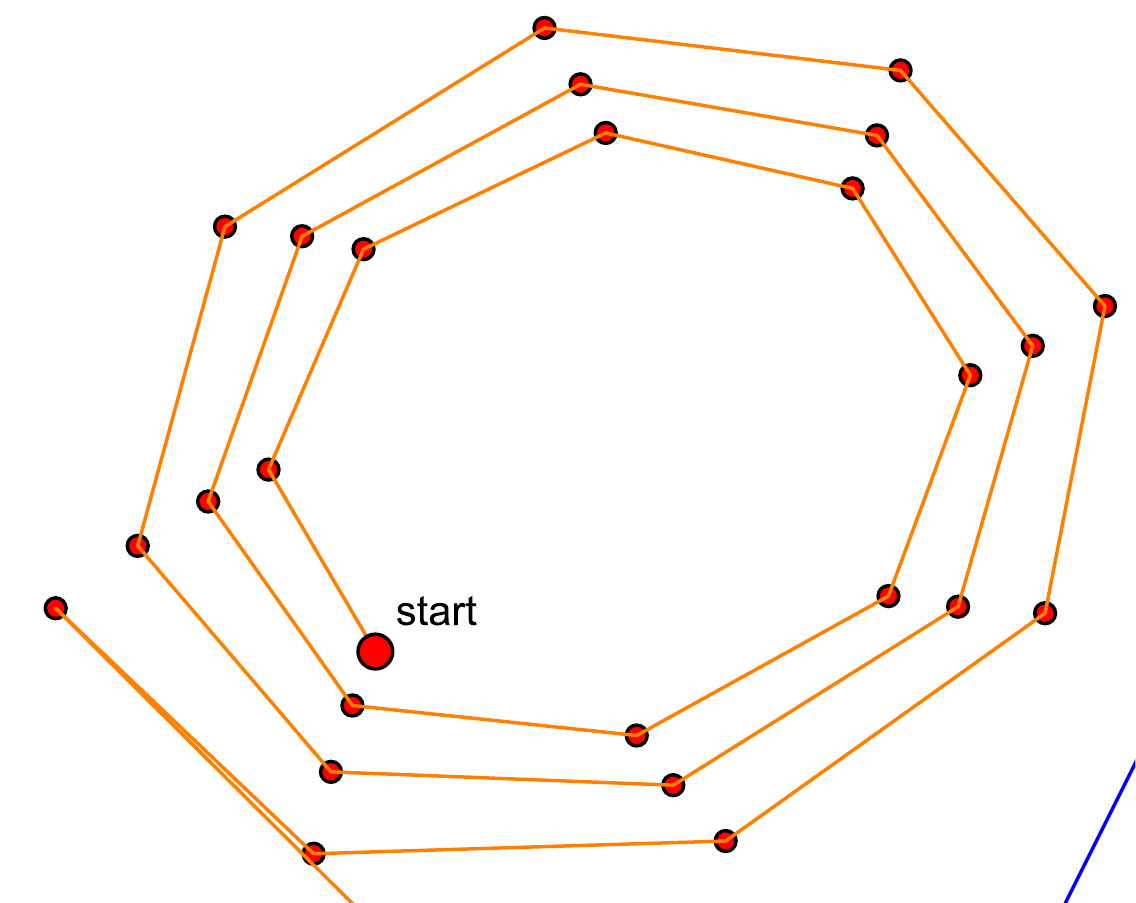} &
				\includegraphics[width=0.16\textwidth]{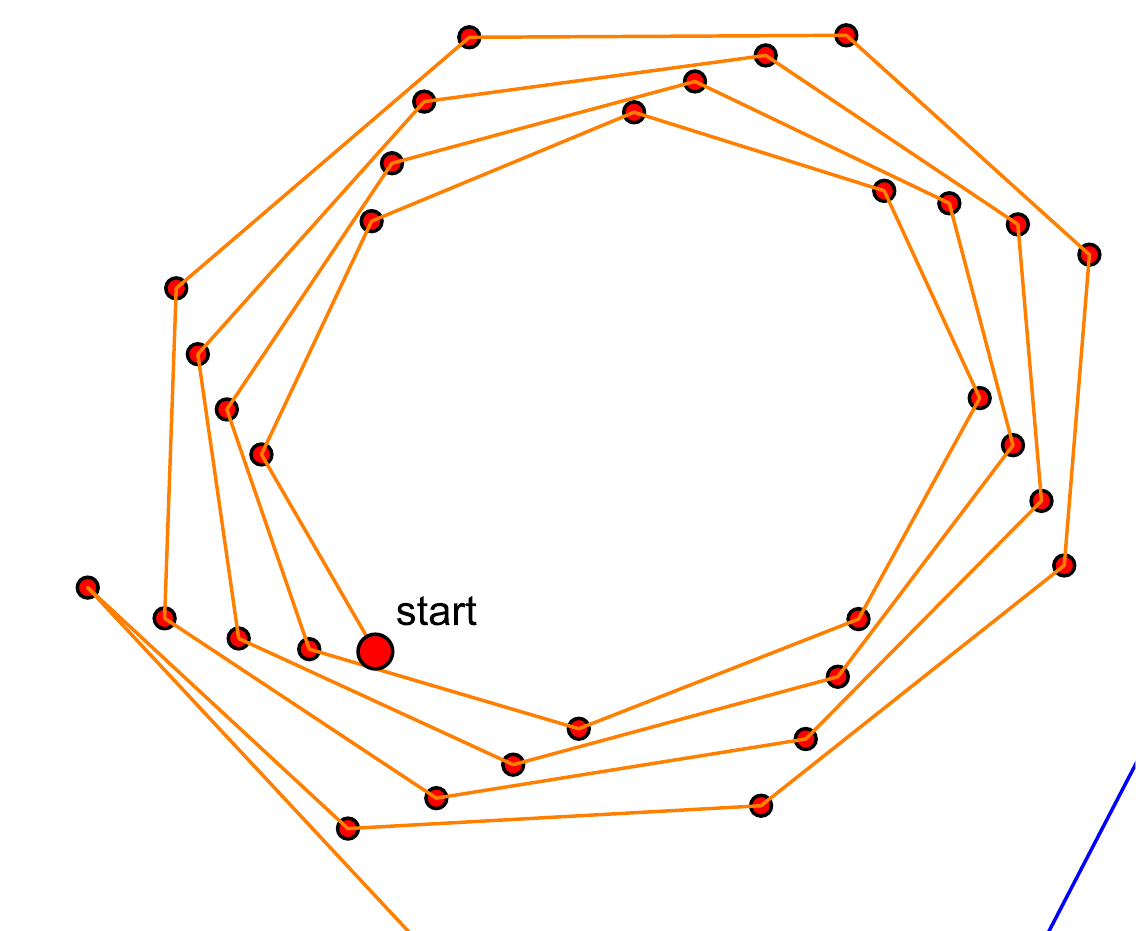} &
				\includegraphics[width=0.16\textwidth]{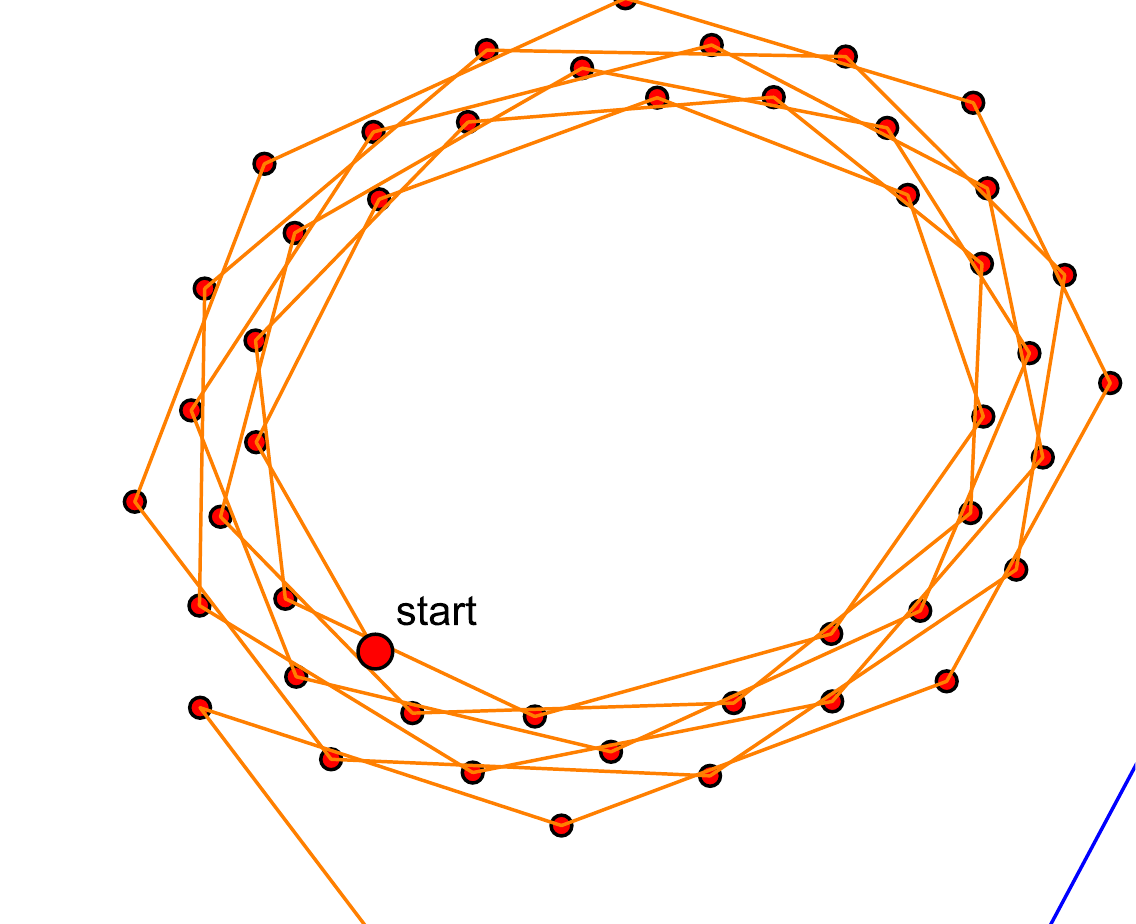} &
				\includegraphics[width=0.16\textwidth]{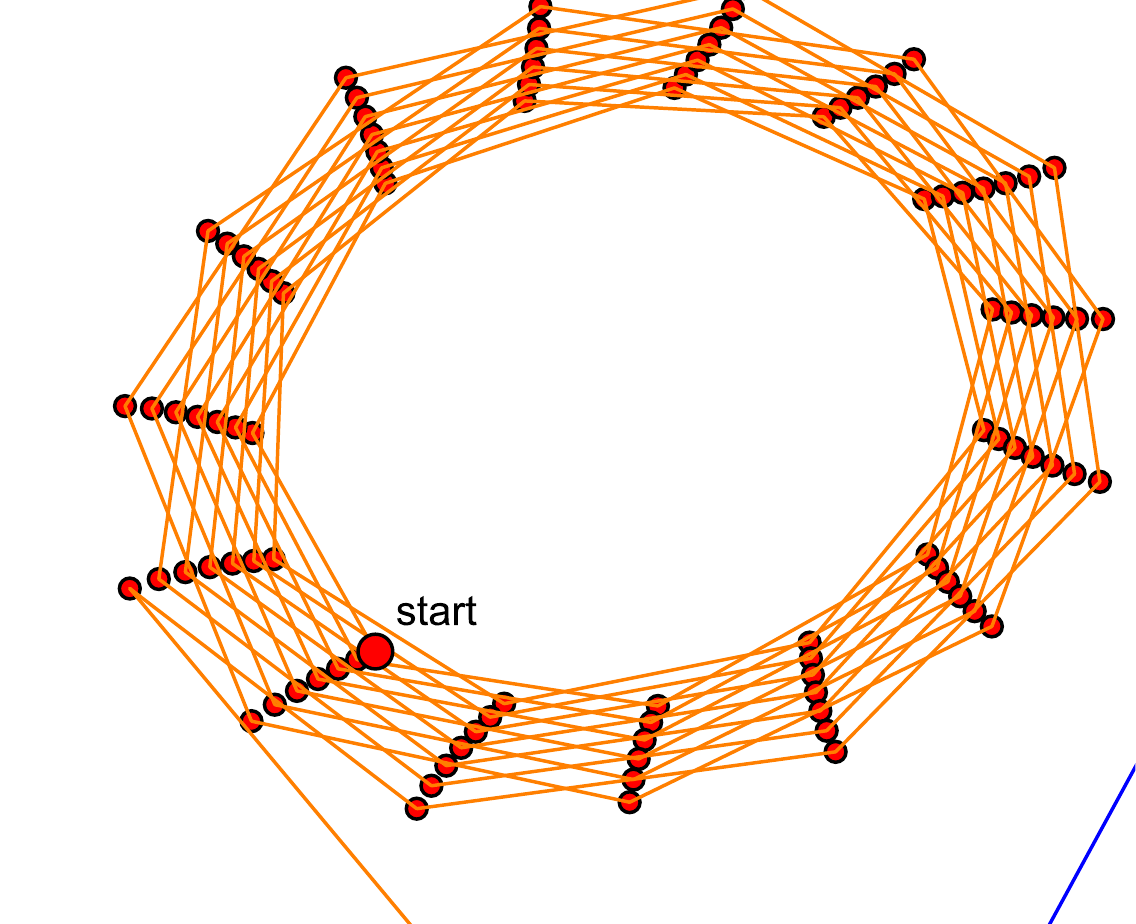} \\
				\includegraphics[width=0.16\textwidth]{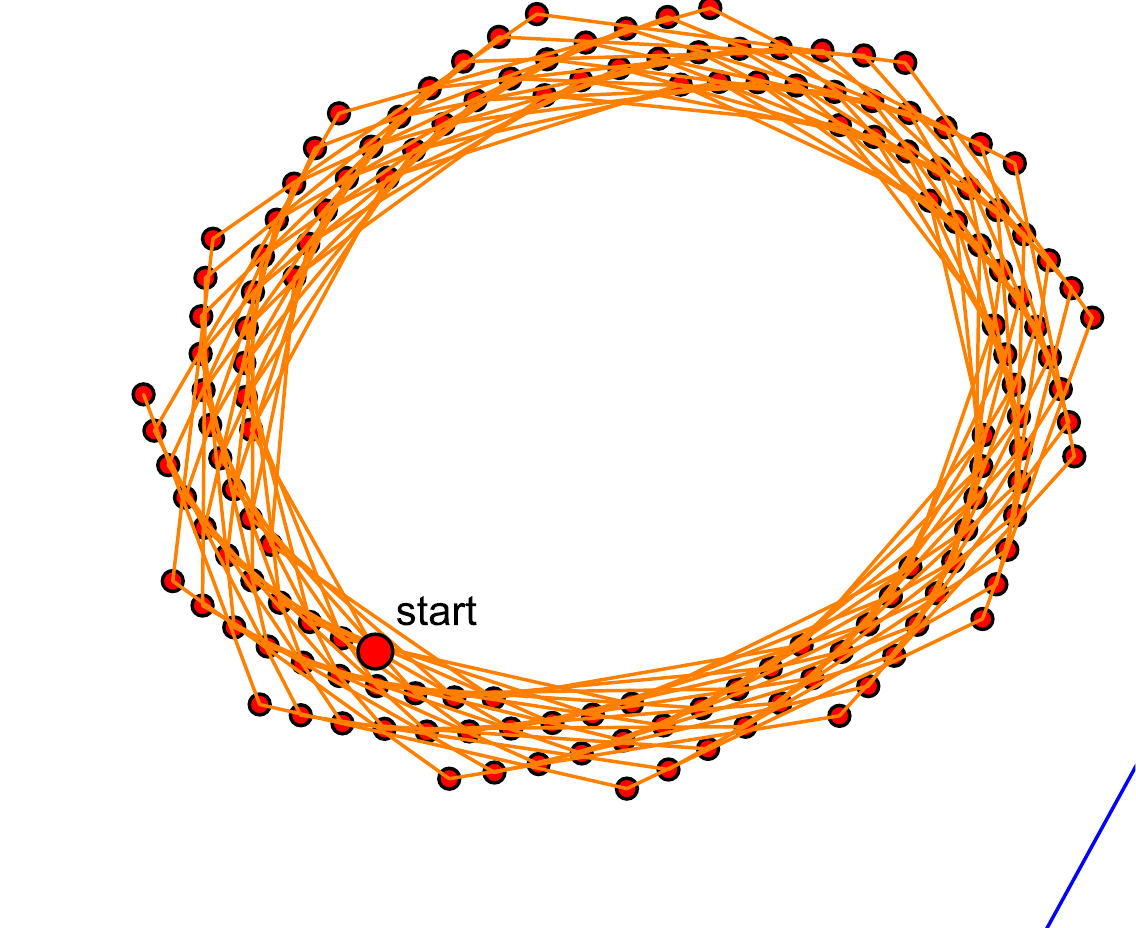} &
				\includegraphics[width=0.16\textwidth]{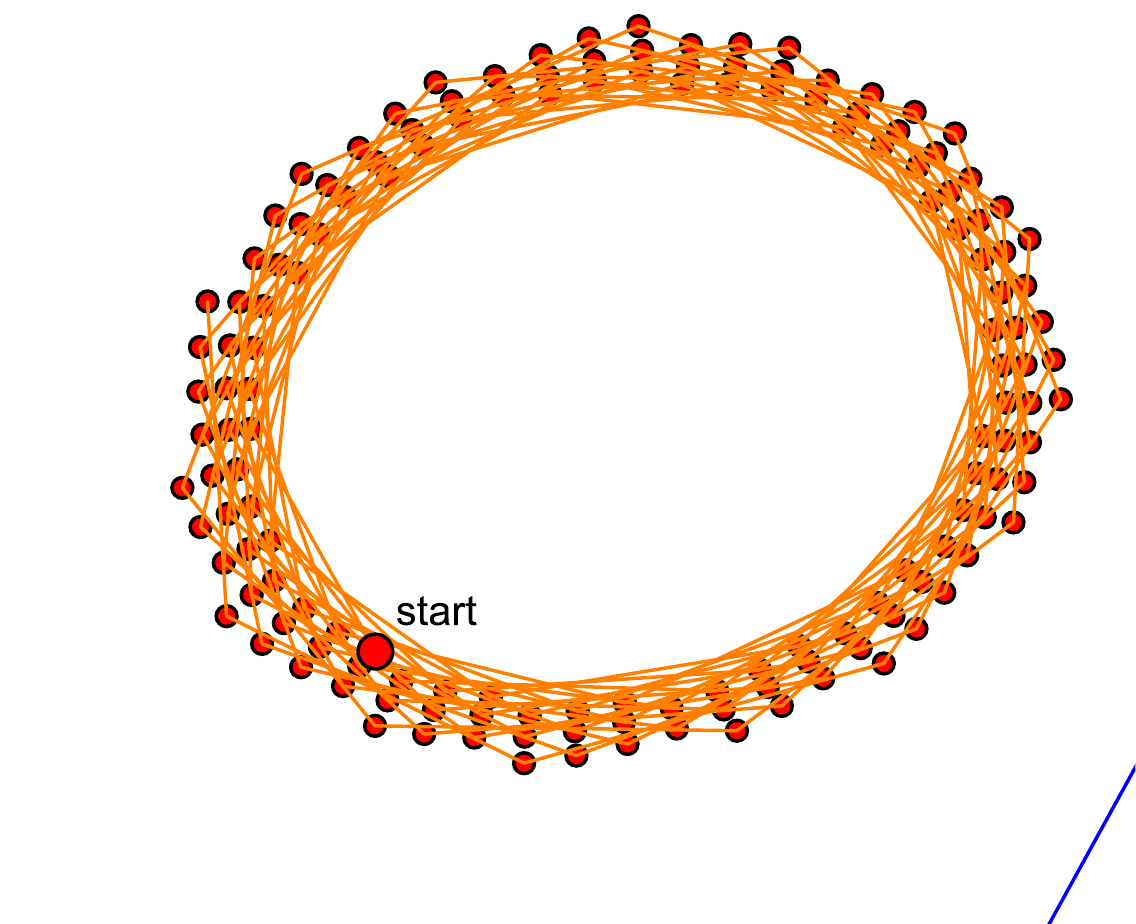} &
				\includegraphics[width=0.16\textwidth]{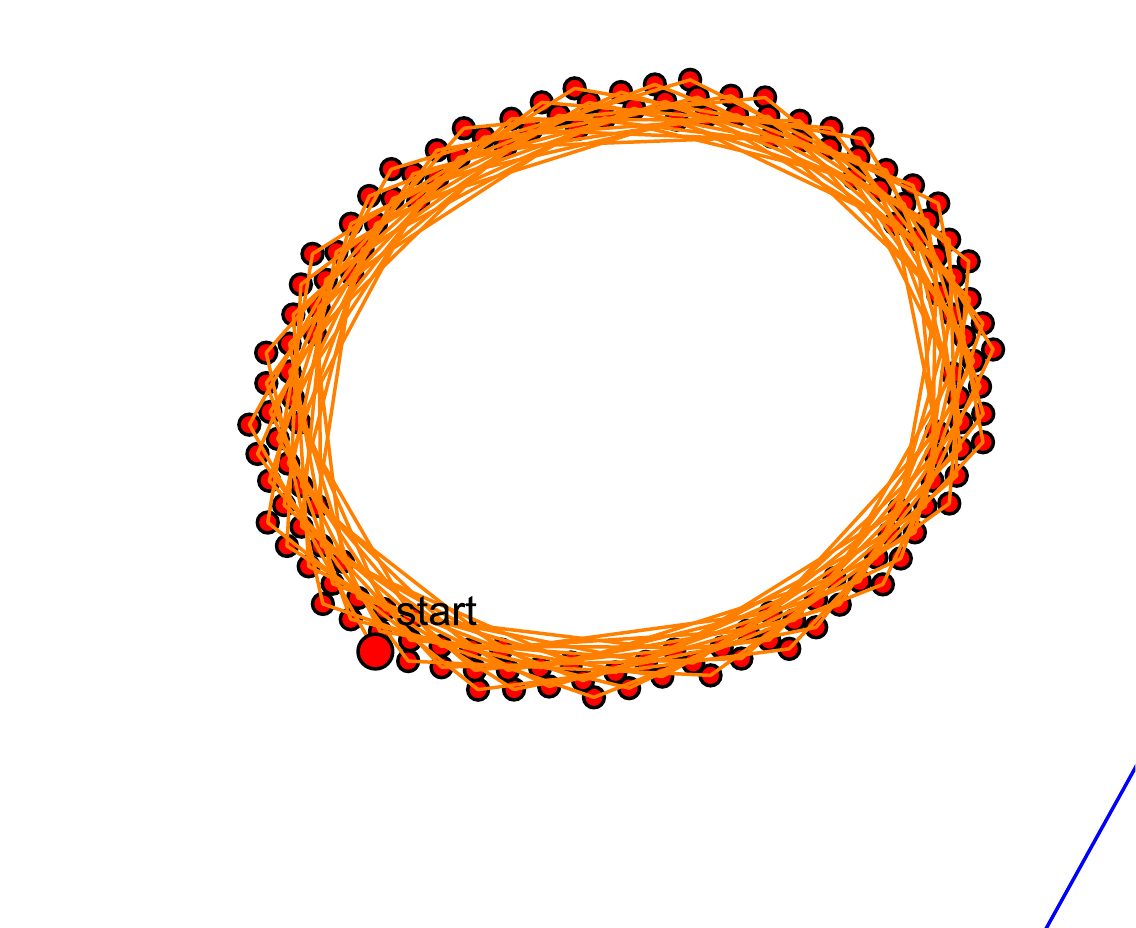} &
				\includegraphics[width=0.16\textwidth]{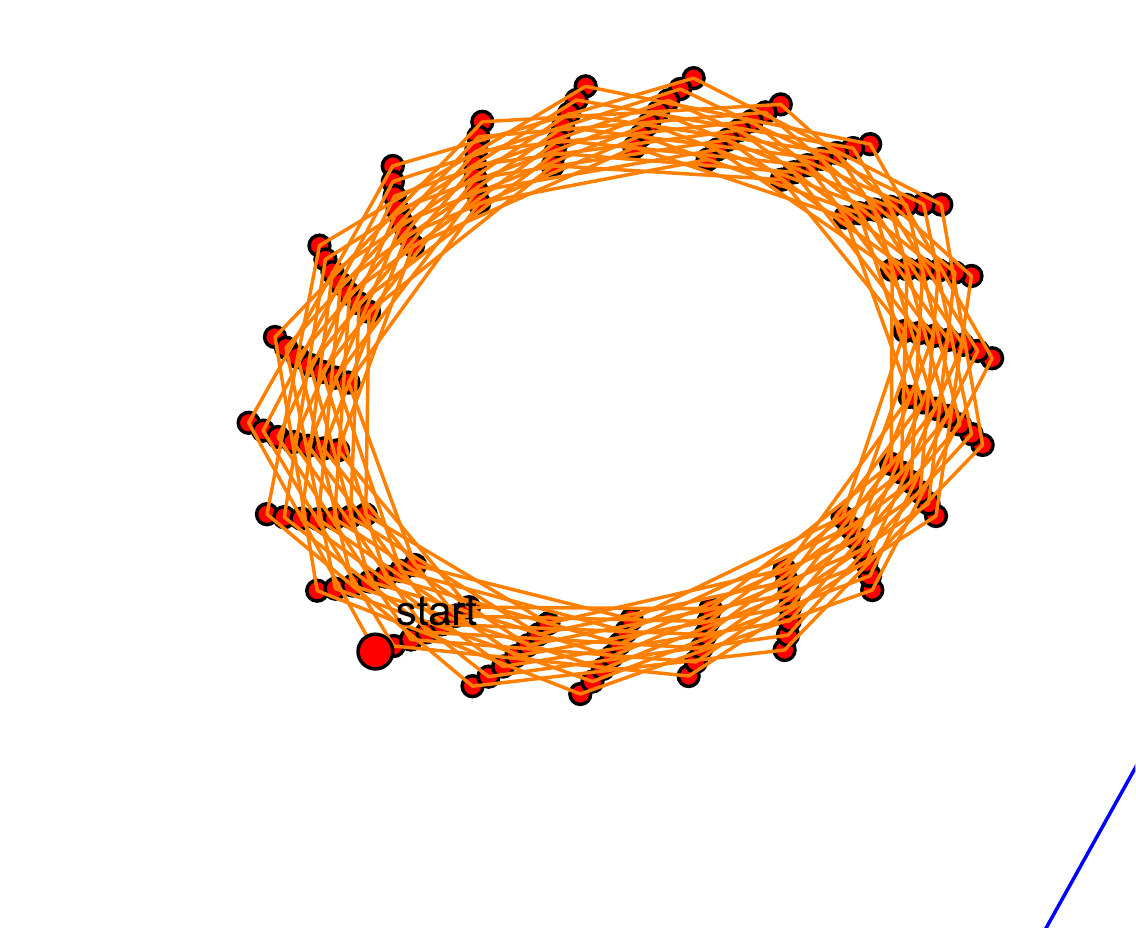} \\
				\includegraphics[width=0.16\textwidth]{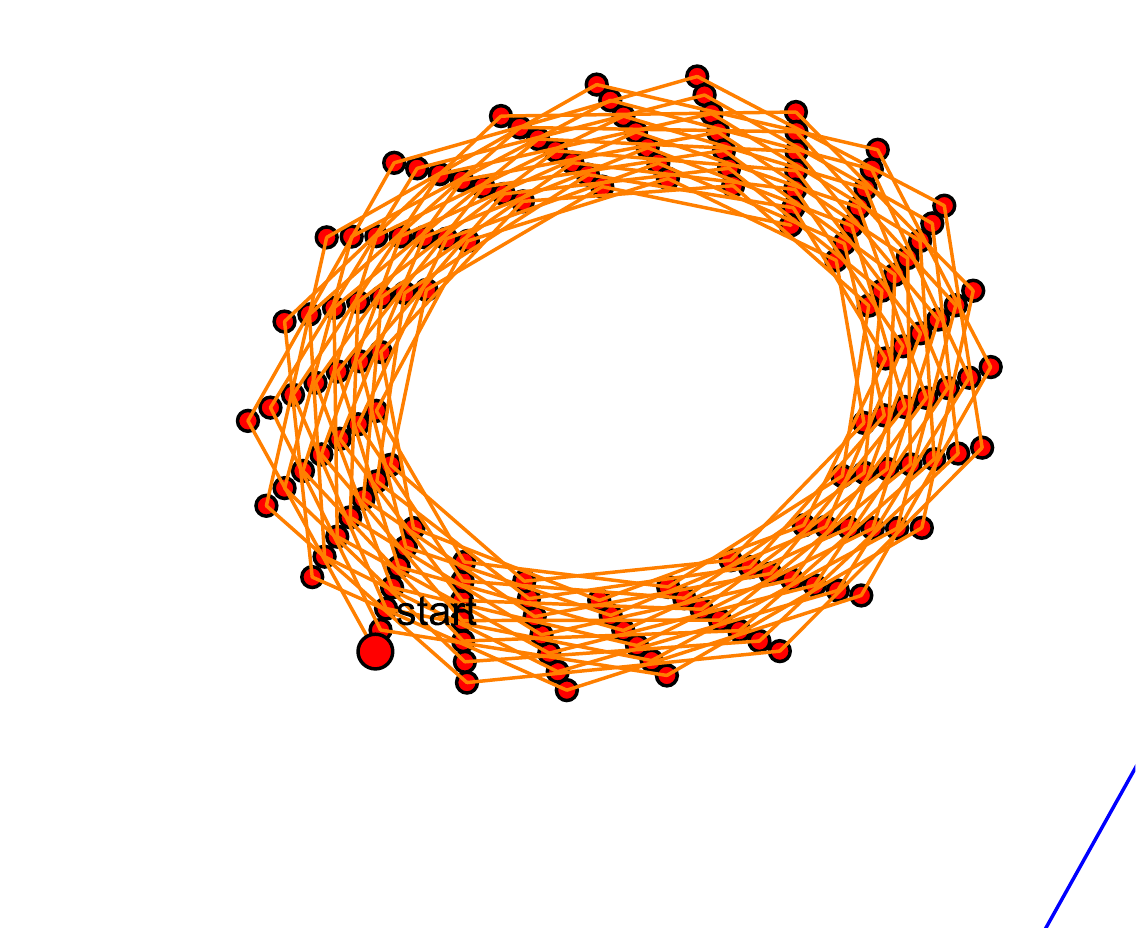} &
				\includegraphics[width=0.16\textwidth]{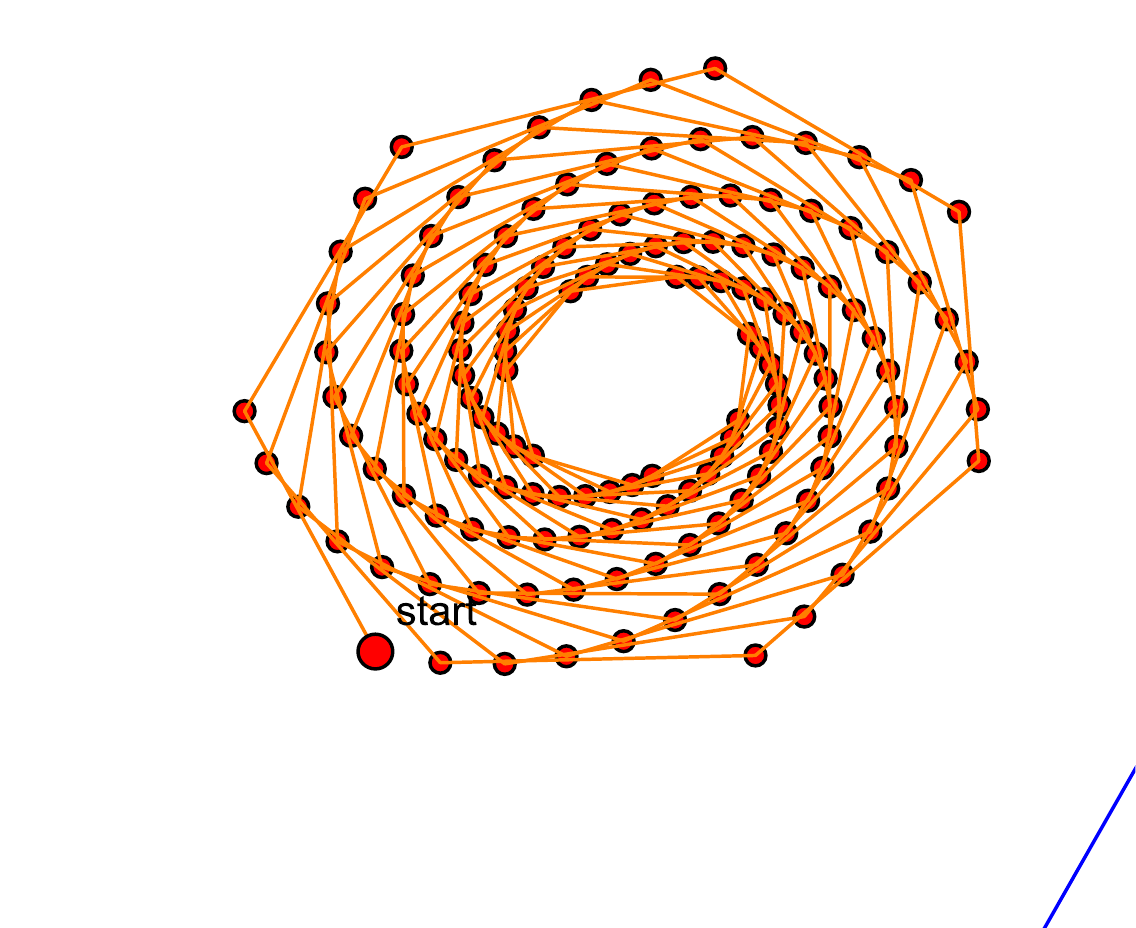} &
				\includegraphics[width=0.16\textwidth]{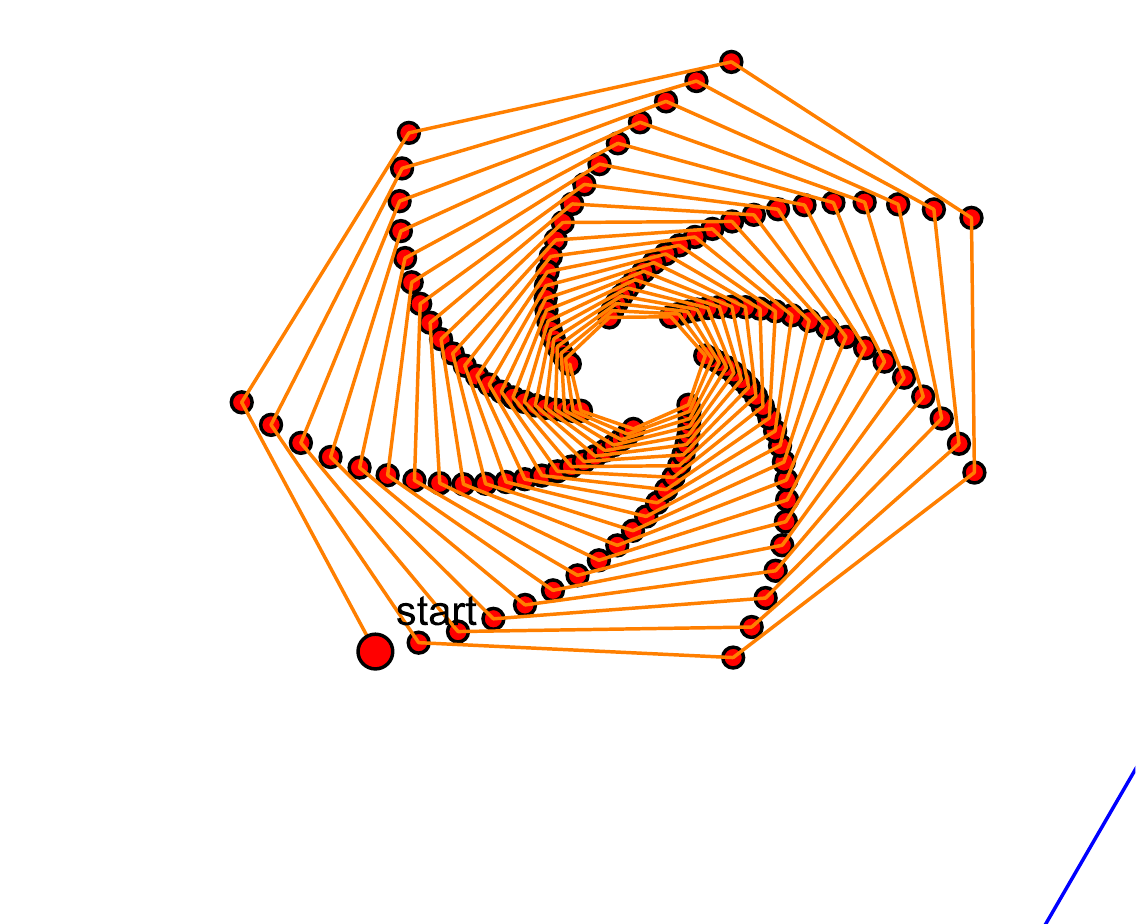} &
				\includegraphics[width=0.16\textwidth]{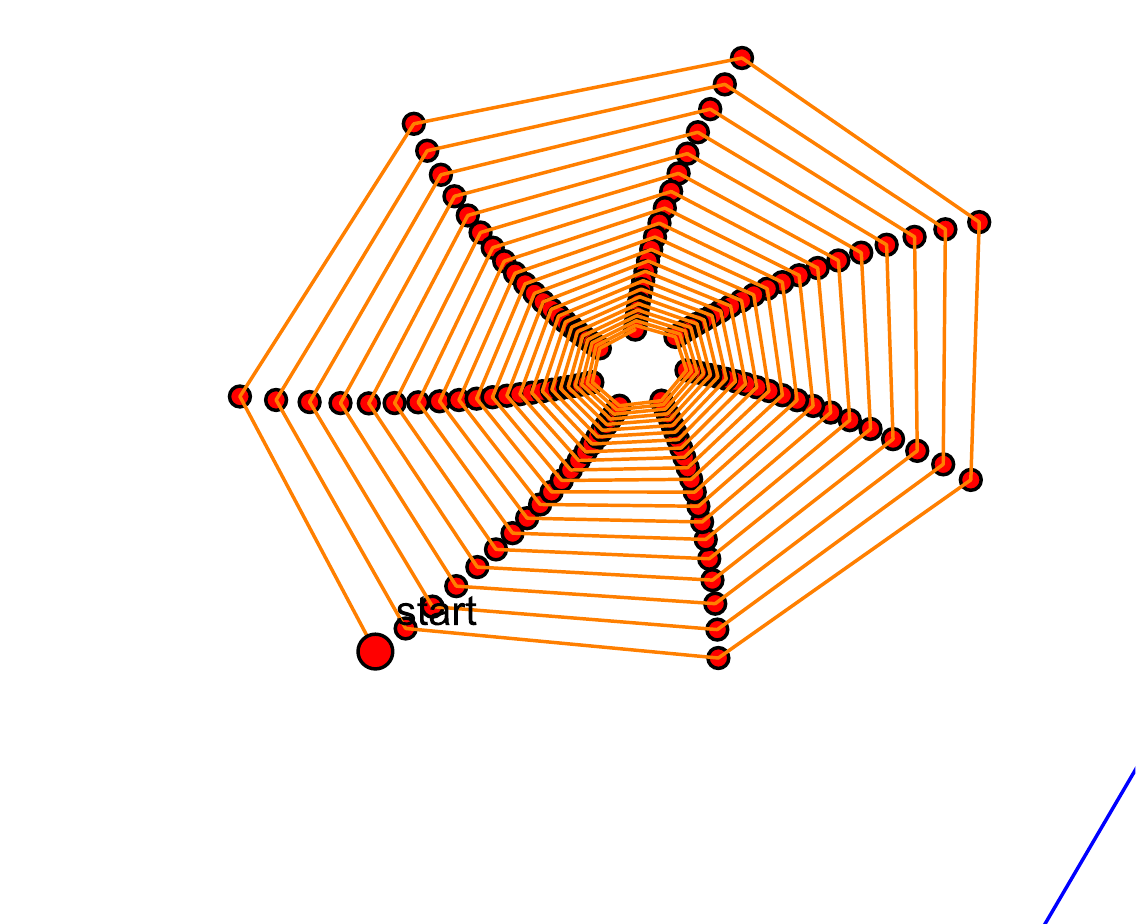} \\
				\includegraphics[width=0.16\textwidth]{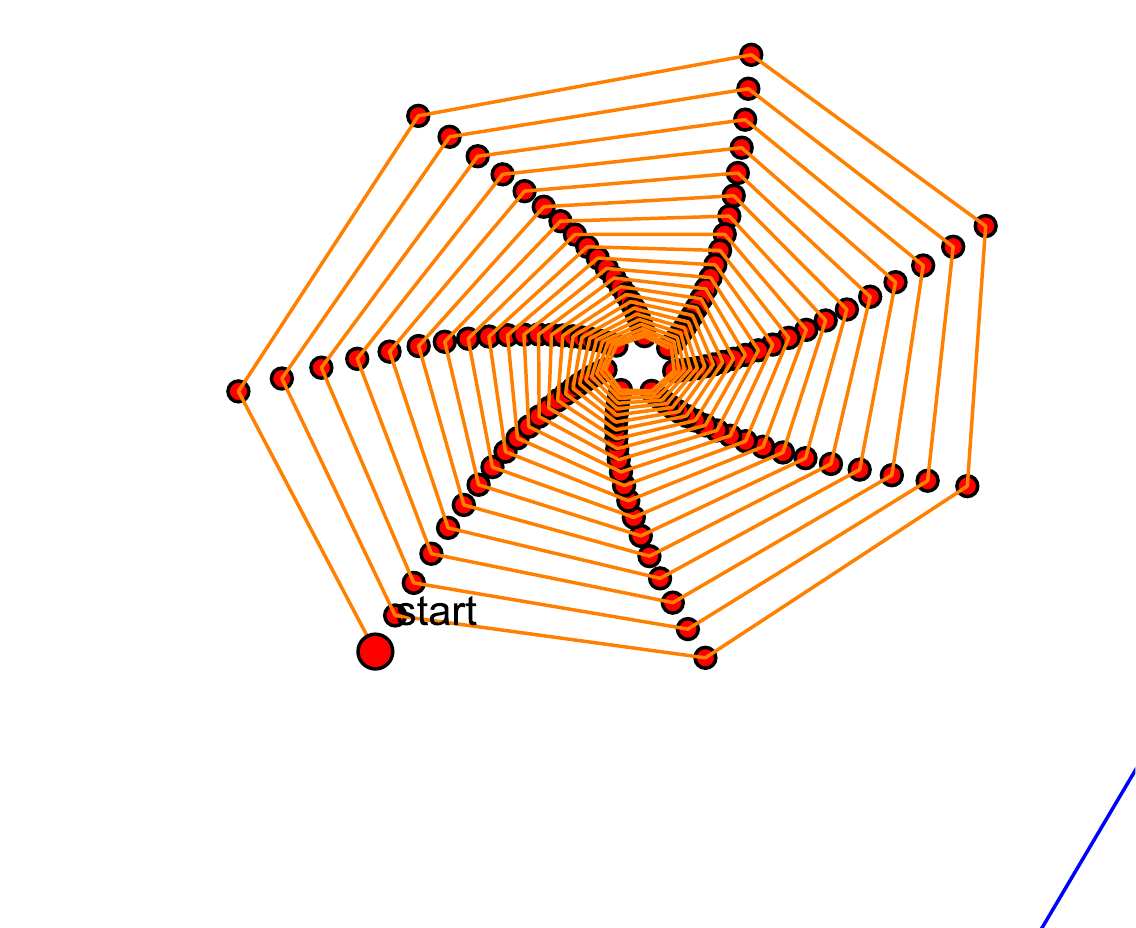} &
				\includegraphics[width=0.16\textwidth]{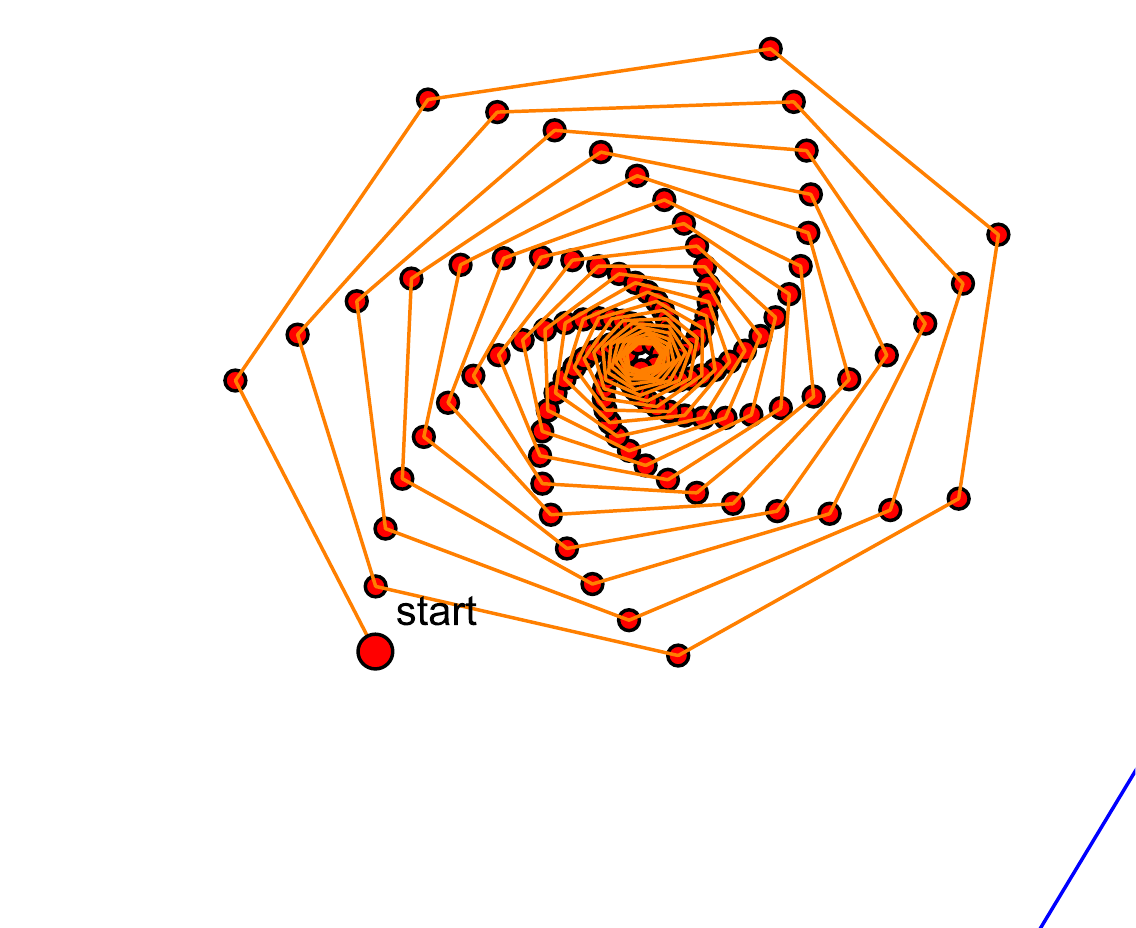} &
				\includegraphics[width=0.16\textwidth]{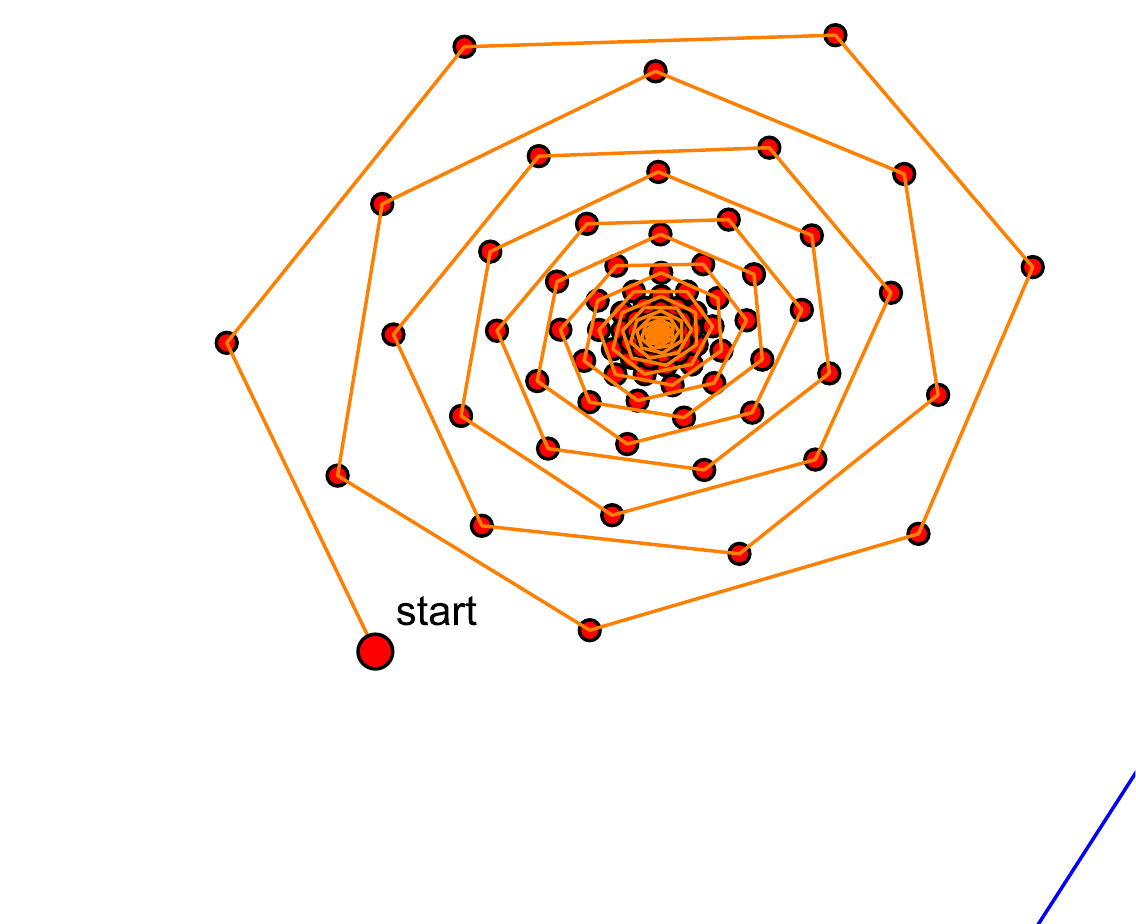} &
				\includegraphics[width=0.16\textwidth]{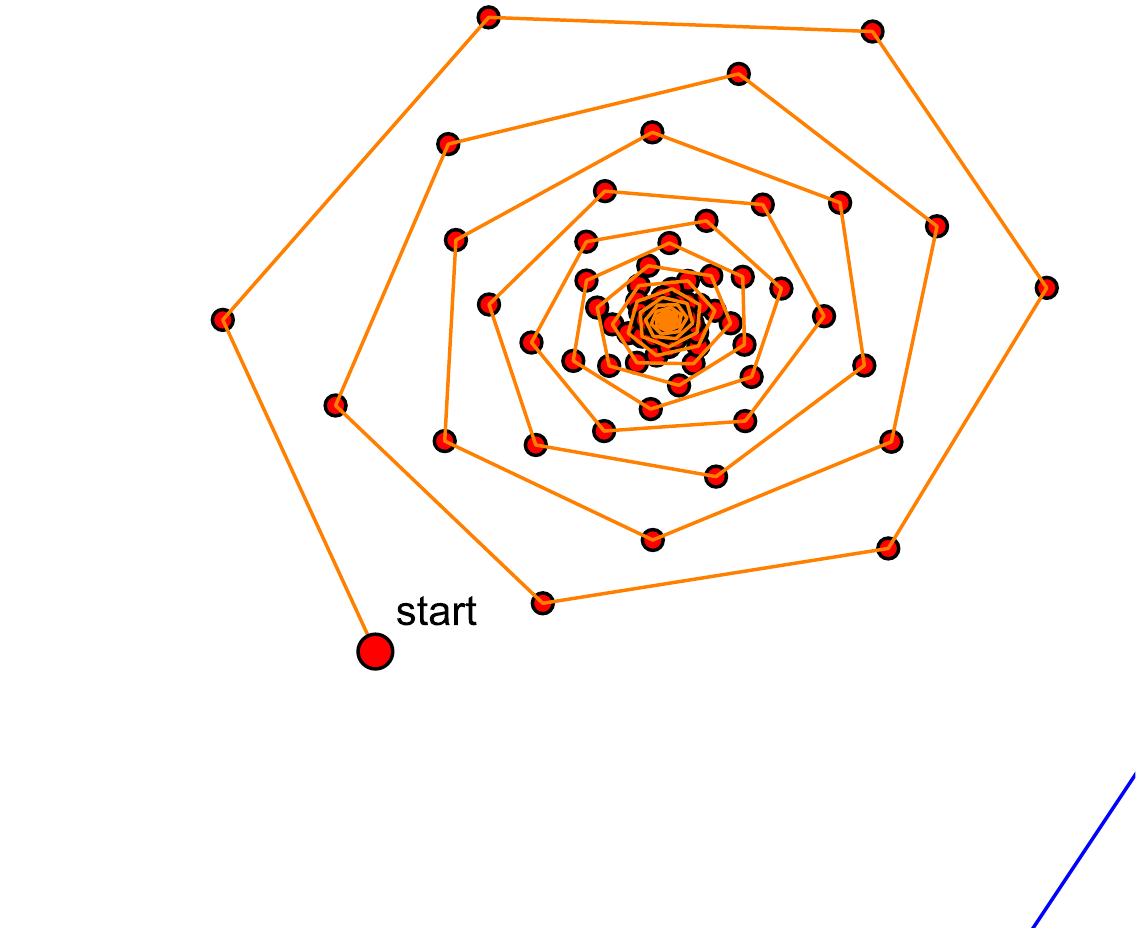}
			\end{tabular}	
			\caption{Evolution in behaviors near two period 2 points as the line slope is changed}
			\label{fig:twist}
		\end{center}
	\end{figure}
	
	This sensitivity to perturbations can be seen in Figures~\ref{fig:SourcesandSinks} and~\ref{fig:twist}. In the former we have connected every second iterate, and see how a small change in slope can affect which feasible point is converged to, as well as the appearance/disappearance of attractive domains. In the latter we plot every second iterate for $T_{E_2,L_m}$ with $300$ iterates, starting at $m=2$ (top left) we slowly rotate the line until we have $m={3/2}$ (bottom right). Part of the line is visible in the bottom right corner of each frame. In the initial configuration, the subsequence of iterates started near to the periodic point are repelled from it. As we rotate the line, we see that the ``speed'' at which they are repelled decreases  until eventually the periodic point becomes an attractive point instead of a repelling point.
	
	\subsection{Studying Convergence: Numerical Motivations}\label{subsec:Numerical}
	
	The complicated nature of the singular set precludes any possibility of constructing a Lyapunov function in any sizable region about the feasible point. Indeed, attempts to even numerically construct the level curves such a function might have near a feasible point proved unstable. Instead we refine our numerical-graphical method of discovery. The method we used for Figure~\ref{fig:BasinsAll}, though useful for discovery, is not, in itself, sufficient for fully understanding the behaviors, even for one specific ellipse and line. There are several reasons for this.
		
	\begin{enumerate}
		\item There may be other periodic points we cannot see because they are repelling or their attractive domains are too small.
		\item The potential for numerical error is accentuated by the fact that the projection onto the ellipse is specified as the root function---induced by the Lagrangian system---whose calculation is, in some configurations, complicated by the presence of nearby incorrect roots.
		\item This method of visualization may be deceptive, as it precludes us from seeing accurately the extent and shape of attractive domains.
	\end{enumerate}
	As an example of the latter, notice how the patterns of the iterates in Figure~\ref{fig:BasinsAll} form orderly spirals. This lovely pattern seems to hold for all the cases we have looked at. If we zoom in on the spirals we see what look like twisting galaxies (see Figure~\ref{fig:BasinSpirals}). Intuition would suggest to us---incorrectly---that this is perhaps indicative of smooth boundaries for the domains.
	\begin{figure}
		\begin{center}
			\setlength\tabcolsep{0.01\textwidth} 
			\begin{tabular}{cccc}
				\includegraphics[width=0.16\textwidth]{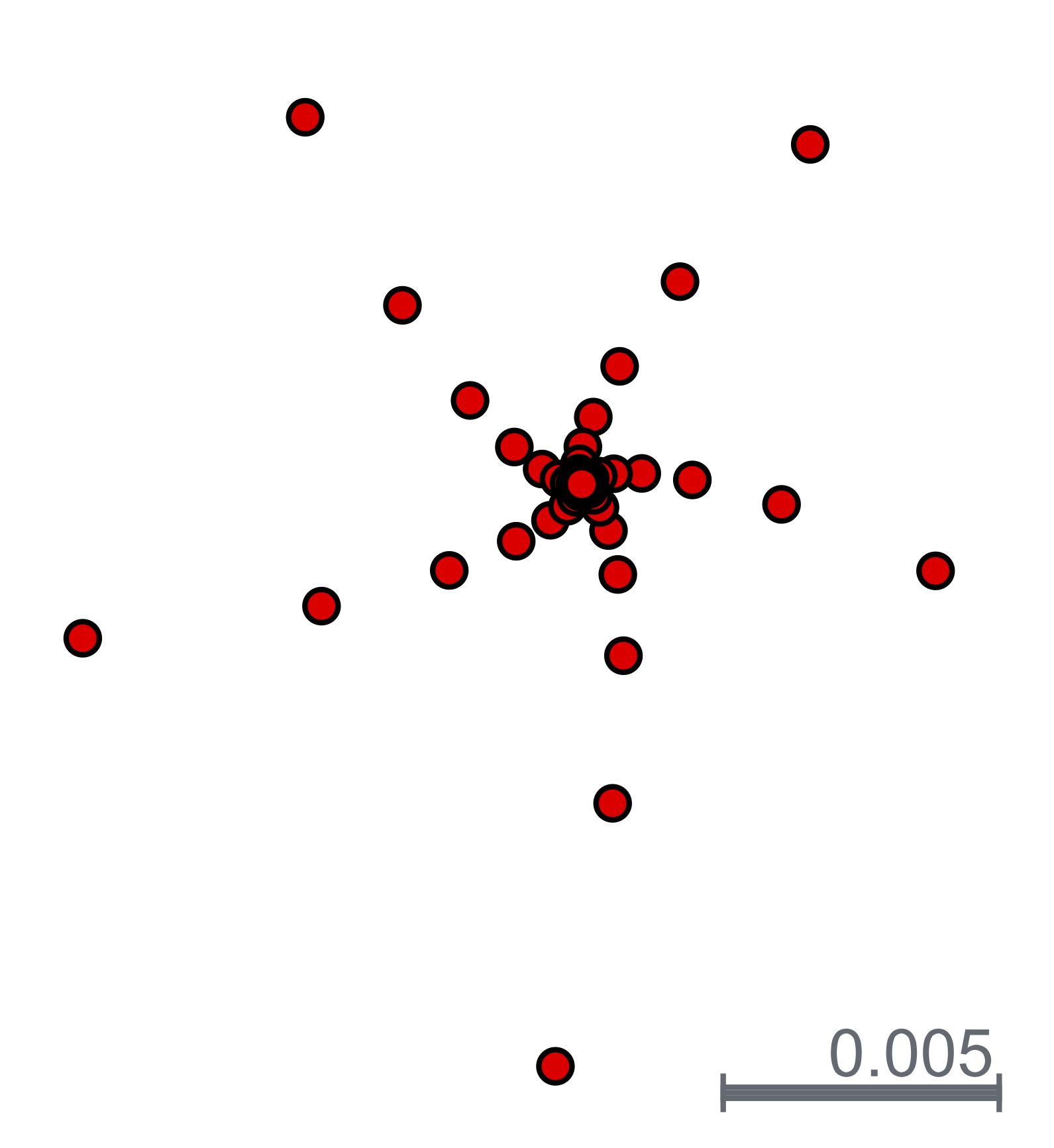} &
				\includegraphics[width=0.16\textwidth]{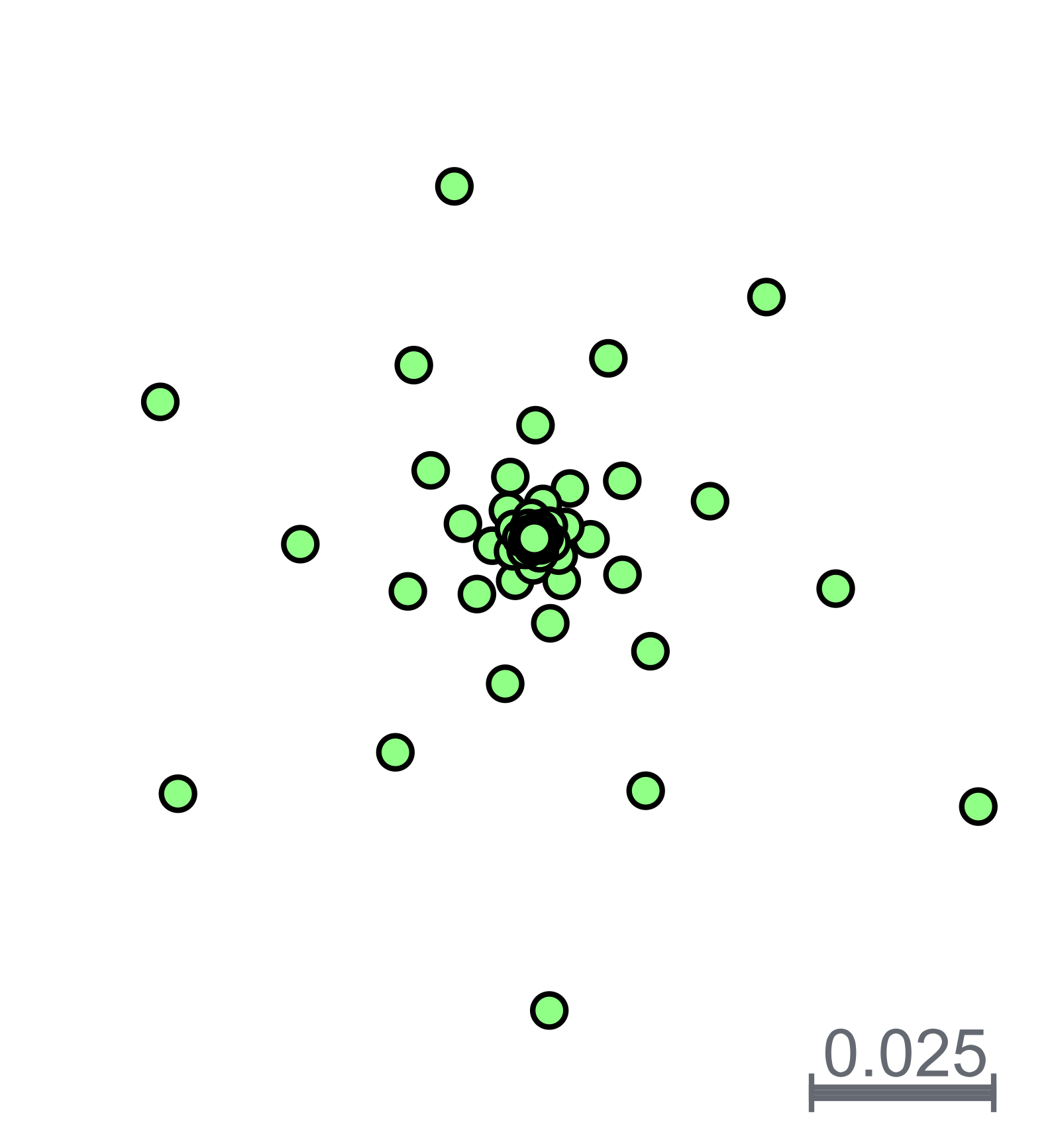} &
				\includegraphics[width=0.16\textwidth]{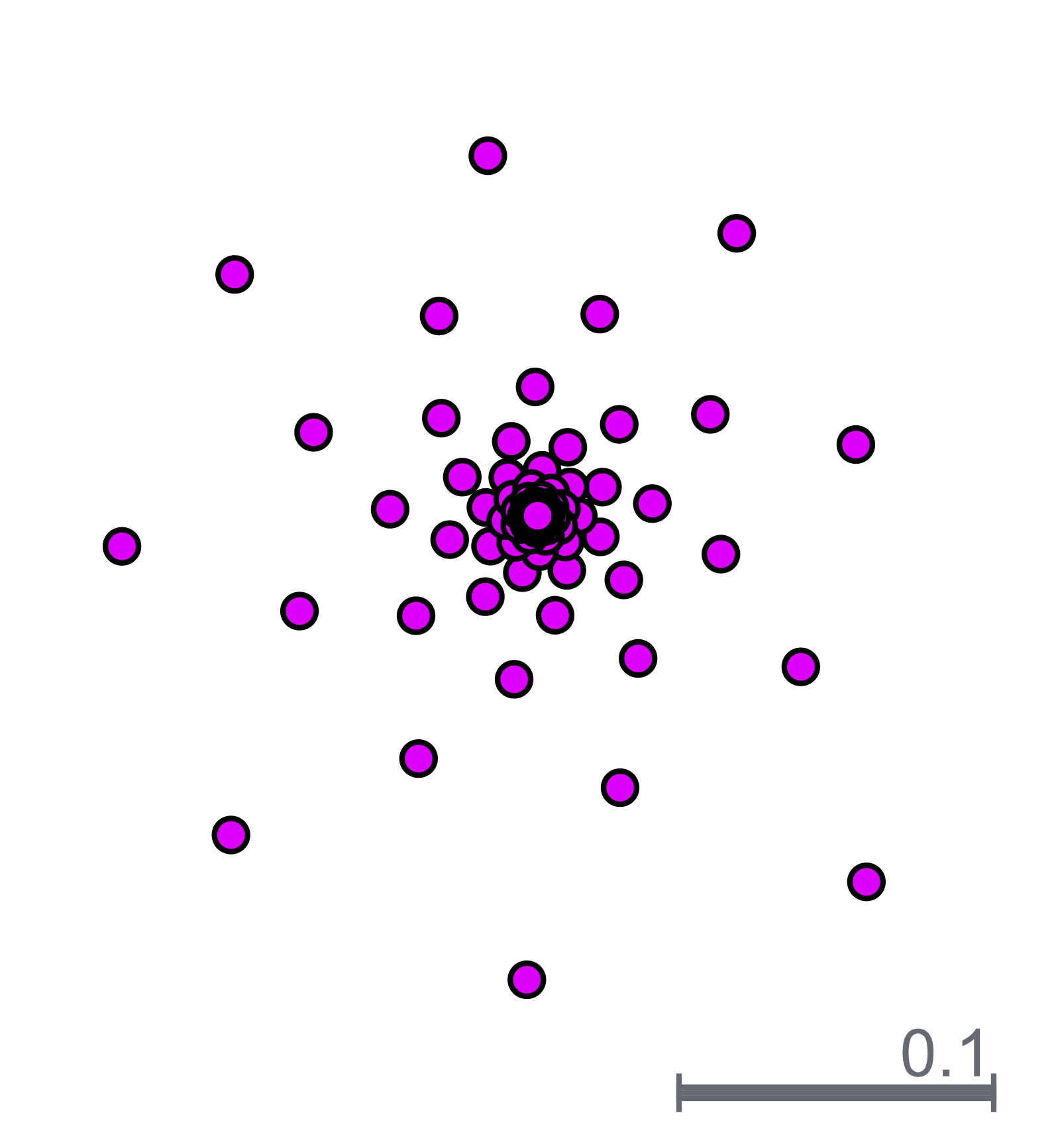} &
				\includegraphics[width=0.16\textwidth]{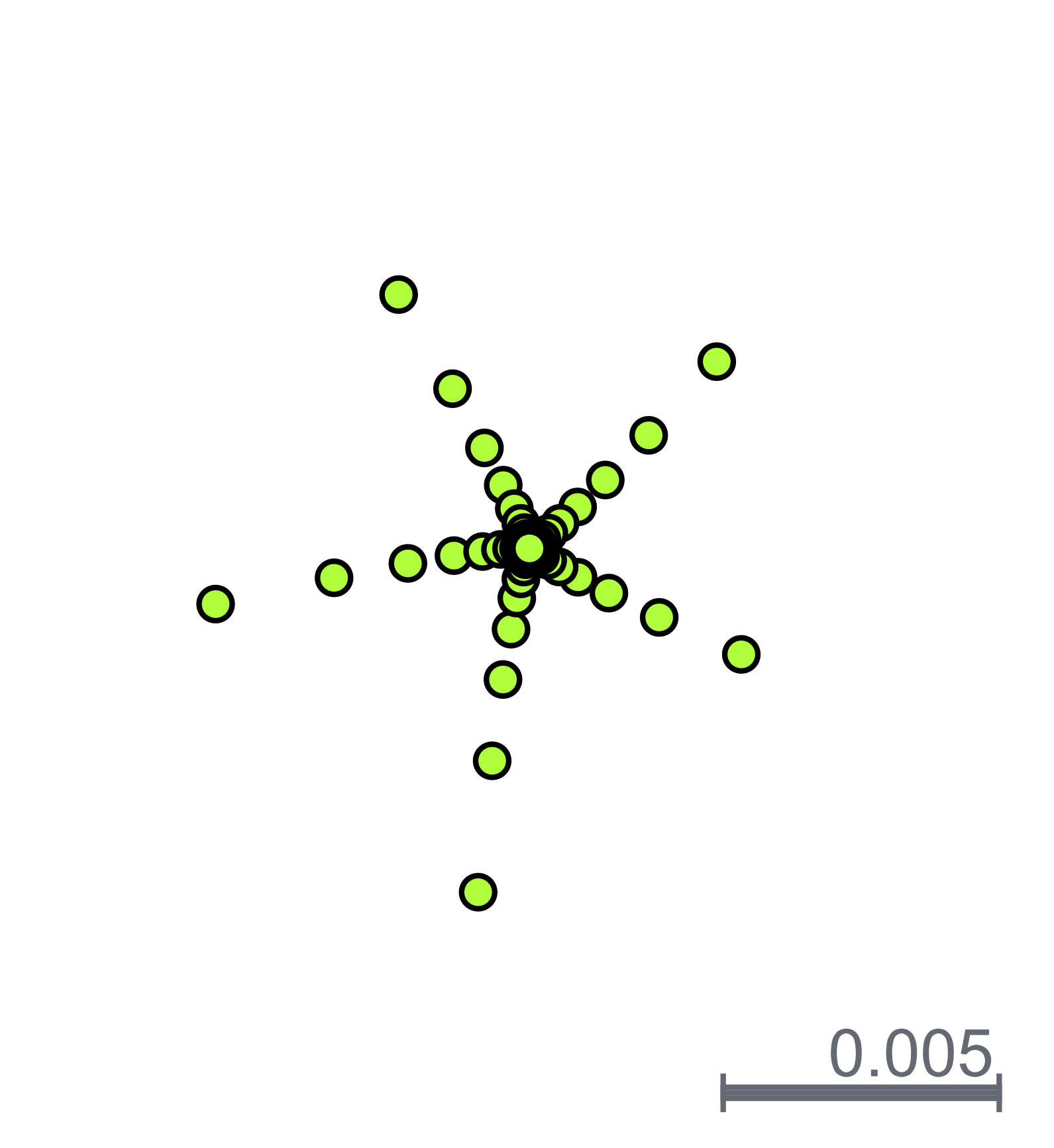} \\
				\includegraphics[width=0.16\textwidth]{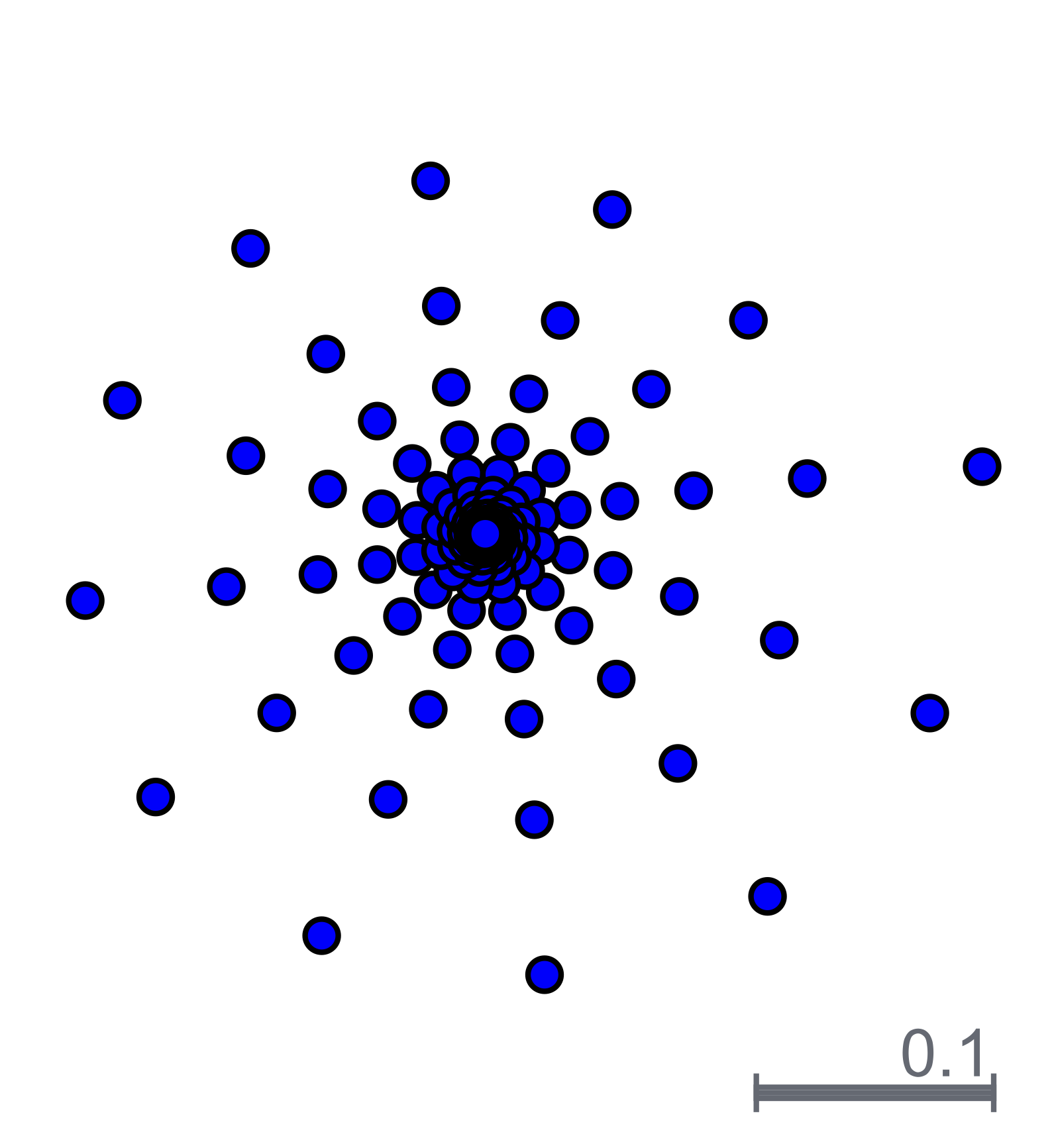} &
				\includegraphics[width=0.16\textwidth]{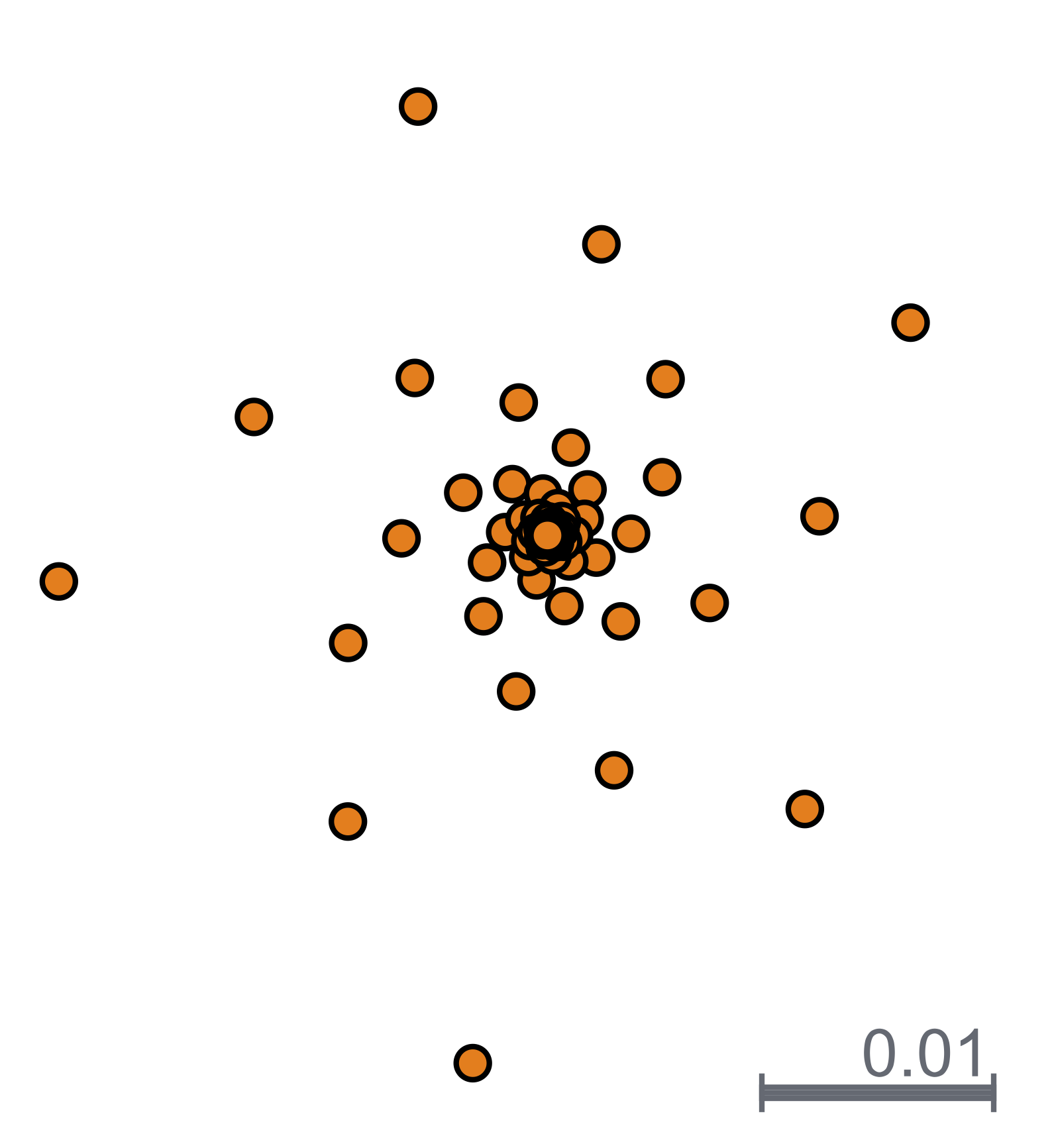} &
				\includegraphics[width=0.16\textwidth]{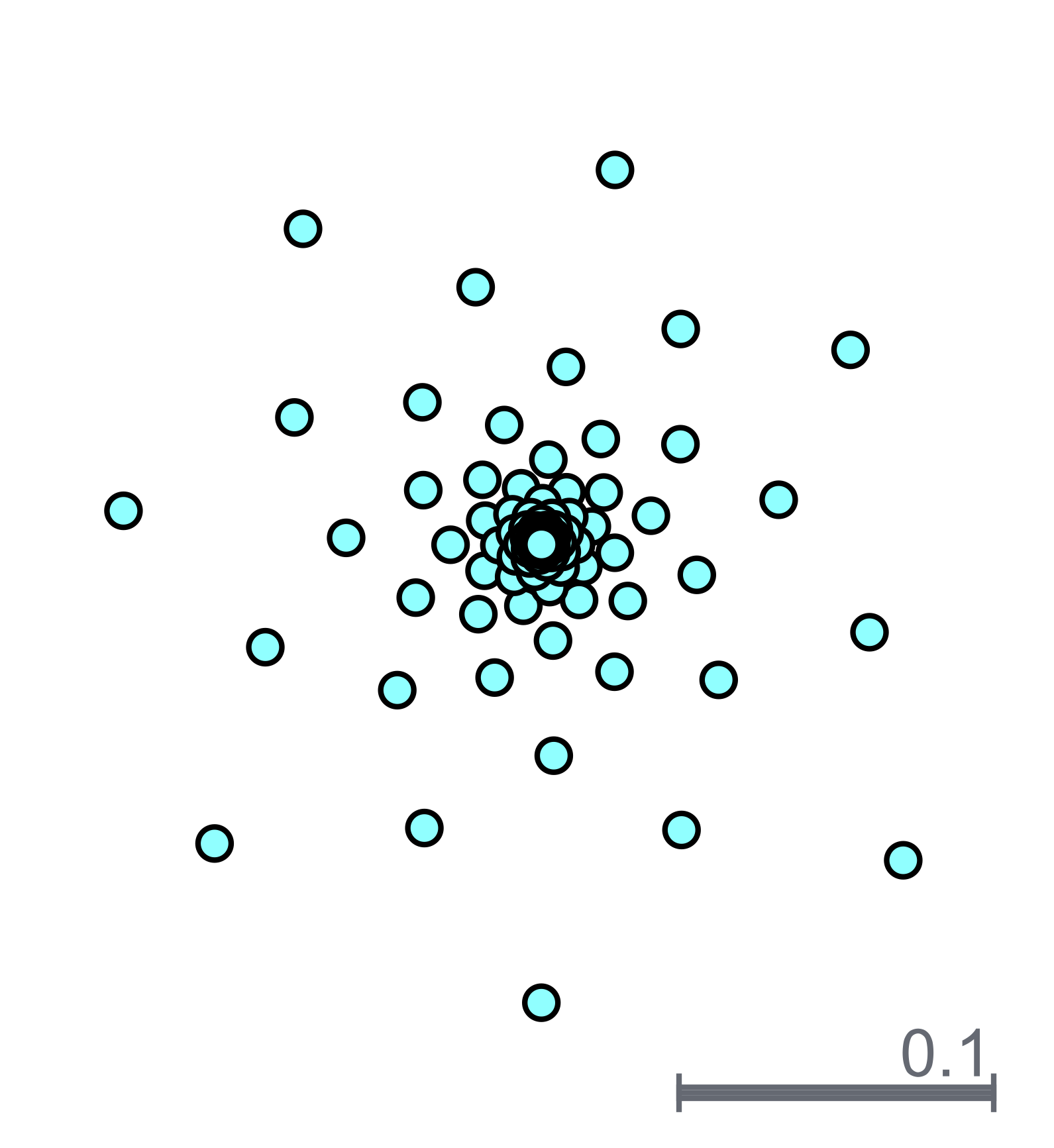} &
				\includegraphics[width=0.16\textwidth]{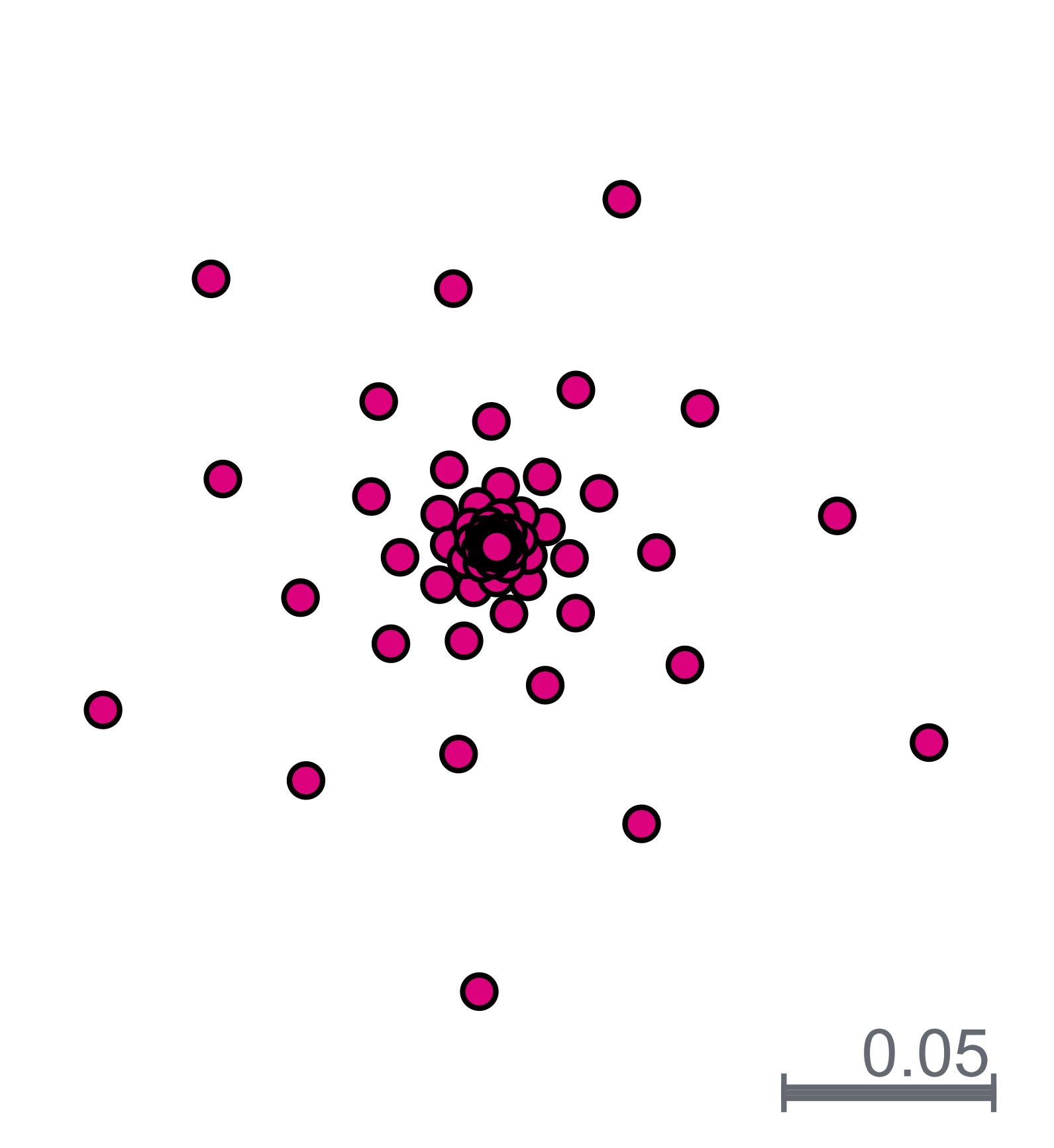} \\
				\includegraphics[width=0.16\textwidth]{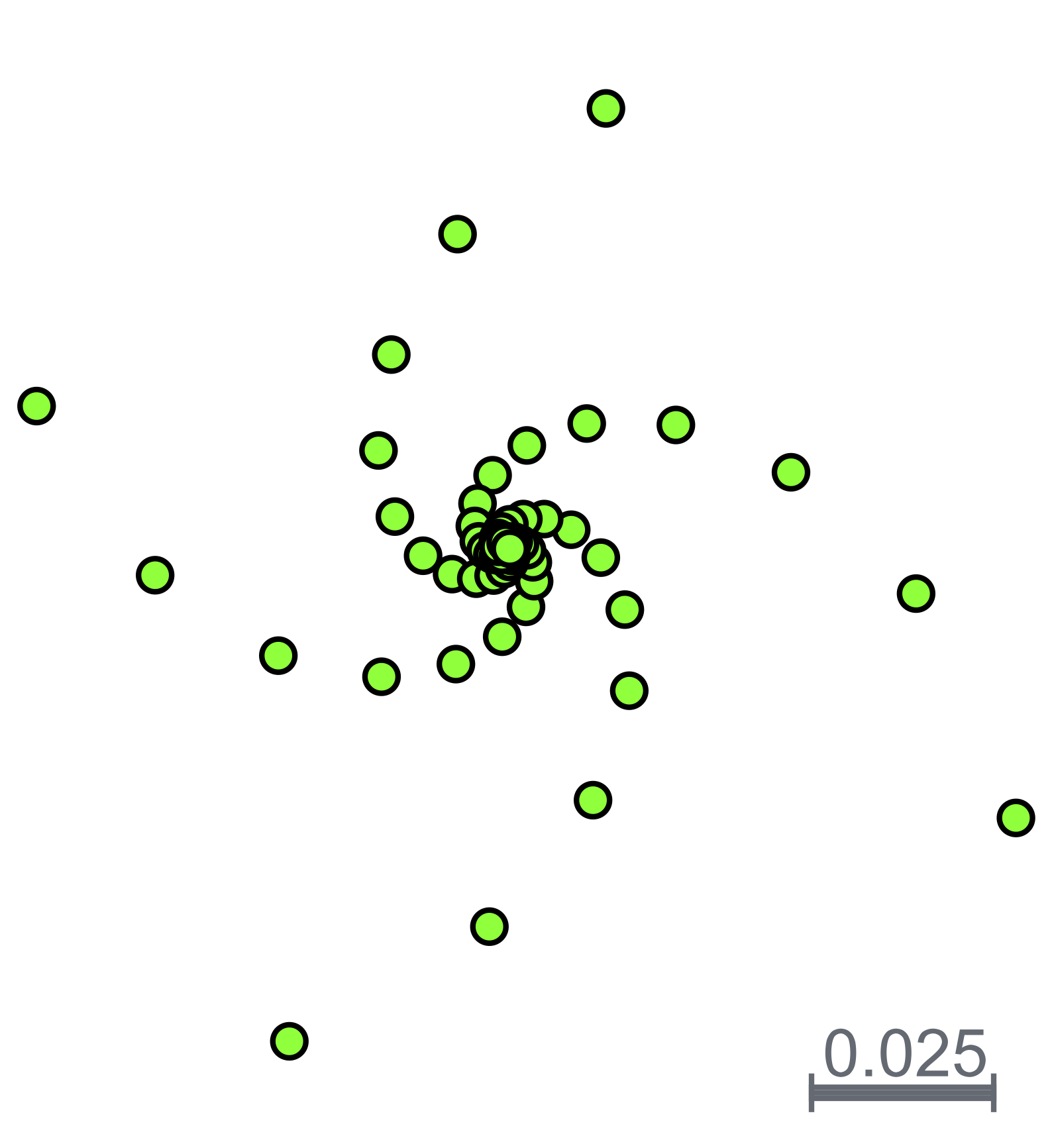} &
				\includegraphics[width=0.16\textwidth]{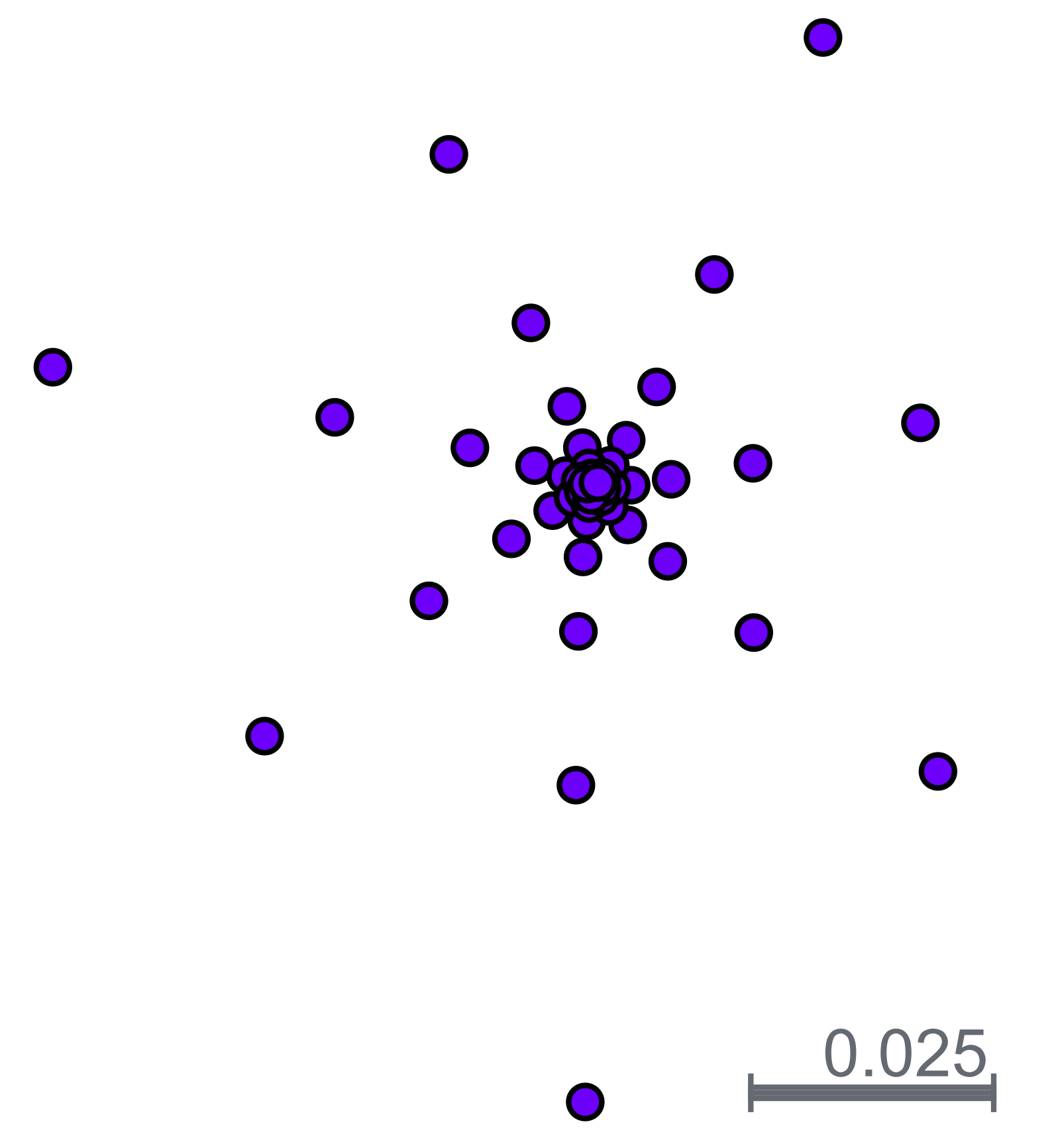} &
				\includegraphics[width=0.16\textwidth]{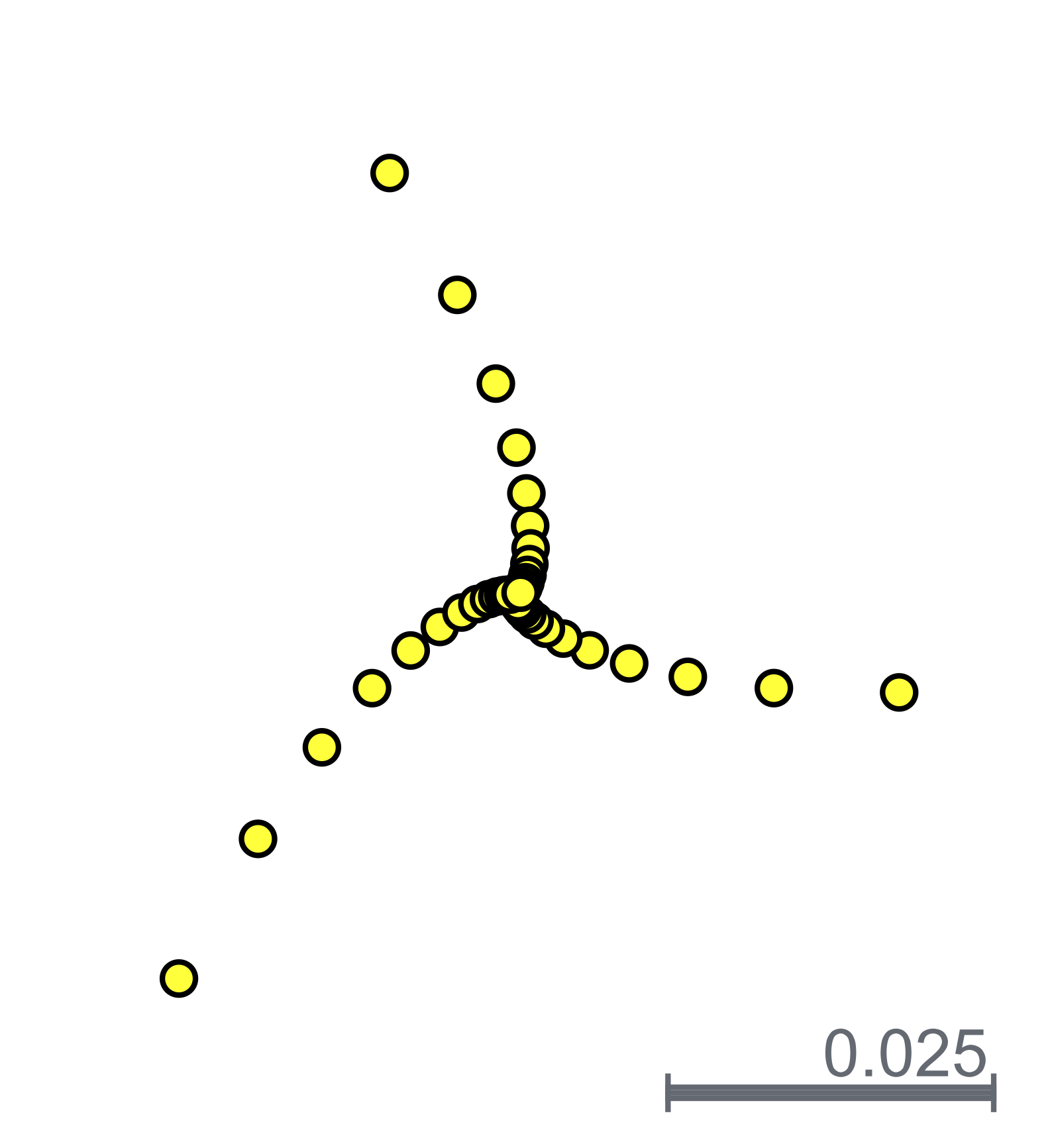} &
				\includegraphics[width=0.16\textwidth]{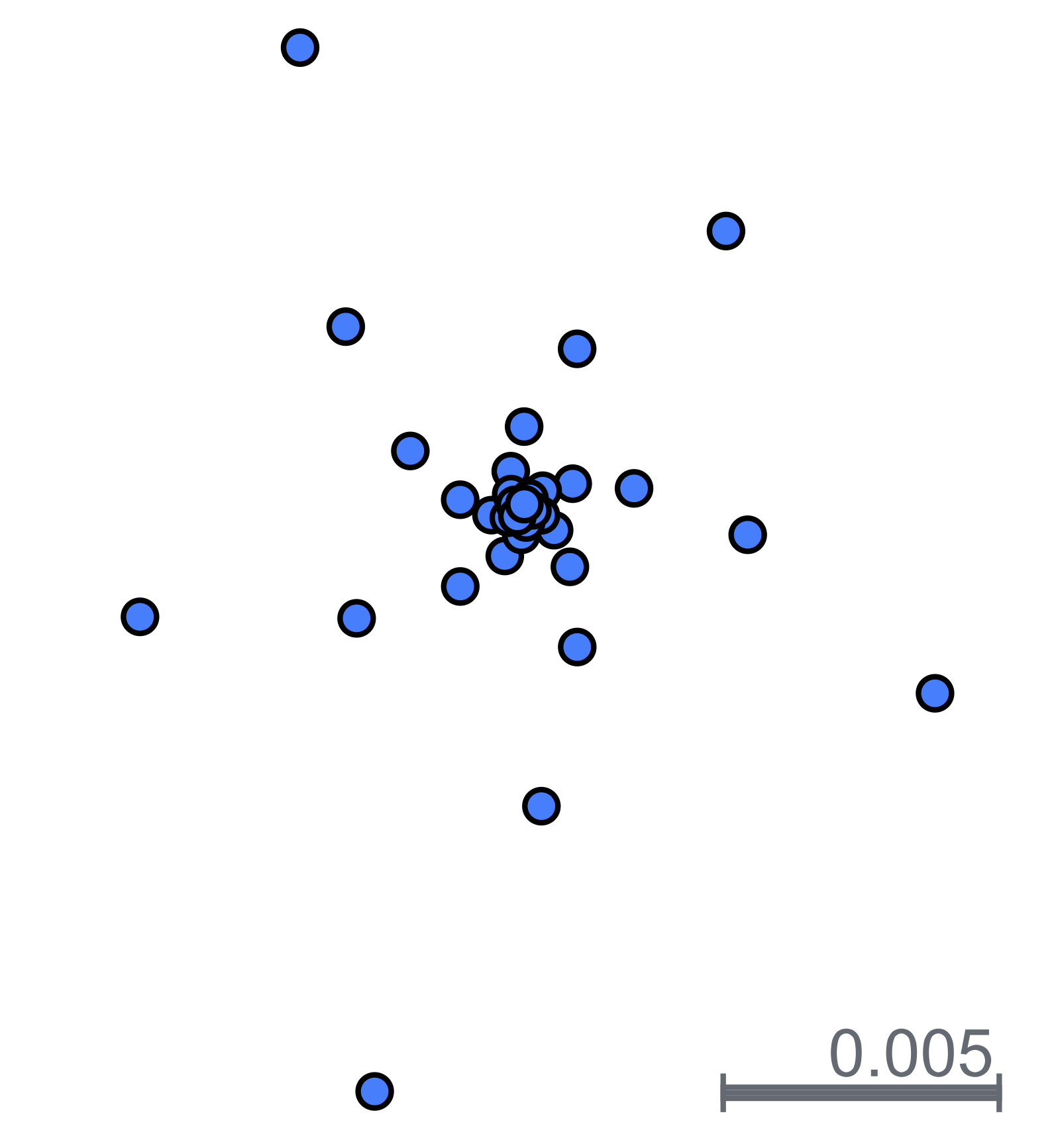}
			\end{tabular}		
			\caption{Close up view of the spirals in seen in Figure~\ref{fig:BasinsAll}}
			\label{fig:BasinSpirals}
		\end{center}
	\end{figure}
		
	\subsection{Visualization through Parallelization}
	
	Seeking clearer pictures with finer resolution, we implemented a new version of our code. In this new version we specify a resolution and for each pixel compute we compute the midpoint, calling it $x_0$. Once computed, the location of $x_{1,000}$ is checked against a list containing the feasible points and approximate periodic points, and the pixel is colored according to which list member it is nearest to. 
	
	The efficacy of this technique is demonstrated in Figures~\ref{fig:sourcebasinspirals} and~\ref{fig:parellelzoomed}. The former clearly depicts the complex structure of the domains of attraction for the two feasible points in the case of $T_{L_2,E_2}$ where two period 2 repelling points are present. Note the interweaving of the attractive domains near the repelling points. The latter shows the domains of attraction for $T_{E_8,L_6}$ in the same region as shown in Figure~\ref{fig:BasinsAll}.
	
	\begin{figure}
		\begin{center}
			\fbox{\includegraphics[angle=90,width=0.7\textwidth]{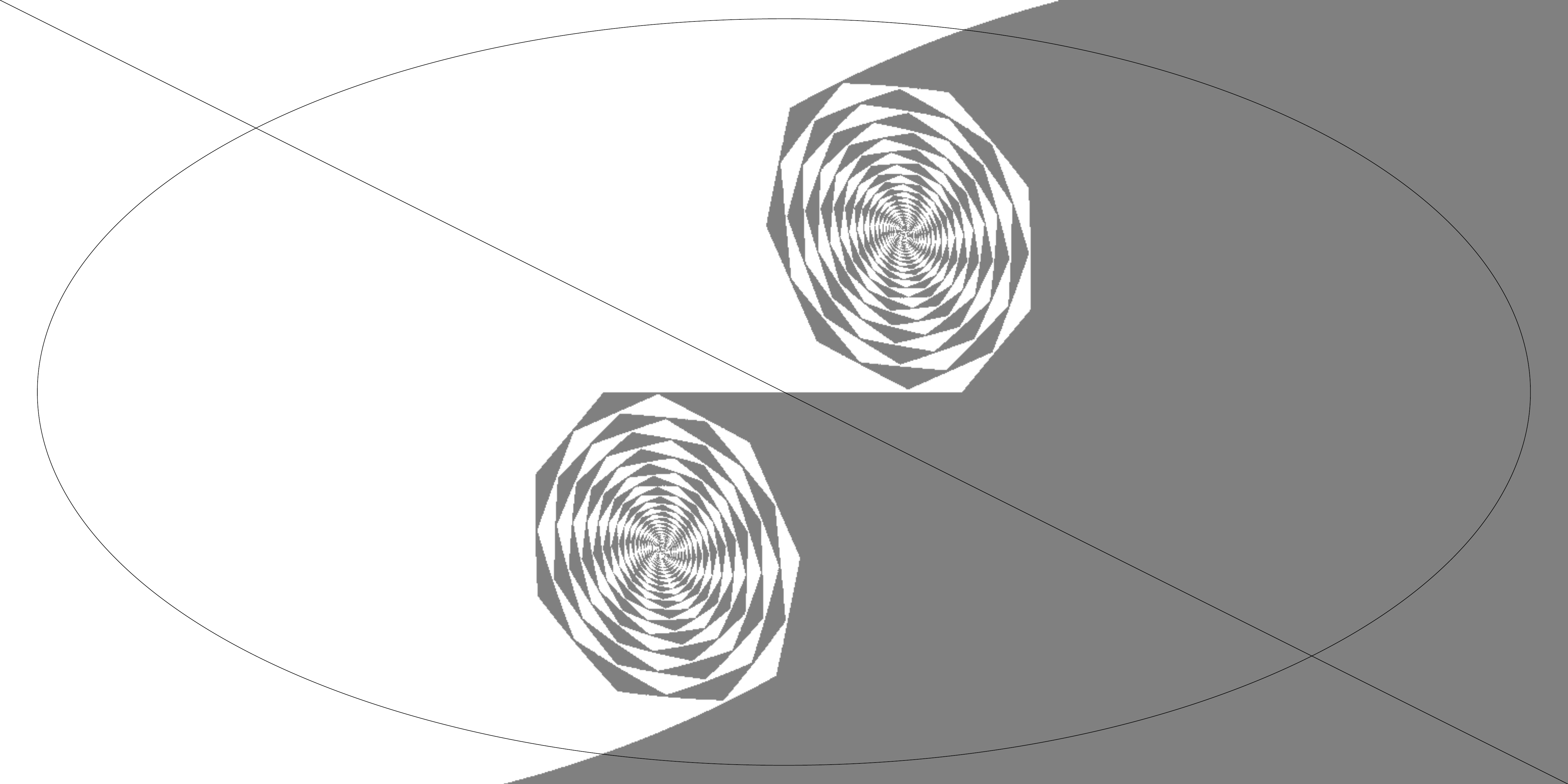}}
			\caption{Domains of attraction for the two feasible points of $T_{E_2,L_2}$}
			\label{fig:sourcebasinspirals}
		\end{center}
	\end{figure}
	
	\begin{figure}
		\begin{center}
			\fbox{\includegraphics[angle=90,width=0.7\textwidth]{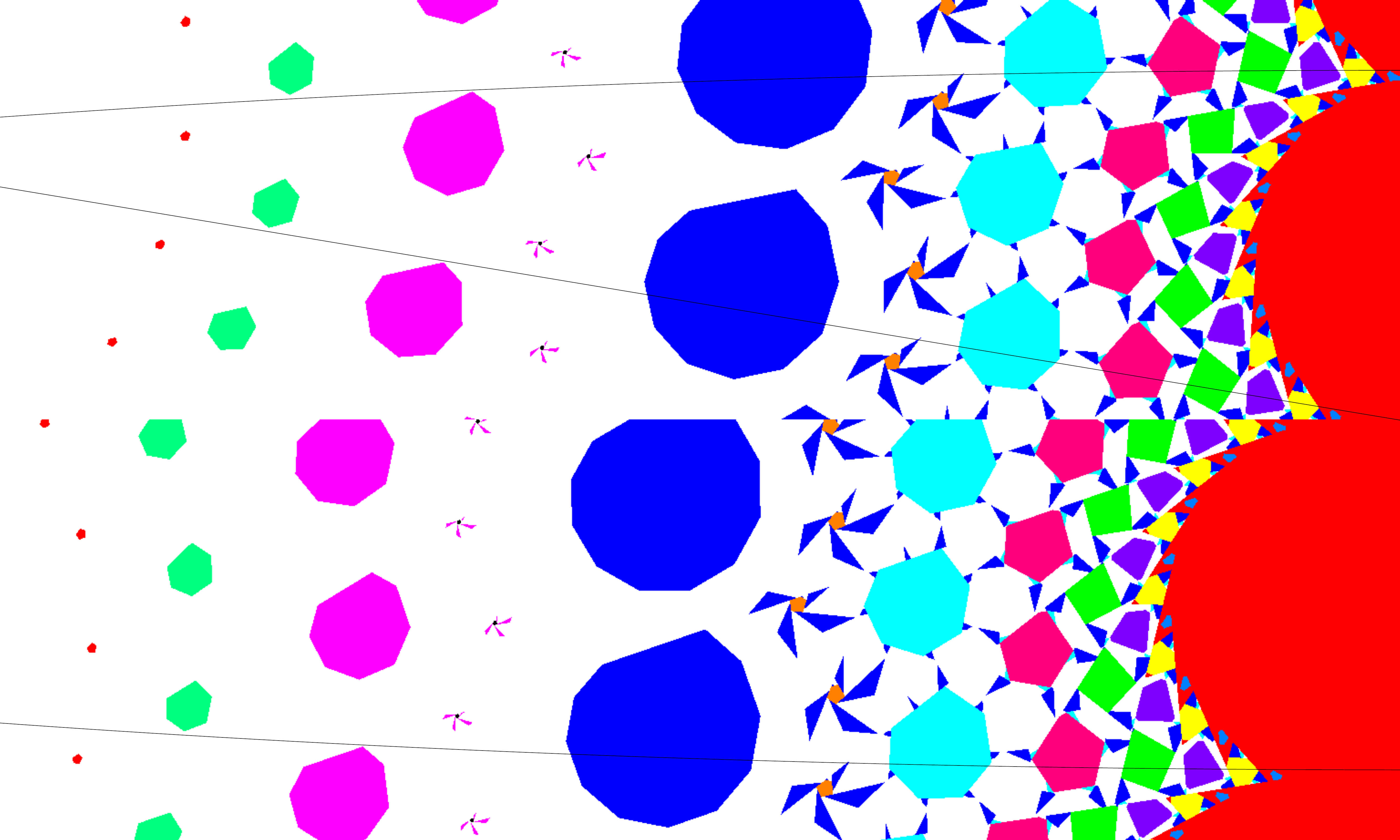}}
			\caption{Domains of attraction for $T_{E_8,L_6}$. Compare with Figure~\ref{fig:BasinsAll}}
			\label{fig:parellelzoomed}
		\end{center}
	\end{figure}
	
	This new implementation was written in such a way as to leverage the highly parallel nature of \textsc{gpu} devices, although it may also be run on regular \textsc{cpu}s. This allowed us to compute the colorings for many pixels simultaneously, reducing the time needed to produce the images and simultaneously affording us the ability to produce images with a much finer resolution. With this method we were able to see the behavior of the system over a larger area, as shown in Figure~\ref{fig:FullEllipseBasinsFullPage}. Note that these images benefit greatly from color coding of the domains. Color versions of the above figures and other images we produced can be found in the appendix \cite{APPENDIX}.
	
	\begin{figure}
		\begin{center}
			\fbox{\includegraphics[angle=90,width=0.7\textwidth]{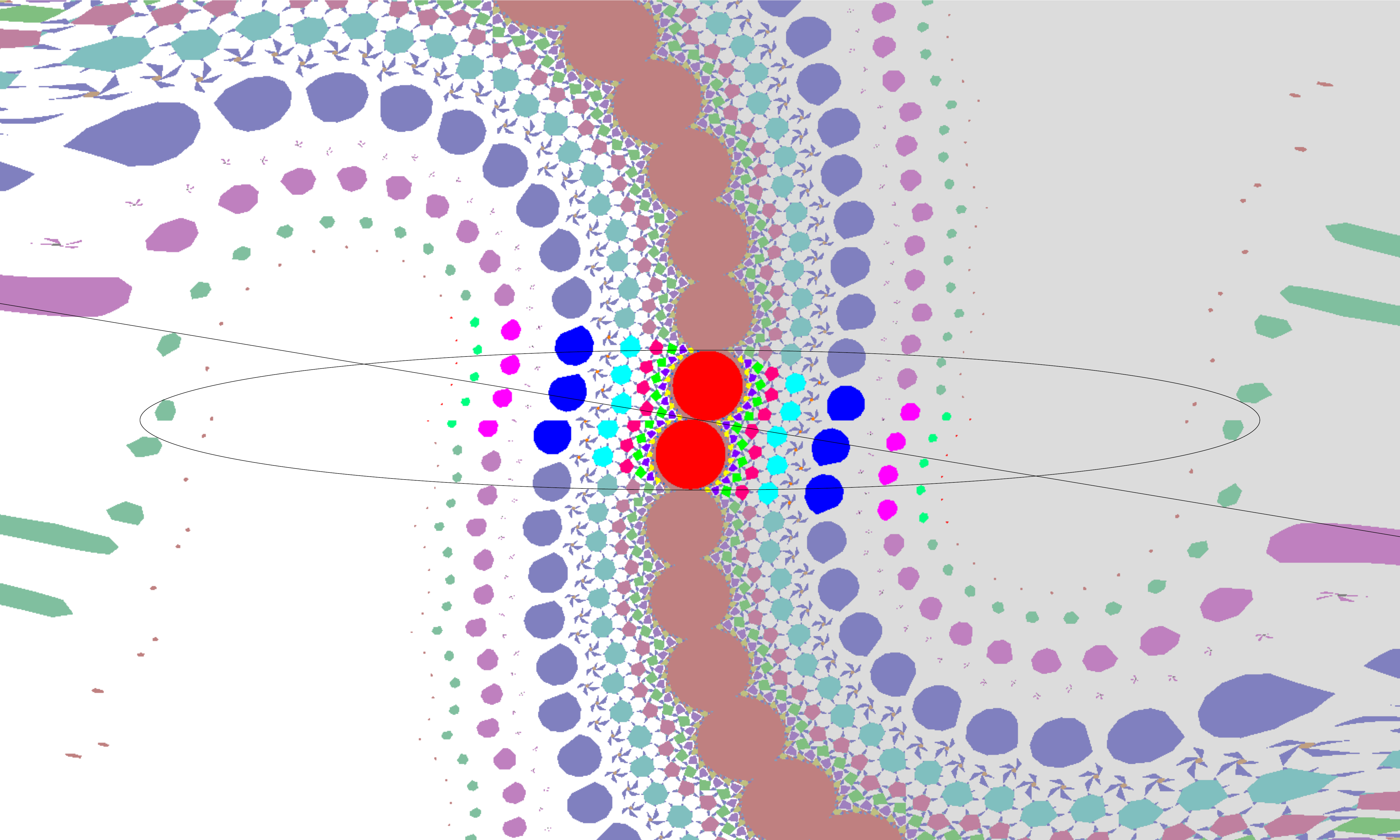}}
			\caption{Domains of attraction for $T_{E_8,L_6}$}
			\label{fig:FullEllipseBasinsFullPage}
		\end{center}
	\end{figure}
	
	\section{Line and $p$-sphere}
	
	Projections onto the 1-sphere can be determined explicitly, so exact analysis is possible. Consequently, much of the behavior is readily determined in particular periodic points with periods higher than 2 are observed. When $p=2$, we recover the circle for which the convergence properties are known, see~\ref{sec:intro}.  Our observations from the case $p>2$ suggest a conjecture: that when the line is not parallel to either of the axes there is at most one pair of periodic points and they are repelling.
	
	\begin{figure}
		\begin{center}
			\includegraphics[width=0.17\textwidth]{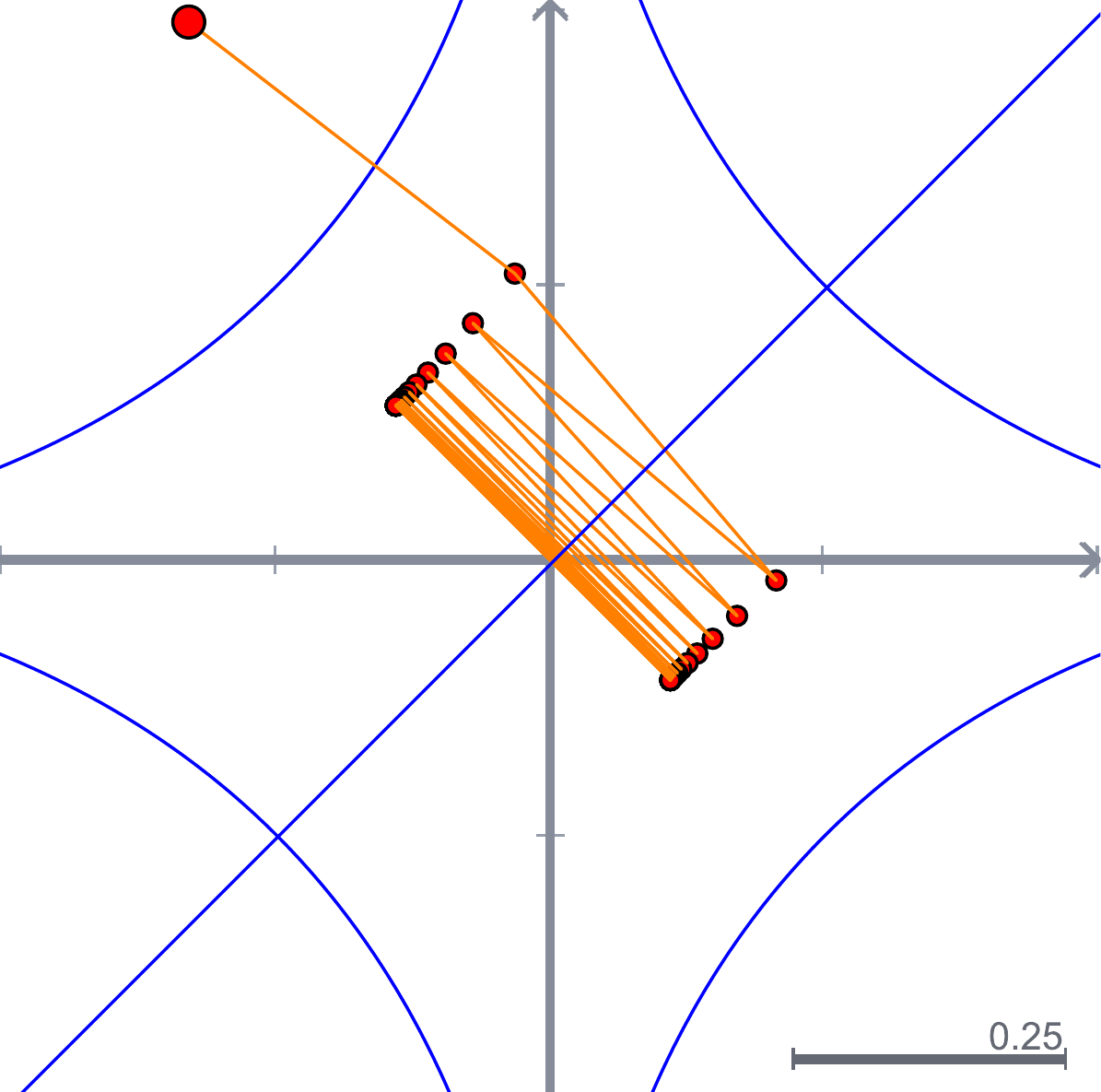}
			\hspace{0.02\textwidth}
			\includegraphics[width=0.17\textwidth]{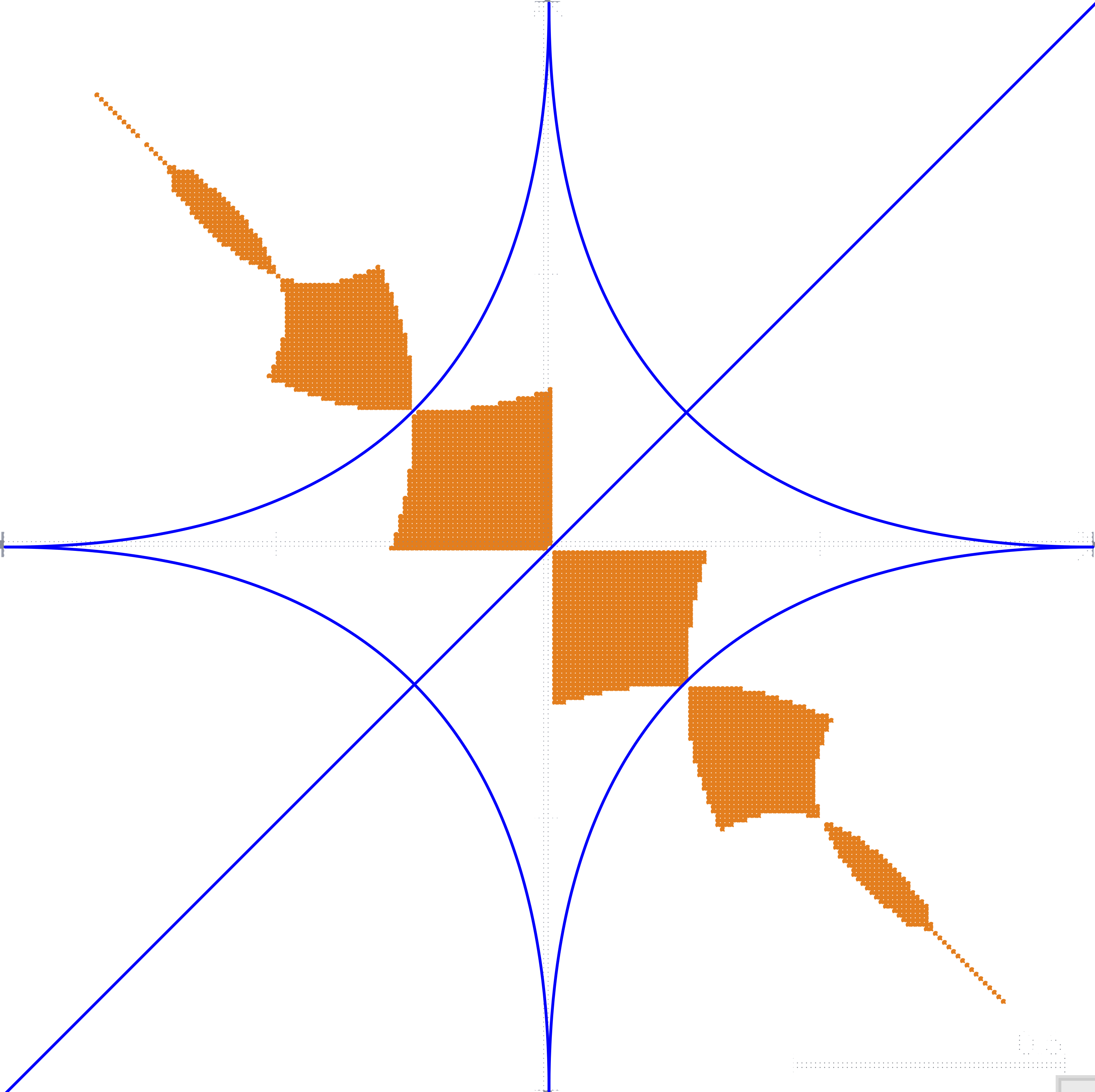}
			\hspace{0.02\textwidth}
			\includegraphics[width=0.17\textwidth]{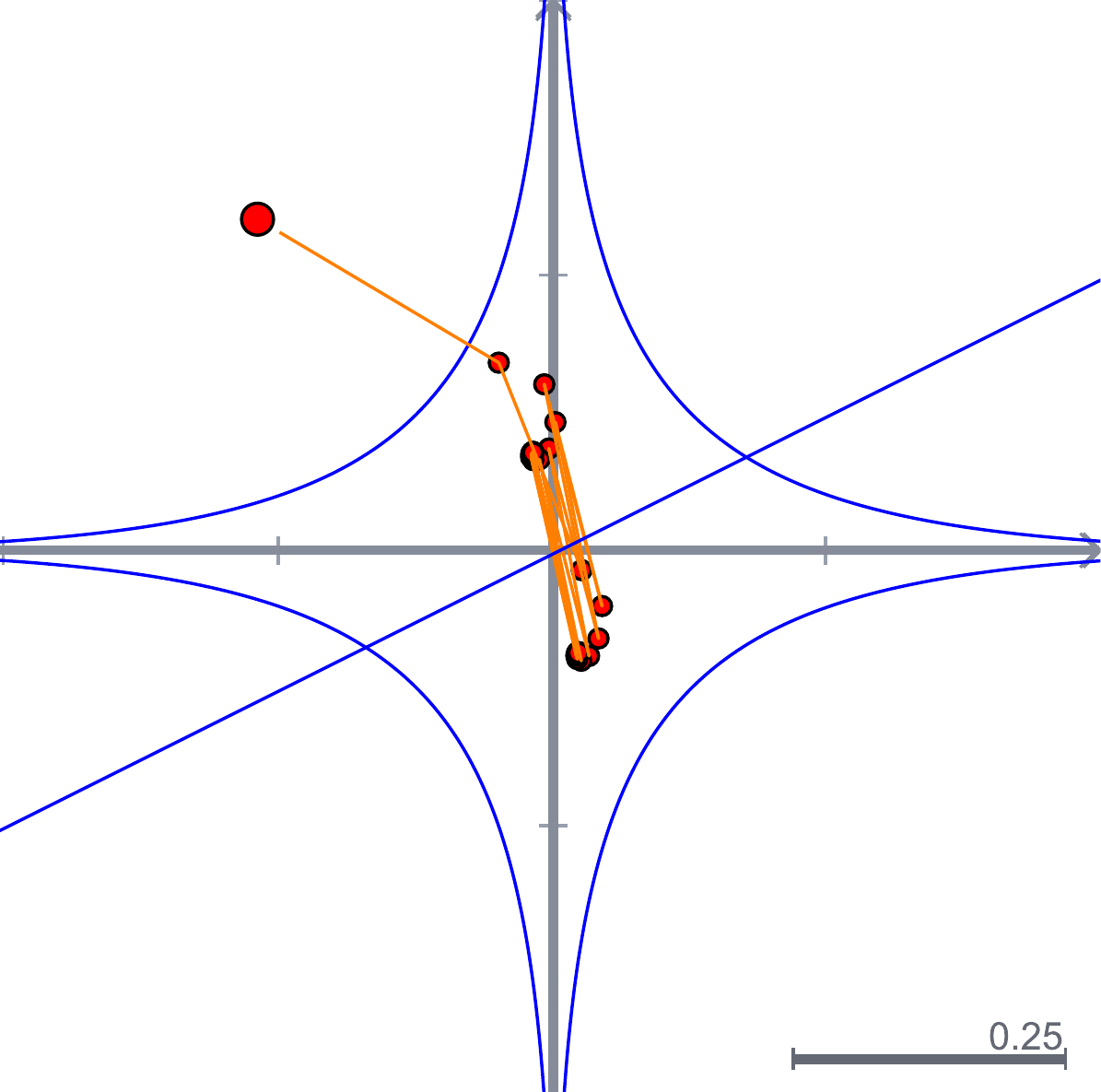}
			\hspace{0.02\textwidth}
			\includegraphics[width=0.17\textwidth]{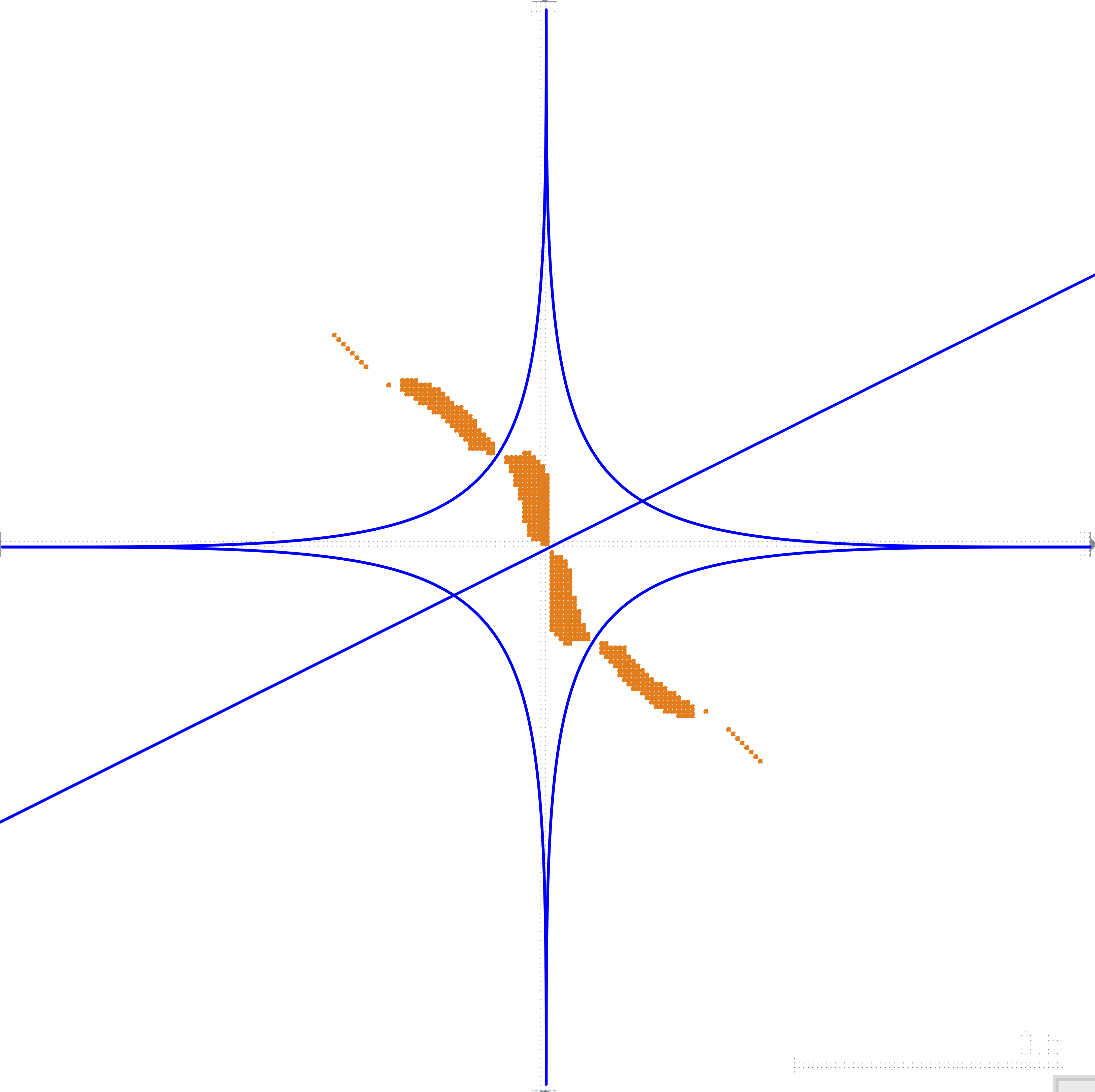}
			\caption{Subsequential convergence to---and attractive domains for---period 2 points. Left: $T_{S_{1/2},L_1}$, right: $T_{S_{1/3},L_{1/2}}$}
			\label{fig:halfandthirdsphere}
		\end{center}
	\end{figure}
	
	For $1/n$-spheres where $n\geq 2$ is a natural number, we see the appearance of period 2 points with attendant local domains of attraction, examples  are shown in Figure~\ref{fig:halfandthirdsphere}. It can be seen---proven easily---that, for the sphere $S_{1/n}$ with line $L_1$, any point $(-t,t)$ or $(t,-t)$ for $t\in (0,\frac{1}{2^n}]$ is a period 2 point. This continuum of period 2 points is analogous to what is observed in the case of a 2-sphere. More interesting is the apparent emergence of attractive domains which have nonzero measure.
	
	These observations already hint at the larger measure and greater complexity of the singular manifold in the case of a line and $p$-sphere, when $p\not=2$, compared to that for a line and 2-sphere. When we rotate the line to, say, $L_{1/2}$ there appear to be only finitely many period 2 points, but they are no longer constrained to lie in an affine submanifold.
	
	\section{A Theoretical Interlude: Local Convergence to a feasible point}
	
	Borwein and Sims \cite{BS} used the Perron theorem on the stability of almost linear difference equations \cite[Corollary 4.7.2]{Lak&DT} to establish local convergence of the Douglas-Rachford algorithm, $x_{n+1} = T_{K,L}(x_{n})$, to an isolated point $f\in L\cap K$ when $L$ is a line and $K$ is the (non convex) unit sphere in $n$-dimensional Euclidean space. We outline a strategy for extending this to the case when $L$ is still a line, but $K$ is a smooth hypersurface ($(n-1)$-manifold). We consider its application when $L$ and $K$ lie in $\mathbb{R}^{2}$; $L$ is the line $\alpha x + \beta y = \gamma$, $K$ is the ellipse $E_b$ as in \eqref{notations}.
	
	The strategy is to show that, in a neighborhoods of the feasible point $f$, the reflection in the supporting hyper-plane $H_{f}$ to $K$ at $f$---as in Figure~\ref{fig:tgtaproxtoellipse}---provides an $o$-order approximation to the reflection in $K$ so that the Perron theorem can be applied to the system of difference equations corresponding to the Douglas-Rachford algorithm. Succinctly, we want
	\begin{equation*}
	R_{K}(p) = R_{H_{f}}(p)\ +\ \Delta, \quad  \textrm{where\ }\|\Delta\| =  o(\|p - f\|) \textrm{\ for\  } p \textrm{\ sufficiently near\ } f.
	\end{equation*}
	For the Euclidean reflection this follows if $\|P_{K}(p) - P_{H_{f}}(p)\| = o(\|p - f\|)$. When this happens we have, for $p$ in a neighborhoods of $f$,
	\begin{align*}
	T_{K,L}(p)\ &=\ \frac{1}{2}\left[p + R_{L}\left(R_{K}(p)\right)\right]\\
	&=\ \frac{1}{2}\left[p + R_{L}\left(R_{H_{f}}(p) + \Delta\right)\right]\\
	&=\ \frac{1}{2}\left[p + R_{L-f}\left(R_{H_{f}}(p) + \Delta - f\right) + f\right]\\
	&=\ \frac{1}{2}\left[p + R_{L-f}\left(\left(R_{H_{f}-f}(p-f) + f\right)+ \Delta - f\right) + f\right]\\
	&=\ \frac{1}{2}\left[p + R_{L-f}\left(R_{H_{f}-f}(p-f)\right)+  R_{L-f}(\Delta) + f\right],\quad \textrm{since\ } R_{L-f} \textrm{\ is linear}\\
	&=\ \frac{1}{2}\left[(p-f) + R_{L-f}\left(R_{H_{f}-f}(p-f)\right)\right]+  \frac{1}{2}R_{L-f}(\Delta) + f\\
	\end{align*}
	Thus we have that
	$$
	T_{K,L}(p)\ =\ f\ +\ T_{(H_{f}-f)(L-f)}(p-f) + \Delta'
	$$
	where $\Delta' = \frac{1}{2}R_{L-f}(\Delta)$ has $\|\Delta'\| = o\|p-f\|$ since $R_{L-f}$ is a bounded linear operator.
	
	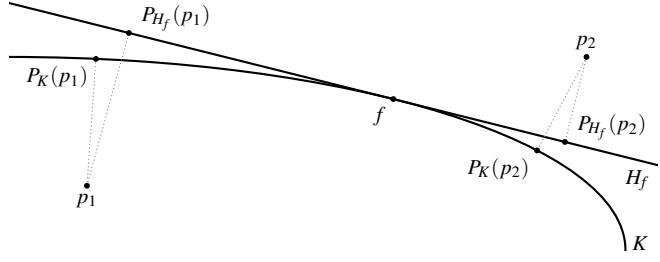
\begin{figure}
		\begin{center}
			\begin{tikzpicture}[x=0.7\textwidth,y=0.7\textwidth] 
			\clip (0,0) rectangle (1,0.39);
			\draw[thick] (0,0) ellipse (0.95 and 0.3);
			
			\coordinate (f) at (.5923153118,.2345494447);
			\coordinate (Hf) at (-.7427399083,.1870469406);
			\coordinate (p1) at (.12,.10);
			\coordinate (PKp1) at (.1338540491,.2970071926);
			\coordinate (PHfp1) at (.1847330394,.3371923972);
			\coordinate (p2) at (.89,.30);
			\coordinate (PKp2) at (.8138347424,.1547603673);
			\coordinate (PHfp2) at (.8567470285,.1679566267);
			
			\draw[thick] (f)+($2*(Hf)$) -- +($-2*(Hf)$);
			
			\draw[gray,densely dotted] (p1) -- (PKp1);
			\draw[gray,densely dotted] (p1) -- (PHfp1);
			\draw[gray,densely dotted] (p2) -- (PKp2);
			\draw[gray,densely dotted] (p2) -- (PHfp2);
			
			\foreach \pt in {(f), (p1), (PKp1), (PHfp1), (p2), (PKp2), (PHfp2)} {
				\fill \pt circle (0.0045);
			}
			
			\node[below left] at (f) {{\(f\)}};
			\node[inner sep=2.5pt, above right] at (0.95,-2.5pt) {\(K\)};
			\node[inner sep=0pt, below right] at (0.95,0.125) {\(H_{\!f}\)};
			
			\node[below] at (p1) {\(p_1\)};
			\node[below left] at (PKp1) {\(P_K\!\left(p_1\right)\)};
			\node[above right] at ($(PHfp1)+(2pt,-2pt)$) {\(P_{H_{\!f}}\!\left(p_1\right)\)};
			
			\node[above] at (p2) {\(p_2\)};
			\node[below left] at (PKp2) {\(P_K\!\left(p_2\right)\)};
			\node[above right] at ($(PHfp2)+(2pt,-2pt)$) {\(P_{H_{\!f}}\!\left(p_2\right)\)};
			\end{tikzpicture}
			\caption{Approximation of $P_{K}$ by $P_{H_{f}}$ near $f$}
			\label{fig:tgtaproxtoellipse}
		\end{center}
	\end{figure}
	
	Thus, by the theorem of Perron (see, \cite{BS} theorem 6.1 or \cite{Lak&DT} Corollary 4.7.2), the system of difference equations corresponding to the Douglas - Rachford algorithm for $K$ and $L$,
	\[ x_{n+1}\ =\ T_{KL}(x_{n}) \]
	is exponentially asymptotically stable at $f$ (in particular $\|x_{n}-f\|\rightarrow 0$ for $x_{0}$ sufficiently near $f$) provided all the eigenvalues of the linear operator $T_{(H_{f}-f)(L-f)}$ have moduli less than one.
	
	\begin{remark} When $M$ is a subspace, the projection $P_{M}$ is linear (as is the case when $M=L-f$) and the Douglas - Rachford operator for $N$ and $M$ becomes
		\begin{align}
		T_{N,M}\ &=\ \frac{1}{2}\left[I + (2P_{M} - I)(2P_{N} - I)\right] \\
		&=\  2P_{M}P_{N}\ -\ P_{M}\ - P_{N}\ +\ I \label{DRopexpanded} \\
		&=\ P_{M}P_{N}\ +\ (I-P_{M})(I-P_{N}).
		\end{align}
		When $N$ is also a subspace (for instance when $N=H_{f}-f$) this may be written as
		\[ T_{NM}\ =\ P_{M}P_{N}\ +\ P_{M^{\perp}}P_{N^{\perp}} \]
		where $^{\perp}$ denotes the orthogonal complement.
		
		As a curiosity, we observe that if in the case of two subspaces we define a \textsl{twisted Douglas-Rachford operator} by $V_{NM}\ :=\ P_{M}P_{N}\ +\ P_{N^{\perp}}P_{M^{\perp}}$, then, since $P_{M}P_{M^{\perp}} = P_{N}P_{N^{\perp}} = 0$, the iterates are $x_{n} = V_{NM}^{n}(x_{0}) = u_{n} + v_{n}$, where $u_{n+1} = P_{M}P_{N}(u_{n})$ and $v_{n+1} = P_{N^{\perp}}P_{M^{\perp}}(v_{n})$. The sequence of twisted Douglas-Rachford approximants is thus the sum of two sequences $(u_{n})$ and $(v_{n})$ resulting from the application of von Neumann's alternating projection algorithm to the pairs of subspaces $M$ and $N$, and $N^{\perp}$ and $M^{\perp}$ respectively. Since the Friedrich angle $\theta$ between  $M$ and $N$ is the same as the angle between $M^{\perp}$ and $N^{\perp}$, the twisted Douglas-Rachford algorithm converges with the same rate as the von Neumann algorithm; namely at a linear rate proportional to $\cos^{2}\theta$ \cite{D}, the same as the rate exhibited by the standard Douglas-Rachford algorithm \cite{BCNPW}.
	\end{remark}
	
	Now we consider the special case of an ellipse and a line. Without loss of generality we consider the ellipse $E_b$ as in \eqref{notations}, and the line $L:\ \alpha x + \beta y = \gamma$, where $ b \geq 1 $, $\beta \geq 0$ and, to ensure the existence of $f = (x_{0},y_{0})\in L\cap E\cap |\mathbb{R}|^{2}$, either $\alpha\leq \gamma\leq \beta b$ or $\beta b <\gamma$ and $\alpha \geq \sqrt{\gamma^{2}-\beta^{2} b^{2}}$.
	
	Following the strategy outlined above leads us to consider the\\
	\emph{Eigenvalues of $T$ for two lines through the origin};
	$$
	L_{1}:  \alpha x + \beta y = 0\quad \textrm{and}\quad L_{2}: Ax + By = 0
	$$
	where, in our context, the latter line, being parallel to the tangent to $E$ at $f$, has $\displaystyle A = x_0$ and $B = {y_{0}}/{b^2}$. It is readily verified that the orthogonal projection onto $L_{1}$ has the matrix
	\begin{equation}\label{projmatrix}
	[P_{L_{1}}] = \frac{1}{\alpha^{2}+\beta^{2}} \left( \begin{array}{cc} \beta^{2} & -\alpha\beta\\ -\alpha\beta & \alpha^{2}\end{array}\right)
	\end{equation}
	with a matching expression for $[P_{L_{2}}]$. Substituting these expressions into $T_{L_{2},L_{1}} = 2P_{L_{1}}P_{L_{2}} - P_{L_{1}} -P_{L_{2}} + I$ yields
	$$
	[T_{L_{1},L_{2}}] = \frac{\psi}{\Delta} \left(\begin{array}{cc} \psi& \omega \\ -\omega & \psi\end{array}\right),
	$$
	where
	$\psi = \alpha A + \beta B$, $\omega =  \alpha B - \beta A$ and $\Delta = (\alpha^{2} + \beta^{2})(A^{2} + B^{2})$,
	which has eigenvalues $\frac{\psi}{\Delta} \left( \psi \pm i\omega\right)$ with modulus squared equal to
	\begin{align*}
	\frac{\psi^{2}}{\Delta^{2}}\left(\psi^{2} + \omega^{2}\right) &= \frac{(\alpha A+ \beta B)^{2}\left((\alpha A+ \beta B)^{2} + (\alpha B - \beta A)^{2}\right)}{\left(\alpha^{2} + \beta^{2}\right)^{2}\left(A^{2} + B^{2}\right)^{2}}\\
	&= \frac{(\alpha A+ \beta B)^{2}}{\left(\alpha^{2} + \beta^{2}\right)\left(A^{2} + B^{2}\right)} < 1
	\end{align*}
	Thus, as expected, for any two lines intersecting in a single point the Douglas - Rachford algorithm with any starting point spirals exponentially to their common point.
	
	Therefore the Douglas-Rachford
	algorithm for a line and an ellipse $E$ is locally convergent at each of the feasible points $f$ provided
	$$
	\|P_{E}(p) - P_{H_{f}}(p)\| = o(\|p - f\|),
	$$
	for all $p$ in some neighborhood of $f$.
	
	To see this we follow an argument suggested by Asen Dontchev \cite{Don}. While we present the argument in the particular case of $E_b=\{x|\varphi_b(x)-1=0\}$, the astute reader will observe that it applies to any smooth hypersurface $K:=\{ g(x) = 0\}$ at any point $f = (x_{0},y_{0})\in K$ at which the gradient $\nabla g$ is non-singular (true for the ellipse as $\nabla g(x) = (2x, 2y/b^2)$) and so applies to $p$-spheres except near the extreme points of the sphere when  $0 < p \leq 1$.
	
	We begin by noting that for the supporting hyperplane (tangent) to $E$ at $f$
	$$
	P_{H_{f}}(p) = f +  P_{H_{f}-f}(p-f)
	$$
	where
	$$
	\left[P_{H_{f}-f}\right]= \frac{b^4}{b^4x_{0}^{2}+y_{0}^{2}}\left(\begin{array}{cc} y_{0}^{2}/b^4& -x_{0}y_{0}/b^2\\ -x_{0}y_{0}/b^2 & x_{0}^{2}\end{array}\right).
	$$
	Next we observe that the nearest point projection  $(u(p),v(p)) = P_{E}(p)$ at $p = (\zeta,\eta)$ is the solution of
	\begin{align*}
	\text{minimize:\ } &\frac{1}{2}\|P_{E}(p) - p\|^{2} = \frac{1}{2}\left((u-\zeta)^{2} + (v-\eta)^{2}\right)\\
	\text{subject to:\ } &g\left(P_{E}(p)\right) =  u^2 + \left(\frac{v}{b}\right)^{2} - 1 = 0,
	\end{align*}
	which, since $\nabla g\left(P_{E}(p)\right)$ is non-singular, is characterized via the method of Lagrange multipliers by $\nabla\frac{1}{2}\left((u-\zeta)^2 + (v-\eta)^2\right) + \lambda\nabla g\left(P_{E}(p)\right) = 0$ together with $g\left(P_{E}(p)\right)=0$, that is,
	\begin{align*}
	f_{1}:&\ u-\zeta  + 2\lambda u = 0\\
	f_{2}:&\ v-\eta + 2\lambda \frac{v}{b^2} = 0\\
	g:&\  u^2 + \left(\frac{v}{b}\right)^{2} - 1 = 0.
	\end{align*}
	As an aside, this yields the implicit specification
	$$
	P_{E}(\zeta,\eta) = \left(\frac{\zeta}{1+2\lambda}, \frac{b^2\eta}{b^2+2\lambda}\right),\quad \textrm{where\ } \frac{\zeta^{2}}{\left(1+2\lambda \right)^{2}}\ +\ \frac{b^2\eta^{2}}{\left(b^2+2\lambda\right)^{2}}=1.
	$$
	In order to apply the implicit function theorem to ensure that $u,v$ and $\lambda$ are differentiable functions of $\zeta$ and $\eta$ in a neighborhoods of $f$ we require the Jacobian of the above system of equations with respect to the dependent variables, $u,v$ and $\lambda$ at $f$,  $J(f)$, to be non-singular. Since $P_{E}(f) = f$ we see from the first (and second) equation that for $p=f$ the corresponding Lagrange multiplier is necessarily $0$. Thus,
	$$
	J(f)\ =\  \left. \frac{\partial(f_{1},f_{2},g)}{\partial(u,v,\lambda)}\right|_{(x_{0},y_{0},0)}\ =\ \left(\begin{array}{ccc} 1 & 0 & 2x_{0}\\ 0 & 1 & 2y_{0}/b^2\\ 2x_{0} & 2y_{0}/b^2 & 0\end{array}\right),
	$$
	more generally $\displaystyle J(f)\ =\ \left(\begin{array}{cc} I & \nabla g(f)^{T} \\ \nabla g(f) & 0 \\ \end{array}\right)$,
	which is indeed non-singular, and in our case
	$$
	J(f)^{-1} = \frac{b^{4}}{b^4x_{0}^{2}+y_{0}^{2}}\left(\begin{array}{ccc} y_{0}^2/b^4 & -x_{0}y_{0}/b^{2} & {x_{0}/2}\\  -x_{0}y_{0}/b^{2} & x_{0}^{2} & y_{0}/2b^2\\ {x_{0}}/2 & y_{0}/2b^2 & -1/4\end{array}\right).
	$$
	Thus the implicit function theorem applies, yielding
	$$
	\left(\begin{array}{c} \left[P^{\prime}_{E}(f)\right]\\ \left[\lambda^{\prime}(f)\right]\end{array}\right) = \left. \frac{\partial(u,v,\lambda)}{\partial(\zeta,\eta)}\right|_{(x_{0},y_{0}))}\ =\
	J(f)^{-1}\frac{\partial(f_{1},f_{2},g)}{\partial(\zeta,\eta)}\ =\
	J(f)^{-1}\left(\begin{array}{c} I\\ 0\end{array}\right),
	$$
	whence $\displaystyle \left[P^{\prime}_{E}(f)\right] = \frac{b^{4}}{b^4x_{0}^{2}+y_{0}^{2}}\left(\begin{array}{cc} y_{0}^{2}/b^4& -{x_{0}y_{0}}/{b^{2}}\\-{x_{0}y_{0}}/{b^{2}}& {x_{0}^{2}}\end{array}\right)$ which we recognize as $\left[P^{\prime}_{H_{f}}(f)\right]$ and we are able to conclude that near $f$
	$$
	\|P_{E}(p) - P_{H_{f}}(p)\| = \|f+P^{\prime}_{E}(f)(p-f)+\Delta\ -\ \left(f\ + \  P_{H_{f}-f}(p-f)\right)\|\ =\ \|\Delta\|
	$$
	where $\|\Delta\|\ = o(\|p-f\|)$ as required.
	
	Thus, for an ellipse and generically intersecting line, the Douglas-Rachford algorithm is locally convergent at each of the feasible points.
	
	\section{Important Lessons about Global Behavior}
	
	What we have observed in our computer-assisted study of these two simple cases of a line together with an ellipse or a $p$-sphere is remarkably informative: it suggests likely explanations for the behavior of the algorithm both for feasible and infeasible cases. We consider feasible cases first.
	
	\begin{figure}
		\begin{center}
			\includegraphics[width=0.30\textwidth]{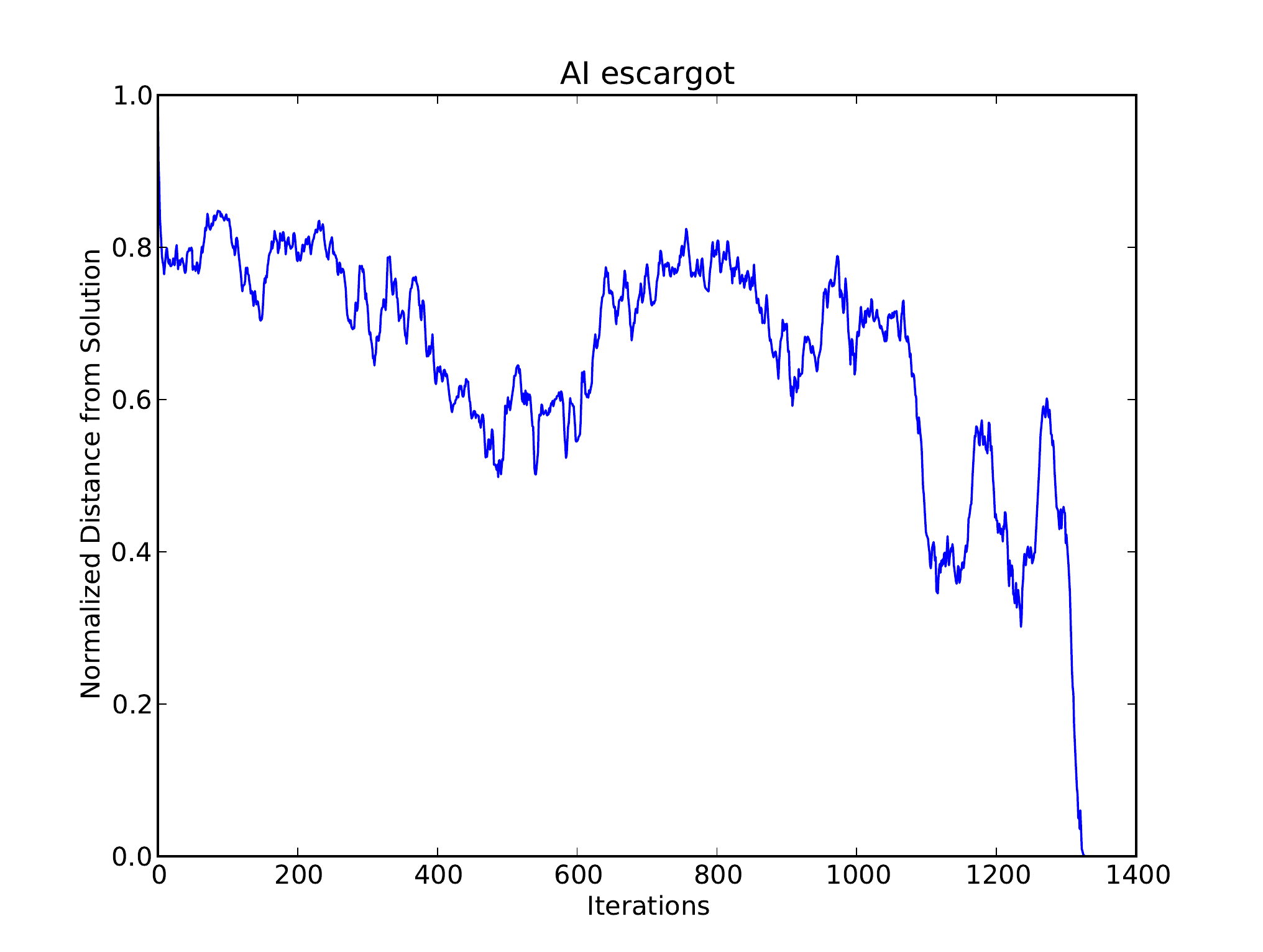}
			\hspace{0.1\textwidth}
			\includegraphics[width=0.30\textwidth]{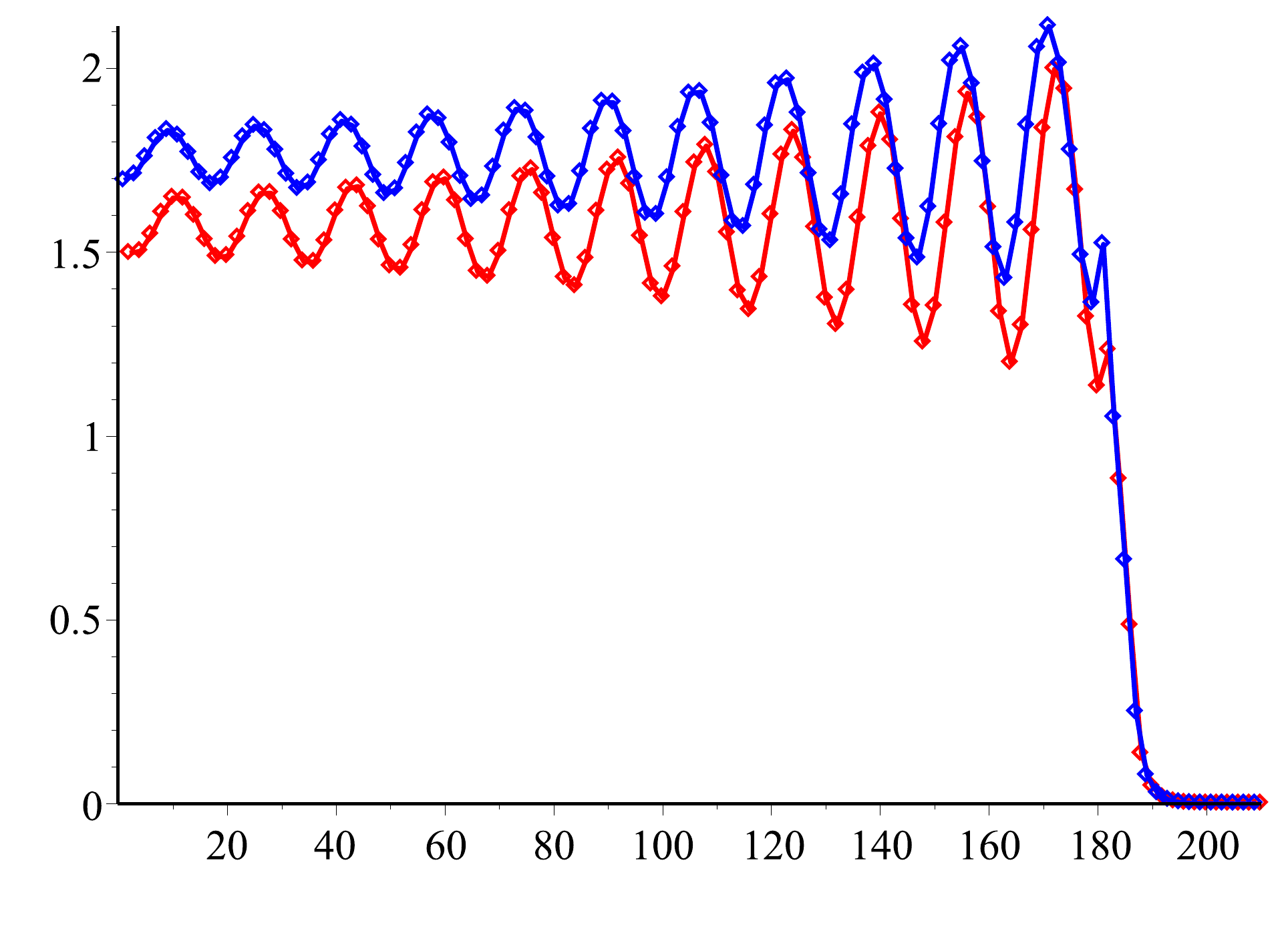}
			\caption{Distance from iterates to solution. Left: for a sudoku puzzle \cite{ABT1}. Right: for $T_{E_2,L_2}$}
			\label{fig:sudoku1}
		\end{center}
	\end{figure}
	
	\subsection{The Feasible Case}
	
	Arag\'{o}n Artacho, Borwein, and Tam experimented with using the Douglas-Rachford method to solve Sudoku puzzles  \cite{ABT1}. The left hand images of Figures~\ref{fig:sudoku1} and \ref{fig:sudoku2}, illustrate the distance to the solution by iterations of Douglas-Rachford for two different sudoku puzzles. First consider Figure~\ref{fig:sudoku1}. On the left, we see the algorithm struggle for a long period of time before finally converging. Compare this to the image on the right: for $T_{E_2,L_2}$ with 210 iterates, distance of each iterate---to the particular feasible point the sequence converges to---is plotted. The subsequences $x_{2k}$ and $x_{2k-1}$ are colored light and dark grey respectively. They correspond respectively to iterates landing in the domain of attraction for the left and right feasible points, see~\ref{fig:sourcebasinspirals}. Without the geometric intuition gleaned from Figure~\ref{fig:sourcebasinspirals}, we would not know to color these two subsequences distinctly, and the error plot would not reflect the behavior as clearly. Once the iterates have finally climbed free of the influence of the repelling period 2 points, we see a sudden rapid convergence to the relevant feasible point.
	
	\begin{figure}
		\begin{center}
			\includegraphics[width=0.3\textwidth]{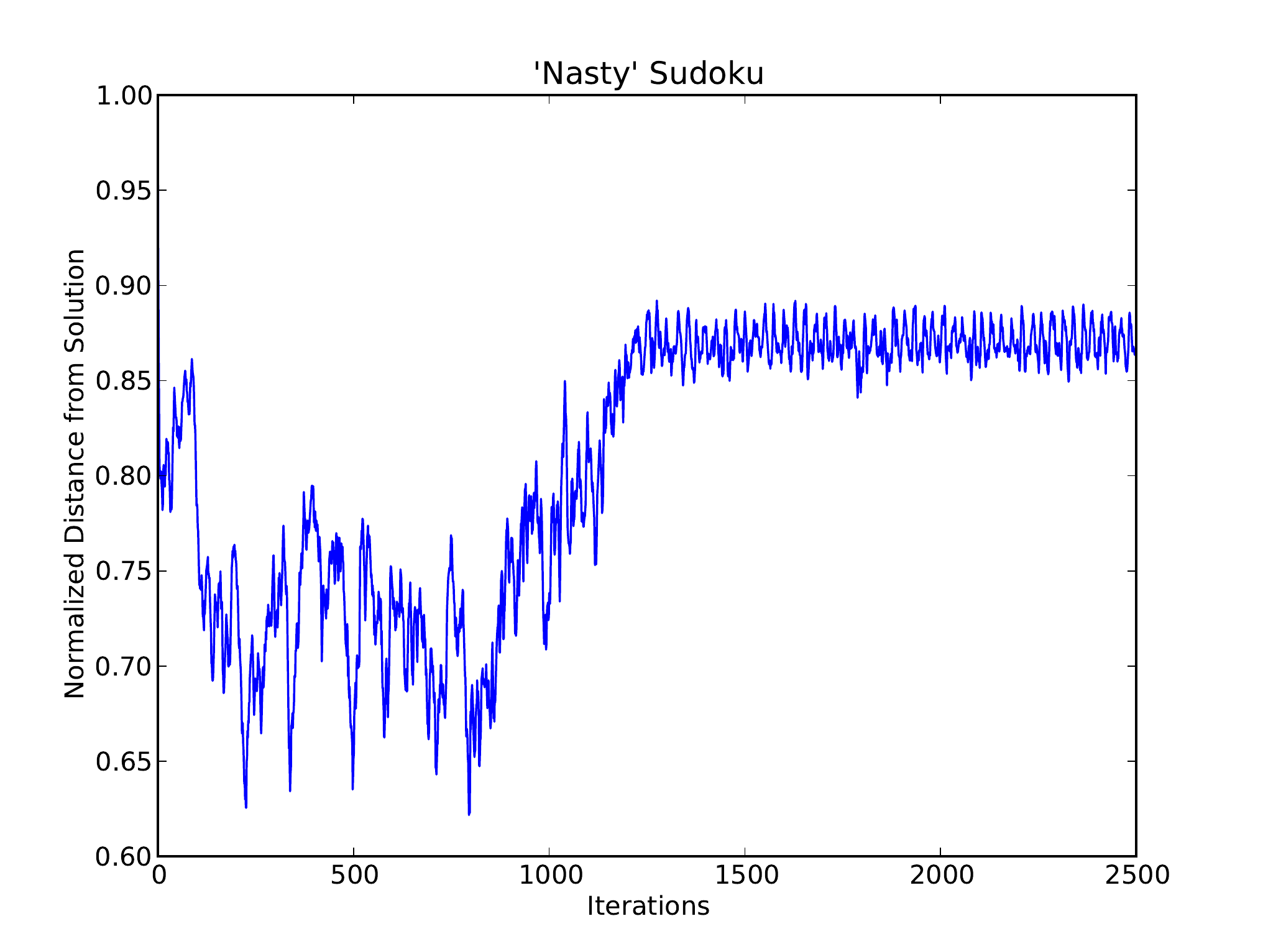}
			\hspace{0.1\textwidth}
			\includegraphics[width=0.3\textwidth]{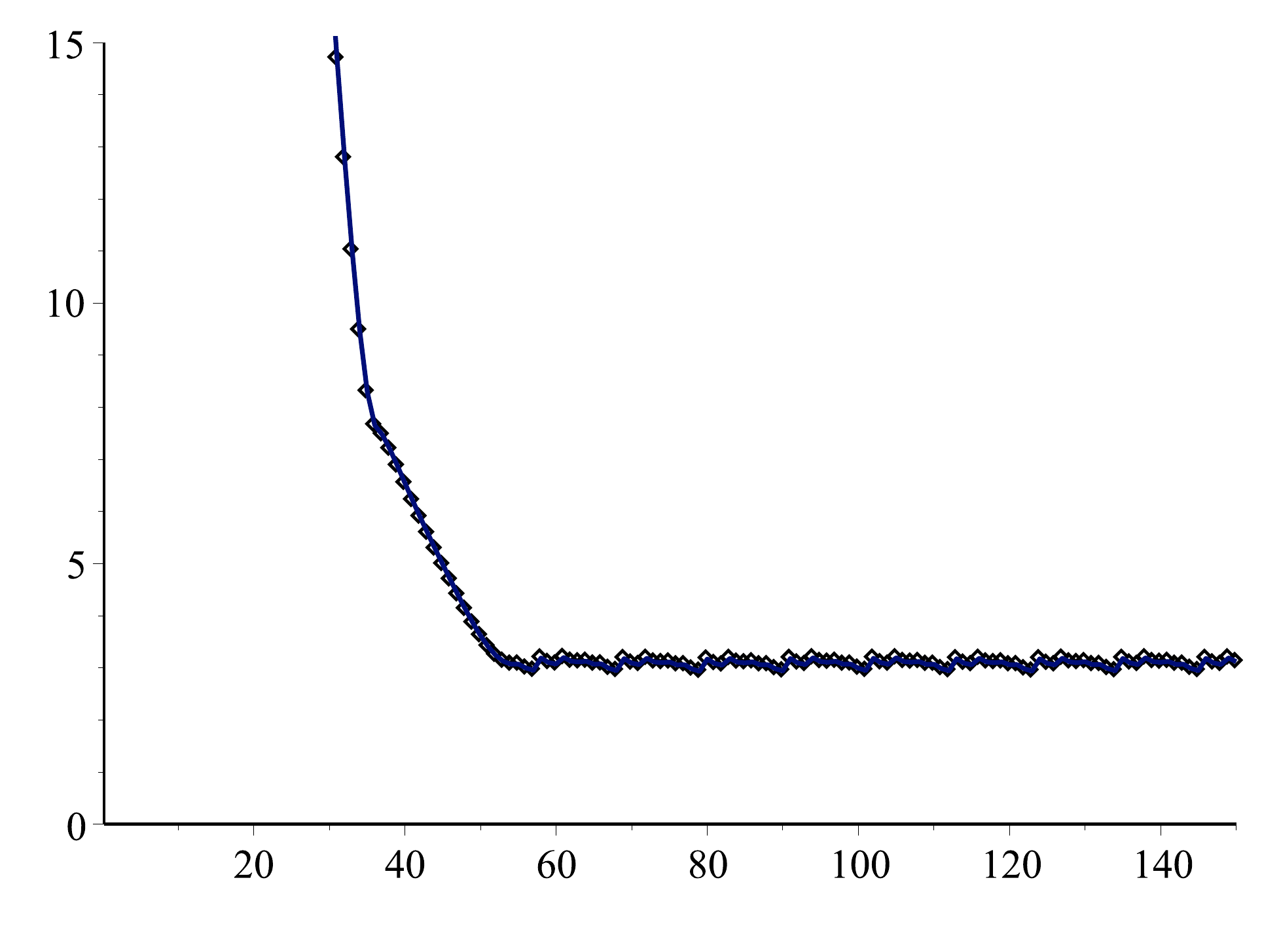}
			\caption{Distance from iterates to a solution. Left: a sudoku puzzle \cite{BT}. Right: $T_{E_{14},L_9}$}
			\label{fig:sudoku2}
		\end{center}
	\end{figure}
	
	Now consider Figure~\ref{fig:sudoku2}. At left, we see distance from the solution of the iterates for a different sudoku puzzle. This time the error stabilizes after a time without any indication of impending convergence. At right, we see the iterates for $T_{E_{14},L_9}$. The iterates approach the ellipse before being pulled into the attractive domain for some period $11$ points, preventing convergence.
	
	\begin{figure}
		\begin{center}
			\includegraphics[width=0.3\textwidth]{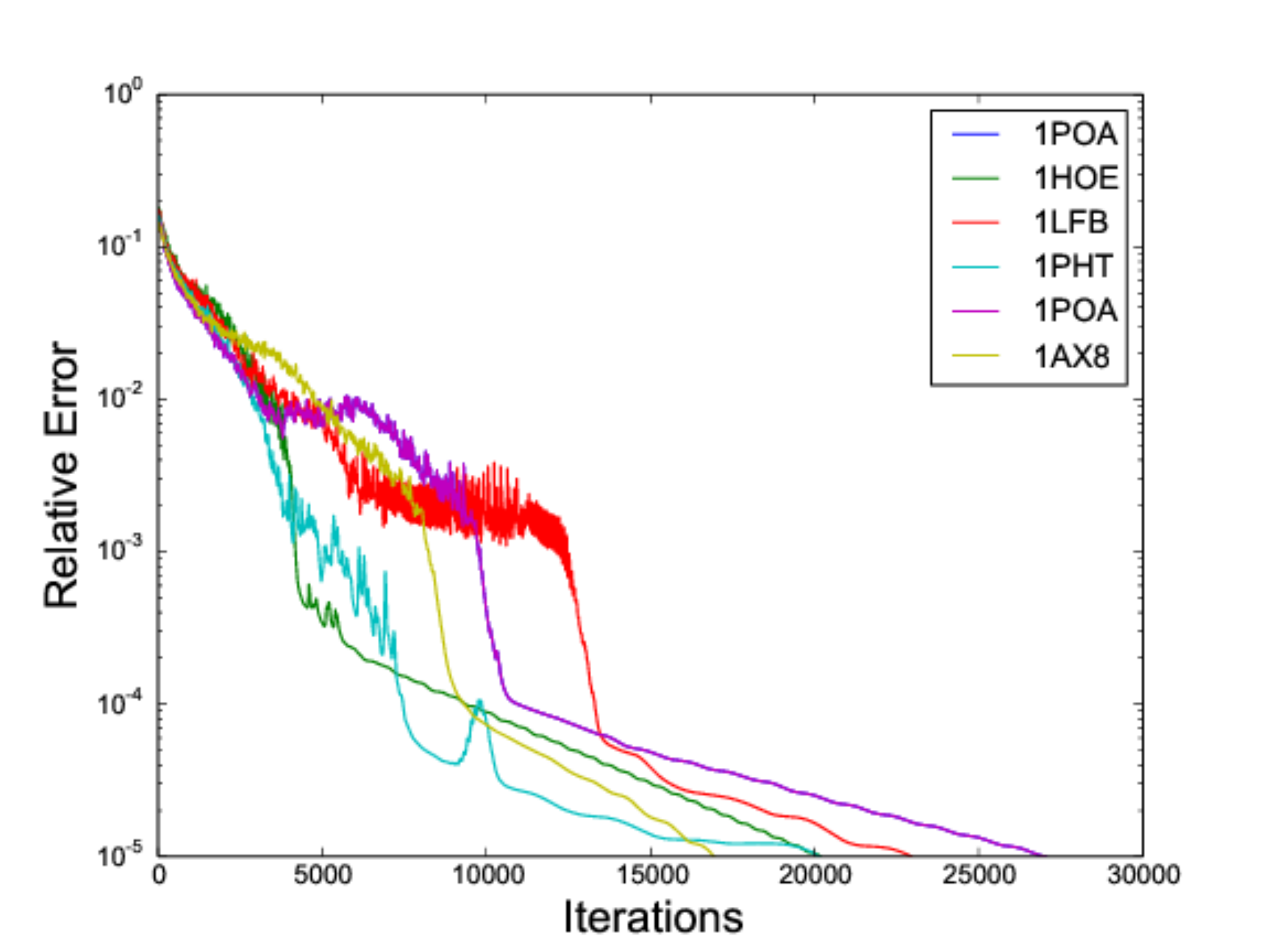}
			\hspace{0.1\textwidth}
			\includegraphics[width=0.3\textwidth]{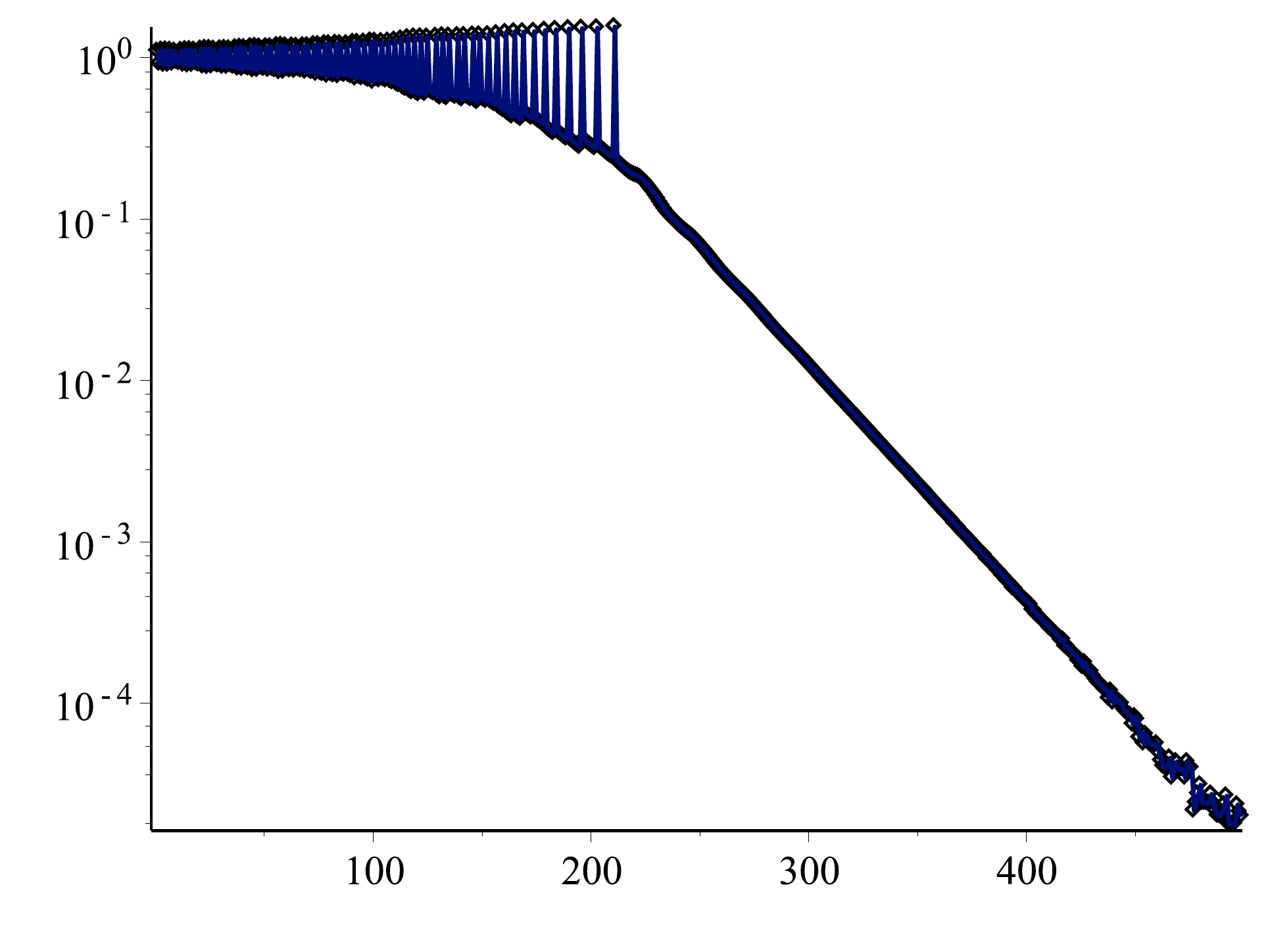}
			\caption{Distance of iterates from solution (scale logarithmic). Left: for five proteins \cite{BT}, Right: for the iterates of $T_{E_8,L_6}$ pictured in Figure~\ref{fig:tracingbasins}}
			\label{fig:proteinconvergence}
		\end{center}
	\end{figure}
	
	Experiments have also been conducted using the Douglas-Rachford method to solve matrix completion problems associated with incomplete euclidean distance matrices for protein mapping \cite{ABT2,BB,BT}. Consider Figure~\ref{fig:proteinconvergence}. The left image shows the relative error of iterates when solving the Euclidean distance matrices for various proteins. The right image shows the relative error for the iterates of $T_{E_8,L_6}$  when the sequence of iterates is started near to domains of attraction.
	
	\begin{figure}
		\begin{center}
			\includegraphics[angle=90,width=0.7\textwidth]{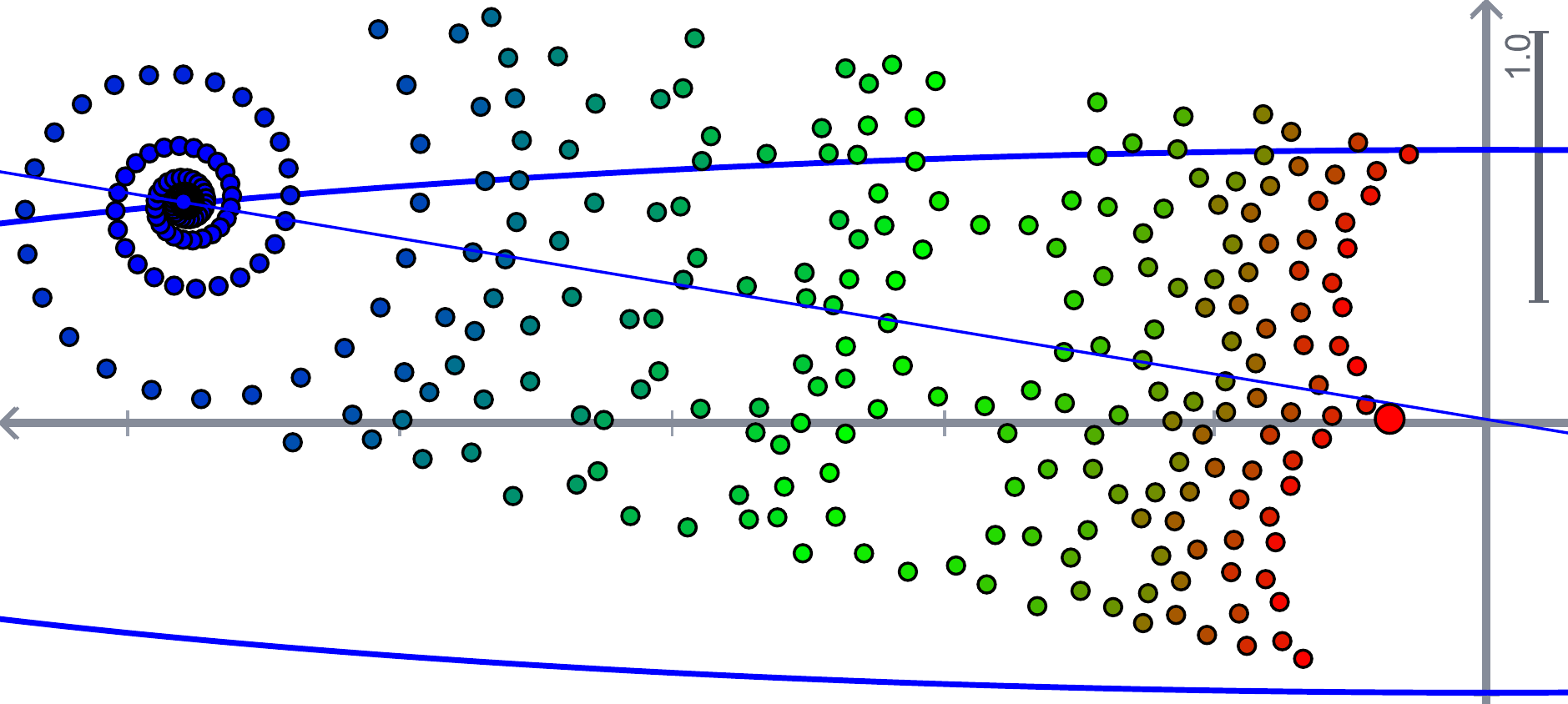}
			\caption{A convergent sequence of iterates of $T_{E_8,L_6}$ traces the outline of the domains}
			\label{fig:tracingbasins}
		\end{center}
	\end{figure}
	
	The exact iterates used to generate this data are shown in Figure~\ref{fig:tracingbasins}; they appear to trace out the shapes of the attractive domains for periodic points, narrowly avoiding them on their way to eventual convergence. Points started relatively close to domains of attraction for periodic points (as in the right hand side of Figure~\ref{fig:sudoku1}) appear to take  longer to converge than those started elsewhere.
	
	While we cannot say with any real certainty that the behavior when solving these Euclidean distance matrices is analogous to iterates climbing away from repelling points or dodging and weaving between a nest of attractive domains, it is surprising that a system as simple as that of a line and ellipse can create behavior so similar to that observed in far more complicated scenarios.
	
	\subsection{Infeasible Cases}
	
	For the infeasible cases of line and the p-sphere or ellipse, we observed that the iterates of the Douglas-Rachford algorithm appear to walk to infinity with a roughly linear step size. In both infeasible cases, it is possible to strictly separate the two sets in question. This led to the following theorem.
	
	\begin{theorem}\label{convexseparation}
		Let $x_{n+1} = T_{A,B}(x_n)$ and suppose one of the following:
		\begin{enumerate}
			\item $A$ is compact and ${\rm co}(A)$ and ${\rm cl}({\rm co}(B))$ are disjoint.
			\item $B$ is compact and ${\rm cl}({\rm co}(A))$ and ${\rm co}(B)$ are disjoint.
		\end{enumerate}
		Then $\|x_n\|$ tends linearly to $\infty$ with a step size of at least $d(A,B)$.
	\end{theorem}
	\begin{proof}
		we suppose that (1) applies. The proof when (2) applies is obtained by interchanging the roles of $A$ and $B$. Now, we can strictly separate ${\rm co}(A)$ and ${\rm cl}({\rm co}(B))$ with a hyperplane $\mathbb{H}=f^{-1}(\alpha)$ for some linear functional $f$. See \cite[Theorem 1.7]{Brezis} for details. By translation invariance, let $\alpha=0$. Then $\mathbb{H}$ is a subspace, so we can uniquely describe any $x \in X$ as $x =h_{x} + y_{x}$ where $h_{x} \in \mathbb{H}$ and $y_{x} \in \mathbb{H}^\perp$. We can impose several additional properties on $f$:
		\begin{align}
		|f(x)| = |f(h_x+y_x)| &= \|y_x\|_X \quad \text{for all } x\in X \label{fvalues2}\\
		f(x) & < 0 \quad \text{for all } x\in A \label{favalues2}\\
		f(x) & > 0 \quad \text{for all } x\in B \label{fbvalues2}.
		\end{align}
		Equations~\eqref{fvalues2}, \eqref{favalues2}, and \eqref{fbvalues2} imply that $f(x) \geq d(A,B)$ for all $x\in B-A$. Now
		\begin{align*}
		x_{n+1}-x_n = \frac{R_B(R_A(x_n))+x_n}{2}-x_n &= \frac{R_B(P_A(x_n))-x_n}{2}\\
		&=\frac{2P_B(R_A(x_{n}))-R_A(x_{n})-x_n}{2}\\
		&=\frac{2P_B(R_A(x_{n}))-(2P_A(x_{n})-x_n)-x_n}{2}\\
		&= P_B(R_A(x_n))-P_A(x_n) \in B-A.
		\end{align*}
		Now $x_{n+1}-x_n \in B-A$ implies that
		$f(x_{n+1}-x_n) \geq d(A,B)$.
		Thus we have that, for all $n$, $f(x_{n+1}) \geq d(A,B) + f(x_n)$. This shows that $\|x_n\|_n \rightarrow \infty$ with a linear step size of at least $d(A,B)$.
		\qed\end{proof}
	From this result we obtain the following corollary, the computer-assisted discovery of which motivated the pursuit of the more general Theorem.
	\begin{corollary}
		In the infeasible case of a line $L$ with an ellipse $E$ or a p-sphere $S$, we have that $\|x_n\| \rightarrow \infty$ with a linear step size greater than or equal to $d(E,L)$ or $d(S,L)$ respectively.
	\end{corollary}
	
	Using these results and the following remark, we can naturally extend some of the convex theory to the non-convex case.
	\begin{remark}\label{nearestpointconvergence}
		As a consequence of \cite[Theorem 4.5]{ShadowSequence} we have, for convex subsets $U,\, V$ of a Hilbert space $H$, with $U\cap(v+V) \ne \emptyset$, where $v=P_{{\rm cl}({\rm ran}({\rm Id}-T_{U,V}))}(0)$ is the \textsl{minimal displacement vector}, and for $x\in X$, that $(P_UT_{U,V}^n x)_n$ converges weakly to a point in $U\cap(v+V)$.
	\end{remark}
	We extend this result in our context using Theorem~\ref{convexseparation}.
	\begin{theorem}
		Let $A,\, B$ be the respective boundaries of two disjoint closed convex sets $U,V$ in $H$, one of which is compact (so $A$, $B$ satisfy the requirements of Theorem~\ref{convexseparation}). Let $ x_{n+1} = T_{A,B}(x_n)$ and $v:=P_{{\rm cl}({\rm ran}({\rm Id}-T_{U,V}))}(0)$, the uniquely defined element in ${\rm cl}({\rm ran}({\rm Id}-T_{U,V}))$ such that $\|v\|=\underset{x\in X}{\inf}\|x-T_{U,V}x\|$, and let $v'=P_{{\rm cl}({\rm ran}({\rm Id}-T_{A,B}))}(0)$, the uniquely defined element in ${\rm cl}({\rm ran}({\rm Id}-T_{A,B}))$ such that $\|v'\|=\underset{x\in X}{\inf}\|x-T_{A,B}x\|$. Then, for $x\in X$, we have that $(P_AT_{A,B}^n x)_n$ converges weakly to a point in $A\cap(v'+B)$.
	\end{theorem}
	
	\begin{proof}
		By the closure of both sets and compactness of one of the sets, we have attainment of elements which minimize the distance. Thus $A\cap(v'+B) = U\cap(v+V) \ne \emptyset$.
		
		Let $f$ be defined as in Theorem~\ref{convexseparation} so that $f(u)< 0$ for all $u \in U$. Then the sequence $f(x_n)$ is monotone increasing and so there exists some $n' \in \mathbb{N}$ such that $f(x_n) \geq 0$ for all $n\geq n'$. Suppose $n\geq n'$. Then we have that $x_n \notin U$. Thus $
		P_A(x_n) = P_{U}(x_n)$, and so $R_A(x_n) = R_U(x_n)$. We also have that $R_A(x_n)\notin V$. Thus $P_B(P_A(x_n)) = P_{V}(R_A(x_n))$, and so $R_B(R_A(x_n)) = R_{V}(R_A(x_n))$. Thus we have that $T_{A,B}(x_n)=T_{U,V}(x_n)$, and so
		\begin{equation}
		(P_AT_{A,B}^{n'+n} x)_n=(P_AT_{A,B}^n x_{n'})_n=(P_AT_{U,V}^n x_{n'})_n
		\end{equation}
		for all $n \in \mathbb{N}$. We have from Remark~\ref{nearestpointconvergence} that the sequence on the right converges to a point $y$ in $U\cap(v+V)$.
		\qed\end{proof}

	\subsection{Closing Remarks}
	
	Given that we are investigating the Douglas-Rachford method applied to some of the simplest possible instances of a non-convex set, the emergence of such complexity is extraordinary. More interesting from a technical standpoint is the similarity with which the behavior in such simple situations appears to resemble some of what is observed for much larger and more complicated ones.
	
	We also hope that we have succeeded in making a case for computer-assisted discovery, visualization, and verification. \emph{"A heavy warning used to be given [by lecturers] that pictures are not rigorous; this has never had its bluff called and has permanently frightened its victims into playing for safety. Some pictures, of course, are not rigorous, but I should say most are (and I use them whenever possible myself)."}---J. E. Littlewood, 1885-1977	[from Littlewood's Miscellany, p. 35 in the 1953 edition], said long before the current powerful array of graphic, visualization and geometric tools were available.


\begin{thebibliography}{99}
		
		\bibitem{ABT1}F.J. Arag\'{o}n Artacho, J.M. Borwein, and M.K. Tam. ``Recent results on Douglas-Rachford methods for combinatorial optimization problems,'' \emph{Journal of Optimization Theory and Applications}, \textbf{163} (2014) 1-30.
		
		\bibitem{ABT2}F.J. Arag\'{o}n Artacho, J.M. Borwein, M.K. Tam. ``Douglas-Rachford feasibility methods for matrix completion problems,'' \emph{ANZIAM Journal}. \textbf{55(4)} (2014) 299-326.
		
		\bibitem{AB}F.J. Arag\'{o}n Artacho, J.M. Borwein. ``Global convergence of a non-convex Douglas-Rachford iteration,'' \emph{Journal of Global Optimization} 57 (2013), Issue 3, 753–769.
		
		\bibitem{BB}D.H. Bailey and J.M. Borwein, ``Experimental computation as an ontological game changer: The impact of modern mathematical computation tools on the ontology of mathematics.'' In \emph{Mathematics, Substance and Surmise: Views on the Meaning and Ontology of Mathematics}, Springer, 25-67.
		
		\bibitem{BCNPW} H.H. Bauschke, J.Y. Bello Cruz, T.T.A. Nghia, H.M. Phan, and X. Wang, ``The rate of linear convergence of the Douglas-Rachford algorithm for subspaces is the cosine of the Friedrichs angle,'' \emph{Journal of Approximation,} 185, pp. 63-79, 2014.
		
		Theory 185, pp. 63-79, 2014.
		
		\bibitem{BC} H.H.\ Bauschke and P.L.\ Combettes, \emph{Convex Analysis and Monotone Operator Theory in Hilbert Spaces}, Springer, 2011.
		
		\bibitem{BCL} H.H. Bauschke, P.L. Combettes, and D.R. Luke: \lq\lq Phase retrieval, error reduction algorithm, and Fienup variants: a view from convex optimization\rq\rq, \emph{Journal of the Optical Society of America}, \textbf{19} (2002), pp. 1334-1345.
		
		\bibitem{ShadowSequence}H.H. Bauschke and W.M. Moursi, ``On the Douglas-Rachford algorithm,''  to appear in \emph{Mathematical Programming}.
		
		\bibitem{hhm}J.M. Borwein. ``The Life of Modern Homo Habilis Mathematicus: Experimental Computation and Visual Theorems.'' \emph{Tools and Mathematics,} 23-90, in \emph{Mathematics Education Library}, vol. 347, Springer, 2016.
		
		\bibitem{BS} J.M. Borwein and B. Sims, ``The Douglas-Rachford algorithm in the absence of convexity'', in \emph{Fixed-Point Algorithms for Inverse Problems in Science and Engineering}, Springer Optimization and its Applications: 49 (2011), 93-109.
		
		\bibitem{BT}J.M. Borwein and Matthew K. Tam. ``Reflection methods for inverse problems with
		applications to protein conformation
		determination,'' Springer volume on the CIMPA school \emph{Generalized Nash Equilibrium Problems, Bilevel programming and MPEC}, New Delhi, India, Dec. 2012.
		
		\bibitem{Benoist} J. Benoist, ``The Douglas-Rachford Algorithm for the Case of the Sphere and the Line,'' \emph{Journal of Global Optimization}: 63 (2015), 363-380.
		
		\bibitem{Brezis}H. Brezis, \emph{Functional Analysis, Sobolev Spaces and Partial Differential Equations}, Springer (2011).
		
		\bibitem{D} F. Deutsch,  Rate of convergence of the method of alternating projections, in B. Brosowski and F. Deutsch (eds), \textsl{Parametric Optimization and Approximation}, Birkh{\"  a}user, Basel, 1983, pp96--107.
		
		\bibitem{Don} A.L. Dontchev, private communication.
		
		\bibitem{DR} J. Douglas and H.H. Rachford, ``On the numerical solution of the heat conduction problem
		in 2 and 3 space variables,'' \emph{Transactions of the AMS} \textbf{82} (1956), 421–439.
		
		\bibitem{Lak&DT} V. Lakshmikantham and D. Trigiante, \textsl{Theory of Difference Equations - Numerical Methods and Applications}, Marcel Dekker, 2002, ppx+300.
		
		\bibitem{LM}P.L. Lions and B. Mercier, ``Splitting algorithms for the sum of two nonlinear operators,'' \emph{SIAM
			Journal on Numerical Analysis,} \textbf{16} (1979), 964–979.
		
		\bibitem{N} T. Needham, \textsl{Visual complex analysis}, Clarendon Press, Oxford, 1997, ppxiii+612.
		
		\bibitem{APPENDIX} J.M. Borwein, S.B. Lindstrom, B. Sims, A. Schneider, and M.P. Skerritt. \emph{Appendix to dynamics of the Douglas-Rachford method for ellipses and p-spheres}. 2017. Available at:  \url{http://hdl.handle.net/1959.13/1330341}.
		
	\end{thebibliography}
\end{document}